\documentclass[final,reqno,onefignum,onetabnum]{siamltex}

\usepackage{amsmath}
\usepackage{amssymb}
\usepackage{amsfonts}
\usepackage{graphicx}
\usepackage{color}
\usepackage{mathptmx} % assumes new font selection scheme installed
\usepackage{times}
\usepackage{datetime}

\def\Sp{\mathrm{Sp}}       % matrix trace
\def\vect{\mathrm{vec}}
\def\vech{\mathrm{vech}}

\def\bTheta{\mathbfit{\Theta}}   % vectorization of matrices
\def\bitDelta{\mathbfit{\Delta}}
\def\bitPhi{\mathbfit{\Phi}}
\def\bitR{\mathbfit{R}}   % vectorization of matrices

\def\<{\leqslant}           % nice less than or equal to sign
\def\>{\geqslant}           % nice larger than or equal to sign
\def\d{\partial}
\def\wh{\widehat}
\def\wt{\widetilde}

\def\Re{\mathrm{Re}}   % real part
\def\Im{\mathrm{Im}}   % imaginary part

\def\cH{\mathcal{H}}   % Hardy space
\def\mA{\mathbb{A}}    % space of real antisymmetric matrices
\def\mR{\mathbb{R}}    % real line
\def\mC{\mathbb{C}}    % complex plane

\def\Tr{\mathrm{Tr}}       % matrix trace
\def\rT{\mathrm{T}}        % matrix transpose
\def\bS{\mathbf{S}}

\def\bitSigma{\mathbfit{\Sigma}}   % vectorization of matrices

\def\bE{\mathbf{E}}    % expectation

\def\[[[{[\![\![}
\def\]]]{]\!]\!]}

\def\bra{\langle}
\def\ket{\rangle}

\def\Bra{\left\langle}
\def\Ket{\right\rangle}

\def\re{\mathrm{e}}        % number e
\def\rd{\mathrm{d}}        % differential

\def\cL{\mathcal{L}}

\def\bA{\mathbf{A}}

\def\bD{\mathbf{D}}

\def\bJ{\mathbf{J}}

\def\x{\times}
\def\ox{\otimes}
\def\op{\oplus}

\def\fB{\mathfrak{B}}

\def\fE{\mathfrak{E}}
\def\fF{\mathfrak{F}}

\def\fH{\mathfrak{H}}

\def\fS{\mathfrak{S}}
\def\fU{\mathfrak{U}}
\def\fR{\mathfrak{R}}
\def\fP{\mathfrak{P}}

\def\fZ{\mathfrak{Z}}

\def\bT{\mathbf{T}}

\def\cF{\mathcal{F}}
\def\cW{\mathcal{W}}

\def\cX{\mathcal{X}}
\def\fX{\mathfrak{X}}

\def\cD{\mathcal{D}}

\def\cM{\mathcal{M}}

\def\cC{\mathcal{C}}
\def\cR{\mathcal{R}}

\def\sP{\mathsf{P}}

\def\sR{\mathsf{R}}

\def\cG{\mathcal{G}}
\def\cI{\mathcal{I}}
\def\cP{\mathcal{P}}
\def\cQ{\mathcal{Q}}

\def\cY{\mathcal{Y}}
\def\cA{\mathcal{A}}
\def\cB{\mathcal{B}}

\def\cov{\mathbf{cov}}

\def\cN{\mathcal{N}}

\def\bL{\mathbf{L}}

\def\mS{\mathbb{S}}

\def\eps{\epsilon}
\def\Ups{\Upsilon}
\def\ups{\upsilon}

\def\diag{\mathop{\mathrm{diag}}}    % diagonal matrix
\DeclareMathAlphabet{\mathbfit}{OML}{cmm}{b}{it}
\title{EFFECTS OF PARAMETRIC UNCERTAINTIES IN CASCADED OPEN QUANTUM HARMONIC OSCILLATORS AND ROBUST GENERATION OF GAUSSIAN INVARIANT STATES\thanks{This work is supported by the Air Force Office of Scientific Research (AFOSR) under agreement number FA2386-16-1-4065.
}}

\author{Igor G. Vladimirov$^{\dagger}$, \qquad Ian R. Petersen$^{\dagger}$, \qquad Matthew R. James\thanks{College of Engineering and Computer Science, Australian National University, Canberra, ACT 2601, Australia, E-mail: {\tt\small igor.g.vladimirov@gmail.com, i.r.petersen@gmail.com, matthew.james@anu.edu.au}}}

\begin{document}
\maketitle

\begin{abstract}
This paper is concerned with the generation of Gaussian invariant states in cascades of open quantum harmonic oscillators governed by linear quantum stochastic differential equations.
We carry out infinitesimal perturbation analysis of the covariance matrix for the invariant Gaussian state of such a system and the related purity functional subject to inaccuracies in the energy and coupling matrices of the subsystems. This leads to the problem of balancing the state-space realizations of the component oscillators through symplectic similarity transformations    in order to minimize the mean square sensitivity of the purity functional to small random perturbations of the parameters.
This results in a quadratic optimization problem with an effective solution in the case of cascaded one-mode oscillators, which is demonstrated by a numerical example. We also discuss a connection of the sensitivity index with classical statistical distances 
and outline infinitesimal  perturbation analysis for translation invariant cascades of identical oscillators. The findings of the paper are applicable to robust state generation in quantum stochastic networks.
\end{abstract}

%%%%%%%%%%%%%%%%%%%%%%%%%%%%%%%%%%%%%%%%%%%%%%%%%%%%%%%%%%%%%%%%%%%%%%%%%%%%%%%%%%%%%%%%%%%%%%%%%%%



%%%%%%%%%%%%%%%%%%%%%%%%%%%%%%%%%%%%%%%%%%%%%%%%%%%%%%%%%%%%%%%%%%%%%%%%%%%%%%%%%%%%%%%%%%%%%%%%%%%

\pagestyle{myheadings}
\thispagestyle{plain}
\markboth{IGOR G. VLADIMIROV,\quad IAN R. PETERSEN,\quad MATTHEW R. JAMES}{BALANCING CASCADED QUANTUM HARMONIC OSCILLATORS FOR GAUSSIAN STATE GENERATION}

\thispagestyle{empty}

%%%%%%%%%%%%%%%%%%%%%%%%%%%%%%%%%%%%%%%%%%%%%%%%%%%%%%%%%%%%%%%%%%%%%%%%%%%%%%%%%%%%%%%%%%%%%%%%%%%
\section{Introduction}\label{sec:intro}
%%%%%%%%%%%%%%%%%%%%%%%%%%%%%%%%%%%%%%%%%%%%%%%%%%%%%%%%%%%%%%%%%%%%%%%%%%%%%%%%%%%%%%%%%%%%%%%%%%%

The present paper is concerned with robustness of  state generation in a class of open quantum systems with respect to unavoidable uncertainties which accompany practical implementation of such systems. This issue is important in the context of the emerging quantum information  and quantum computation  technologies which exploit the potential resources of physical systems at atomic scales described by
quantum mechanics. Compared to classical systems with real-valued state variables  whose evolution obeys the laws of  Newtonian mechanics, quantum systems have more complicated operator-valued dynamic variables  which act on a Hilbert space and evolve according to unitary similarity transformations. This unitary evolution is specified by an operator-valued Hamiltonian which quantifies the self-energy of the system when it is isolated from the environment.  Interaction with a classical measuring device modifies the internal state of the quantum system in a random fashion,  which depends on the quantum observable being measured and makes noncommuting operators inaccessible to simultaneous measurement \cite{H_2001,M_1998,S_1994}. The inherently stochastic nature of quantum systems is reflected in their quantum probability theoretic description \cite{BVJ_2007,M_1995,P_1992} which replaces scalar-valued classical probability measures with density operators (quantum states) acting on the same Hilbert space as the dynamic variables. Regardless of whether measurements are involved, this probabilistic description is more complicated in the case of interaction between several open quantum systems, especially if one of them is organised  as an infinite reservoir of ``elementary'' systems representing a quantum field.

A unified theoretic framework for the modelling  and analysis of open quantum systems interacting with external bosonic fields (such as nonclassical light) is provided by the Hudson-Parthasarathy quantum stochastic calculus \cite{HP_1984,P_1992} (see also \cite{H_1991}). This approach
represents the Heisenberg picture evolution of system operators in the form of quantum stochastic differential equations (QSDEs) driven by noncommutative counterparts of the classical Wiener process \cite{KS_1991}. Reflecting quantized energy exchange, the quantum Wiener processes involve  annihilation and creation operators acting on symmetric Fock spaces \cite{PS_1972}. The energetics of the quantum system itself and its interaction with the external fields, which specifies the structure of the QSDEs, is captured by the system Hamiltonian, system-field coupling operators and the scattering matrix (which pertains to the photon exchange between the fields in multichannel settings). In combination with  the theory of quantum feedback networks \cite{GJ_2009,JG_2010}, this approach  allows for the modelling of  a wide class of interconnected open quantum systems which interact with one another and the environment. In fact, coherent (measurement-free) quantum control and filtering for quantum systems by direct or field-mediated interconnection \cite{JNP_2008,MJ_2012,MP_2009,NJP_2009,P_2014,PH_2015,VP_2013a,VP_2013b,VP_2016,ZJ_2011b}  constitute a promising modern paradigm which can potentially outperform the traditional observation-actuation approach of classical control theory. The principal advantage of this paradigm is that it avoids the loss of quantum information which accompanies the conversion of operator-valued quantum variables into classical real-valued signals in the process of measurement \cite{GZ_2004,WM_2010}.

The interaction of an open quantum system with external fields can be arranged in a dissipative fashion so that the resulting quantum stochastic system has an invariant state with desired properties. This provides an alternative approach \cite{Y_2012,Y_2009}  to the system state preparation which  can otherwise be carried out by steering the system to the required state through varying the parameters of the Hamiltonian  in an open-loop fashion or by using feedback \cite{DP_2010}, similarly to the classical terminal state control problem \cite{PBGM_1962,SW_1997}. Since quantum state generation via dissipation does not involve measurement-based feedback, it is more aligned with the coherent quantum control paradigm mentioned above. State generation is important for quantum computation  protocols \cite{NC_2000}, which often need the quantum system of interest  to be initialized  in a certain class of states.

Engineering a quantum state with required properties is relevant both for finite-level systems (such as qubit registers) and quantum systems with continuous variables which find applications in quantum optical platforms of quantum computing and quantum information processing \cite{NC_2000,WPGGCRSL_2012}. An important class of such systems is constituted by open quantum harmonic oscillators (OQHOs) whose dynamic variables satisfy the canonical commutation relations (CCRs), similar to those of the quantum mechanical positions and momenta \cite{M_1998,S_1994}, and    are governed by linear QSDEs \cite{P_2017}. The linearity of these QSDEs comes from the CCRs and the fact that the Hamiltonian and coupling operators of the OQHO are quadratic and linear functions of the system variables. The dynamics of such oscillators interacting with bosonic fields in the vacuum state  are, in many respects,  similar to those of classical Gaussian Markov diffusion processes (including the Ornstein-Uhlenbeck process \cite{KS_1991}) and are particularly tractable at the level of the first two moments of the system variables. For example, the linear dynamics preserve the Gaussian nature \cite{KRP_2010,PS_2015}  of the system state  in time, provided the OQHO is initialized in such a state. Moreover, irrespective of whether the initial state is Gaussian, linear QSDE with a Hurwitz matrix ensures the weak convergence \cite{B_1968,CH_1971}  of the system state to a unique invariant Gaussian state. It is this property of OQHOs that enables them to be employed for generating  Gaussian quantum states by allowing the system to evolve over a sufficiently long period of time \cite{MWPY_2014,Y_2012}. The dissipation, which is built in this state generation procedure, secures stability of the invariant state being achieved in the long run \cite{PAMGUJ_2014}.

However, the practical realization of OQHOs, which involves quantum optical components such as cavities, beam splitters and phase shifters, is accompanied by modelling inaccuracies and implementation errors. Even if the energetics of such a system remains linear-quadratic, the energy and coupling matrices can deviate from their nominal values. The resulting Gaussian invariant states can have covariance matrices  which differ from the theoretical predictions. These deviations can lead to a deterioration or loss of important properties of quantum states such as purity \cite{SSM_1988}, which can be critical due to the role of the pure states as extreme points of the convex set of density operators. This gives rise to the issue of robustness of the  Gaussian state generation in linear quantum stochastic systems.

This circle of problems is the main theme of the present paper, which is concerned with Gaussian invariant states for cascaded OQHOs driven by vacuum fields. Such chain-like networks of linear quantum stochastic systems are relatively easy to implement and find applications, for example, in the generation of pure states \cite{MWPY_2014,Y_2012}. Since the transfer functions of the component oscillators satisfy the physical realizability conditions \cite{JNP_2008,SP_2012}, their $\cH_{\infty}$-norms are not less than one. This can potentially enhance the propagation of perturbations in the energy and coupling matrices of the subsystems over long cascades. To this  end, we carry out an infinitesimal perturbation analysis of the covariance matrix for the invariant Gaussian state and the related  purity functional subject to inaccuracies in the energy and coupling matrices of the component oscillators. Since these quantities depend in a ``spatially  causal'' fashion  on the energy and coupling matrices of the oscillators along the cascade, their Frechet derivatives with respect to the matrix-valued parameters  are amenable to recursive computation. This computation employs a connection with conditional covariance matrices of auxiliary classical Gaussian random vectors (which are updated similarly to the covariance equations of the discrete-time Kalman filter) and variational techniques for solutions of algebraic Lyapunov and Sylvester equations \cite{SIG_1998,VP_2013a}.

By modelling the parametric uncertainties in the component oscillators as small zero-mean classical random elements which are uncorrelated for different subsystems, the corresponding perturbations in the purity functional are quantified by a mean square sensitivity index which  depends on a particular state-space realization of the system.
This leads to the problem of balancing the subsystems (for example, one-mode oscillators) by using CCR-preserving symplectic similarity transformations  of their realizations in order to minimize the sensitivity of the purity functional to such uncertainties. By using group theoretic properties of Lyapunov equations, we obtain  a closed-form dependence of the mean square sensitivity index on the transformation matrices, which allows its minimization to be decomposed into independent optimization problems for the subsystems. Each of these lower dimensional   problems is organised as the minimization of a quartic polynomial of the corresponding symplectic matrix, which, in the one-mode case,  reduces to a quadratic optimization problem with an effective solution. We demonstrate this result by a numerical example of balancing a cascade of  one-mode oscillators.

The specific choice of the optimality criterion for balancing the cascaded oscillators is not unique and can be based on different cost functionals. To this end, we also discuss a connection of the above criterion with the classical Fisher information distance \cite{S_1984} applied to Gaussian quantum  states. Note that balanced realizations, developed previously for classical linear systems and their quantum counterparts \cite{N_2013}, were mainly concerned with equating the controllability and observability Gramians,  which comes from the Kalman duality principle. However, the quantum setting of the present paper employs different criteria pertaining to the infinitesimal perturbation analysis of the state generation. In the case of translation invariant cascades of identical oscillators,
the propagation of perturbations in subsystems over such cascades is particularly amenable to analysis using the technique of spatial $z$-transforms \cite{VP_2014} (see also \cite{SVP_2015a,SVP_2015b}) which is also considered in the present study.
The results of this paper may also find applications to perturbation analysis and robust Gaussian  state generation in linear quantum stochastic networks with more complicated architectures.

The paper is organised as follows.
Section~\ref{sec:LQSS} provides a background material on quantum stochastic  systems.
Section~\ref{sec:sys} describes the class of cascaded linear quantum stochastic systems under consideration.
Section~\ref{sec:inv} specifies the invariant Gaussian quantum state for the composite system and the purity functional.
Section~\ref{sec:reccov}  provides recurrence equations for the steady-state covariances using the cascade structure of the system.
Section~\ref{sec:diff} carries out an infinitesimal perturbation analysis of the purity functional with respect to the energy and coupling matrices.
Section~\ref{sec:recder} describes a recursive computation of the appropriate Frechet derivatives using the cascade structure of the system.
Section~\ref{sec:opt} quantifies mean square sensitivity of the purity functional with respect to random implementation errors as an optimality criterion and splits its minimization into independent problems.
Section~\ref{sec:example} provides an illustrative numerical example of balancing a cascade of one-mode oscillators.
Section~\ref{sec:conc} makes concluding remarks.
The appendices provide an additional material for completeness.
Appendix~\ref{sec:relent} discusses a connection between the sensitivity index for the purity functional and the Fisher information distance applied to Gaussian states.
Appendix~\ref{sec:trans} carries out an infinitesimal perturbation analysis of steady-state covariances for translation invariant cascades of identical oscillators.

%%%%%%%%%%%%%%%%%%%%%%%%%%%%%%%%%%%%%%%%%%%%%%%%%%%%%%%%%%%%%%%%%%%%%%%%%%%%%%%%%%%%%%%%%%%%%%%%%%%
\section{Linear quantum stochastic systems}
\label{sec:LQSS}
%%%%%%%%%%%%%%%%%%%%%%%%%%%%%%%%%%%%%%%%%%%%%%%%%%%%%%%%%%%%%%%%%%%%%%%%%%%%%%%%%%%%%%%%%%%%%%%%%%%

For the purposes of the subsequent sections, we will provide a background on a class of quantum stochastic systems, including open quantum harmonic oscillators which play the role of building blocks in linear quantum systems theory \cite{P_2017}.
The evolution of such a system in time $t\>0$ is described in terms of an even number $n$ of dynamic variables $x_1(t),\ldots, x_n(t)$ assembled into a vector
\begin{equation}
\label{Xx}
    X
    :=
    \begin{bmatrix}
    x_1\\
    \vdots\\
    x_n
    \end{bmatrix}
\end{equation}
(vectors are organised as columns unless indicated otherwise, and the time arguments are often omitted for brevity). These system variables are time-varying self-adjoint operators\footnote{there also are alternative formulations which employ non-Hermitian operators such as the annihilation and creation operators.} on a complex separable Hilbert space $\fH$ (whose structure is clarified below), with $x_1(0), \ldots, x_n(0)$ acting  on a  subspace $\fH_0$ which is referred to as the initial system space.  Their dynamics are governed by a Markovian Hudson-Parthasarathy quantum stochastic differential equation (QSDE) \cite{HP_1984,P_1992}
\begin{equation}
\label{dX}
    \rd X =
    \cG(X)\rd t  - i[X,L^{\rT}]\rd W
\end{equation}
whose structure is described below.
Although it formally resembles classical SDEs \cite{KS_1991}, the QSDE (\ref{dX}) is driven by a vector
\begin{equation}
\label{W}
    W
    :=
    \begin{bmatrix}
        w_1\\
        \vdots
        \\
        w_m
    \end{bmatrix}
\end{equation}
of an even number $m$ of self-adjoint operator-valued quantum Wiener processes $w_1, \ldots, w_m$ acting on a symmetric Fock space $\fF$. These  processes model the external bosonic fields \cite{H_1991,HP_1984,P_1992}, interacting with the system (see Fig.~\ref{fig:sys}).
%==============================================================================
\begin{figure}[htbp]
\centering
\unitlength=0.9mm
\linethickness{0.2pt}
\begin{picture}(100.00,45.00)
    \put(40,20){\framebox(20,20)[cc]{system}}
    \put(20,38){\vector(1,0){20}}
    \put(20,33){\vector(1,0){20}}
    \put(20,22){\vector(1,0){20}}
    \put(30,29){\makebox(0,0)[cc]{{$\vdots$}}}
    \put(15,29){\makebox(0,0)[cc]{{$\vdots$}}}

    \put(15,38){\makebox(0,0)[cc]{$w_1$}}
    \put(15,33){\makebox(0,0)[cc]{$w_2$}}
    \put(15,22){\makebox(0,0)[cc]{$w_m$}}

    \put(60,38){\vector(1,0){20}}
    \put(60,33){\vector(1,0){20}}
    \put(60,22){\vector(1,0){20}}
    \put(70,29){\makebox(0,0)[cc]{{$\vdots$}}}
    \put(85,29){\makebox(0,0)[cc]{{$\vdots$}}}

    \put(85,38){\makebox(0,0)[cc]{$y_1$}}
    \put(85,33){\makebox(0,0)[cc]{$y_2$}}
    \put(85,22){\makebox(0,0)[cc]{$y_m$}}

\end{picture}\vskip-15mm
\caption{A schematic depiction of an open quantum stochastic system which interacts with the input quantum Wiener processes $w_1, \ldots, w_m$ and produces the output fields  $y_1, \ldots, y_m$.}
\label{fig:sys}
\end{figure}
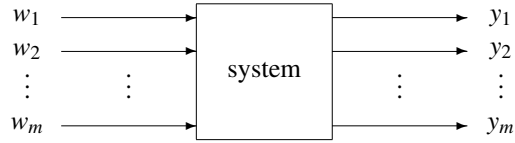
%==============================================================================
Accordingly, the above space $\fH$ is organised as the tensor-product Hilbert space $\fH:= \fH_0 \ox \fF$ which provides a common domain for the system and field operators. The quantum Wiener processes in (\ref{W}) satisfy the quantum Ito relations
\begin{equation}
\label{dWdW}
    \rd W\rd W^{\rT}
    :=
    (\rd w_j\rd w_k)_{1\<j,k\< m}
    =
    \Omega \rd t,
\end{equation}
where
\begin{align}
\label{Omega}
        \Omega
    &
    :=
    I_m + iJ,\\
\label{J}
        J
        & :=
       \begin{bmatrix}
           0 & I_{m/2}\\
           -I_{m/2} & 0
       \end{bmatrix}
       =
        \bJ \ox I_{m/2},\\
\label{bJ}
    \bJ
    & :=
    \begin{bmatrix}
        0 & 1\\
        -1 & 0
    \end{bmatrix}.
\end{align}
Here, $i:= \sqrt{-1}$ is the imaginary unit, $\ox$ denotes the Kronecker product of matrices, $I_r$ is the identity matrix of order $r$ (which will often be omitted when it is clear from the context), and $\bJ$
spans the space of antisymmetric matrices of order $2$.
Also, the transpose $(\cdot)^{\rT}$ acts on matrices of operators as if their entries were scalars.
In contrast to the identity diffusion matrix of the standard Wiener process, $\Omega$ in (\ref{Omega}) is  a complex positive semi-definite Hermitian matrix with an orthogonal antisymmetric imaginary part   $\Im \Omega=J$ (so that $J^2 =-I_m$). In view of (\ref{dWdW}),  the quantum Wiener processes  $w_1, \ldots, w_m$ do not commute with each other. Furthermore,  their two-point commutator matrix is given by
\begin{align}
\nonumber
    [W(s), W(t)^{\rT}]
    & :=
    ([w_j(s), w_k(t)])_{1\<j,k\<m}\\
\label{WWst}
    & = 2i\min(s,t)J ,
    \qquad
    s,t\>0,
\end{align}
where $[\alpha, \beta]:= \alpha \beta - \beta \alpha$ is the commutator of linear operators $\alpha$ and $\beta$. In accordance with $W$ representing the $m$-channel input field, the vector
\begin{equation}
\label{Lell}
    L
    :=
    \begin{bmatrix}
        \ell_1\\
        \vdots\\
        \ell_m
    \end{bmatrix}
\end{equation}
in (\ref{dX}) consists of system-field coupling operators $\ell_1, \ldots, \ell_m$ which are also self-adjoint operators on the space $\fH$. Since the entries of the commutator matrix
$$
    [X, L^{\rT}]
    :=    ([x_j,\ell_k])_{1\< j\< n,1\<k\< m}
$$
are skew-Hermitian operators, the dispersion $(n\x m)$-matrix $-i [X, L^{\rT}]$ in (\ref{dX}) consists of self-adjoint operators on $\fH$.
The $n$-dimensional
drift vector
$    \cG(X)
$
of the QSDE (\ref{dX}) is obtained by the entrywise application of
the Gorini-Kossakowski-Sudar\-shan-Lindblad (GKSL) generator    \cite{GKS_1976,L_1976}, which acts on a system operator $\xi$ (a function of the system variables)  as
\begin{equation}
\label{GKSL}
   \cG(\xi)
   := i[H,\xi]
     +    \frac{1}{2}
    \sum_{j,k=1}^m
    \omega_{jk}
    \big( [\ell_j,\xi]\ell_k + \ell_j[\xi,\ell_k]\big),
\end{equation}
where $\omega_{jk}$ are the entries of the quantum Ito matrix $\Omega$ from (\ref{Omega}).
Here, $H$ denotes
the system Hamiltonian which is also a self-adjoint operator on $\fH$.  The energy operators $H$ and $\ell_1, \ldots, \ell_m$ are functions (for example, polynomials) of the system variables $x_1, \ldots, x_n$ and inherit dependence on time from them.
The GKSL superoperator $\cG$ in (\ref{GKSL})  is a quantum analogue of the infinitesimal generators of classical Markov diffusion processes \cite{KS_1991,S_2008} and specifies the drift of the QSDE
\begin{equation}
\label{dxi}
    \rd \xi
    =
    \cG(\xi)\rd t - i[\xi,L^{\rT} ]\rd W.
\end{equation}
The specific structure of the drift and diffusion terms in (\ref{dxi}) (and its particular case (\ref{dX})) comes from  the evolution
\begin{equation}
\label{xiuni}
    \xi(t)
    =
    U(t)^{\dagger} (\xi(0)\ox \cI_{\fF}) U(t)
\end{equation}
of the system operator $\xi$, with $\xi(0)$ acting on the initial system space $\fH_0$.  Here, $(\cdot)^{\dagger}$ denotes the operator adjoint, and $U(t)$ is a unitary operator which acts on the system-field space $\fH$ and is governed by the QSDE
\begin{equation}
\label{dU}
    \rd U(t)
    =
    -U(t) \Big(i(H(t)\rd t + L(t)^{\rT} \rd W(t)) + \frac{1}{2}L(t)^{\rT}\Omega L(t)\rd t\Big),
\end{equation}
where $U(0)=\cI_{\fH}$  is the identity operator on $\fH$. The operator $U(t)$, which  is associated with the system-field interaction over the time interval from $0$ to $t$, is adapted in the sense that  it acts effectively on the subspace $\fH_0\ox \fF_t$, where $\{\fF_t:\, t\>0\}$ is the Fock space filtration.

The QSDE (\ref{dxi}) can be  obtained from (\ref{xiuni}) and (\ref{dU}) by using the quantum Ito formula \cite{HP_1984,P_1992} in combination with (\ref{Omega}), unitarity of $U$ and commutativity between the forward Ito increments $\rd W(t)$ and adapted processes (including $U$) taken at time $s\< t$.
The corresponding quantum stochastic flow at time $t$ involves a unitary similarity transformation which acts  on operators $\zeta$ on the space $\fH$  as $ \zeta\mapsto U(t)^{\dagger} \zeta U(t)$ and applies entrywise to vectors of such operators. In accordance with (\ref{xiuni}),  the system variables evolve as
\begin{equation}
\label{uni}
    X(t)
    =
    U(t)^{\dagger} (X(0)\ox \cI_{\fF}) U(t).
\end{equation}
The system-field interaction, which drives the unitary operator $U(t)$ in (\ref{dU}), produces the output fields $y_1, \ldots, y_m$ (see Fig.~\ref{fig:sys}) which are also time-varying self-adjoint operators on the space $\fH$. The output field vector evolves as
\begin{equation}
\label{Y}
    Y(t)
    :=
    \begin{bmatrix}
        y_1(t)\\
        \vdots\\
        y_m(t)
    \end{bmatrix}
    =
 U(t)^{\dagger}(\cI_{\fH_0}\ox W(t))U(t)
\end{equation}
and satisfies the QSDE
\begin{equation}
\label{dY}
  \rd Y = 2JL\rd t + \rd W
\end{equation}
which is obtained similarly to (\ref{dxi}).
The QSDEs (\ref{dX}) and (\ref{dY}) describe a particular yet important scenario of quantum stochastic dynamics with the identity scattering matrix, which corresponds to the absence of photon exchange between the fields  \cite{HP_1984,P_1992}. Endowed with additional features (for example, more general scattering matrices and related gauge processes which affect the dynamics of the unitary operator $U$ in (\ref{dU})),
such QSDEs are employed in the formalism for modelling feedback networks of quantum systems which  interact with each other and the external fields \cite{GJ_2009,JG_2010}. 
Irrespective of a particular form of (\ref{dU}),  the unitary similarity transformation in (\ref{uni}) and (\ref{Y}) preserves the commutativity between the system and output field variables in the sense that
\begin{equation}
\label{XY}
        [X(t),Y(s)^{\rT}]
     =
    0,
    \qquad
    t\> s\> 0
\end{equation}
(that is, future system variables commute with the  past output variables).
At the same time, the output variables $y_1, \ldots, y_m$ do not commute with each other
and, in view of (\ref{WWst}) and (\ref{dY}),
inherit from the input fields the two-point commutator matrix:
$$
    [Y(s), Y(t)^{\rT}] = 2i\min(s,t)J,
    \qquad
    s,t\>0.
$$
This noncommutativity makes the output fields inaccessible to simultaneous measurement \cite{H_2001,M_1998,S_1994}. However, they can be fed in a measurement-free fashion as an input to other open quantum systems. 
The resulting  coherent  field-mediated    interconnection gives rise to fully quantum communication channels in the form of cascades of (possibly distant) quantum systems. In Section~\ref{sec:sys}, we will consider such connections of linear quantum stochastic systems which are used, in particular, for the generation of certain classes of quantum states (see, for example, \cite{MWPY_2014}).

The open quantum system, described above, is referred to as a linear quantum stochastic system, or an open quantum harmonic oscillator (OQHO) with $\frac{n}{2}$ modes, if its Hamiltonian and the coupling operators are quadratic and linear functions of the system variables, respectively, and the latter satisfy the following form of canonical commutation relations (CCRs). More precisely, the system variables $x_1, \ldots, x_n$ of the OQHO
%a linear quantum stochastic system, or an open quantum harmonic oscillator (OQHO),
satisfy the Weyl  CCRs
\begin{equation}
\label{CCR}
    \cW_u \cW_v = \re^{i v^{\rT}\Theta u} \cW_{u+v}
\end{equation}
for all $u,v\in \mR^n$. Here,
$\Theta:= (\theta_{jk})_{1\< j,k\< n}$ is a real antisymmetric matrix of order $n$ (we denote the subspace of such matrices by $\mA_n$) which is assumed to be nonsingular. The CCRs (\ref{CCR}) are formulated in terms of the unitary Weyl operators
$$
  \cW_u := \re^{iu^{\rT} X} = \cW_{-u}^{\dagger},
$$
where $u^{\rT} X = \sum_{k=1}^n u_k x_k$ is a linear combination of the system variables assembled into the vector (\ref{Xx}),
with the coefficients comprising the vector $u:=(u_k)_{1\< k\< n} \in \mR^n$, so that $u^{\rT}X$ is also a self-adjoint operator.
 The relations (\ref{CCR}) imply that
$$
    [\cW_u, \cW_v]
    =
    -2i\sin(u^{\rT}\Theta v)\cW_{u+v},
$$
which leads to the
Heisenberg  infinitesimal form of the Weyl CCRs, specified on a dense subset of the space $\fH$ by the commutator matrix
\begin{equation}
\label{Thetalin}
    [X, X^{\rT}]
    =
     2i \Theta.
\end{equation}
For example, in the case when the system variables are $q_1, \ldots, q_{n/2}, p_1, \ldots, p_{n/2}$  and consist of conjugate quantum mechanical position $q_k$ and momentum $p_k = -i\d_{q_k}$ operators (with an appropriately normalized Planck constant) \cite{M_1998,S_1994} which satisfy $[q_j,p_k] = i\delta_{jk}$ for all $j,k=1,\ldots, \frac{n}{2}$, with $\delta_{jk}$ the Kronecker delta,  the CCR matrix $\Theta$ takes the form
\begin{equation}
\label{posmom}
    \Theta
    :=
    \frac{1}{2}
    \bJ \ox I_{n/2}
    =
    \frac{1}{2}
    \begin{bmatrix}
        0 & I_{n/2}\\
        -I_{n/2} & 0
    \end{bmatrix},
\end{equation}
where the matrix $\bJ$ is given by (\ref{bJ}). Not restricted to this particular case, 
the system Hamiltonian $H$ of the OQHO is a quadratic function and the system-field coupling operators $\ell_1, \ldots, \ell_m$ in (\ref{Lell}) are linear functions of the system variables:
\begin{align}
\label{H}
    H &
    =
    \frac{1}{2}
    \sum_{j,k=1}^{n}r_{jk}
    x_jx_k = \frac{1}{2} X^{\rT} R X,\\
\label{LM}
    L&
    =
    MX.
\end{align}
Here, $R:= (r_{jk})_{1\< j,k\< n}$ is a real symmetric matrix of order $n$ (the subspace of such matrices is denoted by $\mS_n$), and $M \in \mR^{m\x n}$. The matrices $R$ and $M$  will be referred to as the energy and coupling matrices, respectively.   Due to the CCRs (\ref{Thetalin}) and the linear-quadratic energetics (\ref{H}) and (\ref{LM}), the QSDEs (\ref{dX}) and (\ref{dY}) become linear with respect to the system variables:
\begin{align}
\label{dXlin}
  \rd X
    & =
    A X \rd t+ B \rd W,\\
\label{dYlin}
  \rd Y
    & =
    C X \rd t+ \rd W,
\end{align}
where the matrices $A\in \mR^{n\x n}$, $B\in \mR^{n\x m}$, $C \in \mR^{m\x n}$ are computed as
\begin{align}
\label{Alin}
    A
    &
    :=
    2\Theta (R + M^{\rT} J M),\\
\label{Blin}
    B
    &
    := 2\Theta M^{\rT},\\
\label{Clin}
    C
    &
    := 2J M.
\end{align}
In combination with the symmetry of $R$ and antisymmetry of $J$ and $\Theta$ in  (\ref{J})  and (\ref{Thetalin}),
the specific dependence of the  matrices $A$, $B$, $C$ on the energy and coupling matrices $R$ and $M$ in (\ref{Alin})--(\ref{Clin})  is equivalent to the physical realizability (PR) conditions \cite{JNP_2008,SP_2012}:
\begin{align}
\label{APR}
    A \Theta + \Theta A^{\rT} + BJB^{\rT} & = 0,\\
\label{BCPR}
    \Theta C^{\rT} + BJ & = 0.
\end{align}
The equality (\ref{APR}) is related to the preservation of the CCRs (\ref{Thetalin}) in time, while (\ref{BCPR})
corresponds to (\ref{XY}).

The linearity of the QSDEs (\ref{dXlin}) and (\ref{dYlin}) makes the OQHO a basic model in linear quantum control \cite{JNP_2008,NJP_2009,P_2017}. If the initial system variables have finite second moments (that is, $\bE(X(0)^{\rT}X(0))< +\infty$), the linear dynamics (\ref{dX}) preserve the mean square integrability in time and lead to finite limit values of the first and second moments
\begin{equation}
\label{EXX}
    \lim_{t\to +\infty} \bE X(t) = 0,
    \qquad
    \lim_{t\to +\infty} \bE(X(t)X(t)^{\rT}) = P+i\Theta,
\end{equation}
provided the matrix $A$ in (\ref{Alin}) is Hurwitz.
Here, $\bE \zeta := \Tr(\rho\zeta)$ is the quantum expectation over the system-field density operator $\varrho := \varpi \ox \ups$, which is the tensor product of the initial system state $\varpi$ and the vacuum state $\ups$ of the input bosonic  fields on the Fock space $\fF$ \cite{P_1992}. 
The matrix $P$ in (\ref{EXX}) coincides with the infinite-horizon controllability Gramian
\begin{equation}
\label{P}
    P
    =
    \int_0^{+\infty}
    \re^{tA}
    BB^{\rT}
    \re^{tA^{\rT}}
    \rd t
\end{equation}
of the matrix pair $(A,B)$,  which is a unique solution  of an algebraic Lyapunov equation (ALE)
due to the matrix $A$ being Hurwitz:
\begin{equation}
\label{PALE}
    AP+PA^{\rT} + BB^{\rT} = 0.
\end{equation}
Despite this connection with classical linear systems theory, the quantum covariance matrix in (\ref{EXX}) satisfies $P + i\Theta\succcurlyeq 0$, which reflects the generalized Heisenberg uncertainty principle \cite{H_2001} and is a stronger property than the positive semi-definiteness of $P$ alone in the classical case.
This property can also be obtained directly by combining the ALE (\ref{PALE}) with the PR condition (\ref{APR}), which leads to the ALE
$
    A (P+i\Theta) + (P+i\Theta) A^{\rT} + B\Omega B^{\rT} = 0
$
whose solution satisfies
\begin{equation}
\label{Ppos}
    P+i\Theta
    =
    \int_0^{+\infty}
    \re^{tA}B\Omega
    B^{\rT}\re^{tA^{\rT}}\rd t
    \succcurlyeq 0
\end{equation}
in view of the positive semi-definiteness of the quantum Ito matrix $\Omega$ in (\ref{Omega}). Application of a matrix-valued  version of the Plancherel theorem leads to the frequency-domain representation of (\ref{Ppos}):
\begin{equation}
\label{Ppos1}
    P+i\Theta
    =
    \frac{1}{2\pi}
    \int_{-\infty}^{+\infty}
    F(i\lambda)\Omega F(i\lambda)^*
    \rd \lambda,
\end{equation}
where $(\cdot)^*:= (\overline{(\cdot)})^{\rT}$ denotes the complex conjugate transpose. Here, $\mC \ni s\mapsto F(s)\in \mC^{n\x m}$ is a rational transfer function from $W$ to $X$ given by
\begin{equation}
\label{Flin}
    F(s)
     := (sI_n - A)^{-1}B.
\end{equation}
A related transfer function $G$ from $W$ to $Y$, which, together with $F$, is associated   with the QSDEs (\ref{dXlin}) and (\ref{dYlin}) of the OQHO, is given by
\begin{equation}
\label{Glin}
    G(s)
    := C F(s) + I_m.
\end{equation}
In view of (\ref{Flin}), the poles of $F$ and $G$ belong to the spectrum of the matrix $A$ and hence, are entirely in the left half-plane $\Re s<0$ due to $A$ being Hurwitz.   Similarly to classical linear stochastic systems, these transfer functions  relate the Laplace transforms  of the quantum processes $X$, $Y$, $W$ (considered in the right-half plane $\Re s>0$)   as
\begin{align}
\nonumber
    \wt{X}(s)
    & :=
    \int_{0}^{+\infty}
    \re^{-st}
    X(t)\rd t \\
\label{Xtilde}
    & =
    F(s)\wt{W}(s)
    +
    (sI_n - A)^{-1}X(0),\\
\nonumber
    \wt{Y}(s)
    & :=
    \int_{0}^{+\infty}
    \re^{-st}
    \rd Y(t) \\
\nonumber
    & =
    C \wt{X}(s) + \wt{W}(s)\\
\label{Ytilde}
    & =
    G(s)\wt{W}(s)
    +
    C(sI_n - A)^{-1}X(0),\\
\label{Wtilde}
    \wt{W}(s)
    & :=
    \int_0^{+\infty}
    \re^{-st}
    \rd W(t).
\end{align}
Due to the PR conditions (\ref{APR}) and (\ref{BCPR}), the transfer function $G$ in (\ref{Glin}) is $(J,J)$-unitary \cite{SP_2012} in the sense that
\begin{equation}
\label{GJJ}
    G(i\lambda) JG(i\lambda)^* = J
\end{equation}
for all $\lambda \in \mR$. Since $G$ has an identity feedthrough matrix $\lim_{s\to \infty}G(s) =
I_m$, its $\cH_{\infty}$-norm (in the appropriate Hardy space) satisfies $\|G\|_{\infty}\> 1$.

If the matrix $A$ is Hurwitz, then the reduced system state converges to an invariant Gaussian quantum state \cite{KRP_2010,PS_2015}  in the sense of appropriately modified weak convergence of probability measures \cite{B_1968}. This property holds
irrespective of whether the initial system state is Gaussian or has finite second-order moments of the system variables (which is essential for (\ref{EXX})) and is equivalent to the point-wise convergence of the quasi-characteristic function (QCF) \cite{CH_1971}:
\begin{equation}
\label{limQCFlin}
    \lim_{t\to +\infty}
    \bE \re^{iu^{\rT} X(t)}
    =
    \re^{-\frac{1}{2} \|u\|_P^2},
    \qquad
    u\in \mR^n.
\end{equation}
Here, the matrix $P$ is given by (\ref{P}), and $\|u\|_P := \sqrt{u^{\rT} P u} = |\sqrt{P} u |$ denotes the corresponding weighted Euclidean norm. The QCF on the right-hand side of (\ref{limQCFlin}) is identical to the characteristic function of the classical Gaussian distribution in $\mR^n$ with zero mean vector and covariance matrix $P$.  However, as mentioned above, the quantum nature of the setting manifests itself in the stronger property  $P + i\Theta\succcurlyeq 0$ of the matrix $P$ in (\ref{Ppos}).

The convergence (\ref{limQCFlin})
can be used in order to generate a zero-mean Gaussian state with a given quantum covariance matrix  $P+i\Theta$ as an invariant state of the OQHO described by (\ref{dXlin}) and (\ref{dYlin}); see, for example, \cite{Y_2012}. To this end, the energy and coupling matrices $R$ and $M$ have to be chosen so that the matrix $A$ in (\ref{Alin}) is  Hurwitz and (\ref{P}) is satisfied for a given admissible matrix $P$.  Any particular choice of $R$ and $M$ does not affect the CCR matrix $\Theta$ in (\ref{Thetalin}) or the Gaussian nature of the invariant state and only influences the matrix $P$. Indeed,  the energetics of the OQHO remains linear-quadratic and the QSDE, driven by the vacuum input fields, remains linear.\footnote{A wider class of perturbations of the Hamiltonian and coupling operators, leading to nonlinear QSDEs and non-Gaussian invariant states, is considered, for example, in \cite{SVP_2014,V_2015c,VPJ_2017}.} However, the matrix $P$ determines other important properties of Gaussian states such as purity \cite{SSM_1988}. In application to the invariant Gaussian system state,
with a reduced density operator $r$ (which is also a positive semi-definite self-adjoint operator of unit trace),
the purity is quantified by
\begin{equation}
\label{puritylin}
    \Tr (r^2)
    =
    \sqrt{\frac{\det \Theta}{\det P}}
    \< 1.
\end{equation}
The Gaussian state is pure if and only if the inequality in (\ref{puritylin}) is an equality, in which case, $\det P$ achieves its minimum   value $\det \Theta$ in view of the matrix inequality in (\ref{Ppos}).

In the case when the dimension $n$ of the OQHO is high, of practical interest is scalability in the  generation of pure (or nearly pure) states. This can be achieved, for example, by using the cascade architecture \cite{GZ_2004,MWPY_2014} which allows such an OQHO to be assembled from relatively simple components (such as one-mode oscillators). At the same time, the resulting large number of subsystems  makes it important to secure robustness of the purity functional with respect to the cumulative effect of perturbations in the individual components.

%%%%%%%%%%%%%%%%%%%%%%%%%%%%%%%%%%%%%%%%%%%%%%%%%%%%%%%%%%%%%%%%%%%%%%%%%%%%%%%%%%%%%%%%%%%%%%%%%%%
\section{Cascaded open quantum harmonic oscillators}
\label{sec:sys}
%%%%%%%%%%%%%%%%%%%%%%%%%%%%%%%%%%%%%%%%%%%%%%%%%%%%%%%%%%%%%%%%%%%%%%%%%%%%%%%%%%%%%%%%%%%%%%%%%%%

Consider the field-mediated cascade connection of $N$ OQHOs in Fig.~\ref{fig:casc}, which are driven by
%==============================================================================
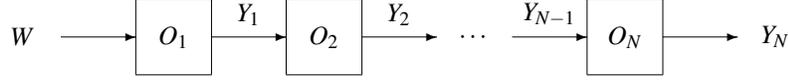
\begin{figure}[htbp]
\centering
\unitlength=1mm
\linethickness{0.2pt}
\begin{picture}(110.00,20.00)
    \put(5,10){\makebox(0,0)[cc]{$W$}}
    \put(10,10){\vector(1,0){10}}
    \put(20,5){\framebox(10,10)[cc]{$O_1$}}
    \put(30,10){\vector(1,0){10}}
    \put(35,13){\makebox(0,0)[cc]{$Y_1$}}

    \put(40,5){\framebox(10,10)[cc]{$O_2$}}
    \put(50,10){\vector(1,0){10}}
    \put(55,13){\makebox(0,0)[cc]{$Y_2$}}

    \put(65,10){\makebox(0,0)[cc]{$\cdots$}}
    \put(70,10){\vector(1,0){10}}
    \put(75,13){\makebox(0,0)[cc]{$Y_{N-1}$}}
    \put(80,5){\framebox(10,10)[cc]{$O_N$}}
    \put(90,10){\vector(1,0){10}}
    \put(105,10){\makebox(0,0)[cc]{$Y_N$}}
\end{picture}
\caption{A block-diagram of $N$ cascaded OQHOs $O_1, \ldots, O_N$ which are driven by a quantum Wiener process  $W$ and produce the output fields $Y_1, \ldots, Y_N$, so that $O_k$ has output $Y_k$ and input $Y_{k-1}$, with $Y_0$ being identified with $W$. }
\label{fig:casc}
\end{figure}
%==============================================================================
the $m$-channel quantum Wiener process 
$W$ in (\ref{W}) as described in Section~\ref{sec:LQSS}.
The oscillators have output fields $Y_1, \ldots, Y_N$ of the same dimension \cite{GJ_2009,MWPY_2014}.
For every $k=1,\ldots, N$, the $k$th oscillator $O_k$ of the cascade is endowed with the initial space $\fH_k$ and a vector $X_k$ of $n_k$ system variables. In accordance with (\ref{Thetalin}), these system variables satisfy CCRs and commute with one another for different subsystems:
\begin{equation}
\label{XXk}
    [X_j(t), X_k(t)^{\rT}]
    =
    \left\{
    \begin{matrix}
     2i \Theta_k & {\rm if}\ j=k\\
     0 & {\rm otherwise}
    \end{matrix}
    \right.
\end{equation}
for any time $t\>0$ and all $j,k=1,\ldots, N$, where $\Theta_k\in \mA_{n_k}$ are nonsingular CCR matrices. 
The system and output field variables are governed by a set of $N$ coupled QSDEs
\begin{align}
\label{dXk}
    \rd X_k
    & =
    A_kX_k \rd t + B_k \rd Y_{k-1},\\
\label{dYk}
    \rd Y_k
    & = C_kX_k\rd t + \rd Y_{k-1},
\end{align}
where $Y_0$ is identified with $W$, as mentioned before.   The matrices
$A_k\in \mR^{n_k\x n_k}$, $B_k\in \mR^{n_k\x m}$, $C_k \in \mR^{m\x n_k}$ of these QSDEs are given by
\begin{align}
\label{Ak}
    A_k
    & := 2\Theta_k (R_k + M_k^{\rT}JM_k),\\% = 2\Theta R - \frac{1}{2}BJB^{\rT}\Theta^{-1},
\label{Bk}
    B_k
    &:= 2\Theta_k M_k^{\rT},\\
\label{Ck}
    C_k
    & := 2J M_k,
\end{align}
similarly to (\ref{Alin})--(\ref{Clin}).
Here, $J\in \mA_m$ is the imaginary part of the quantum Ito matrix $\Omega$ in (\ref{Omega}) which is
inherited by the output fields $Y_1, \ldots, Y_N$ in (\ref{dYk}) from the input quantum Wiener process $W$ in (\ref{W}):
\begin{equation}
\label{YYk}
    \rd Y_j\rd Y_k^{\rT}
    =
    \rd W \rd W^{\rT}
    =
    \Omega\rd t
\end{equation}
for all $j,k=1,\ldots, N$.
Also, similarly to (\ref{H}) and (\ref{LM}), the energy and coupling matrices $R_k \in \mS_{n_k}$  and $M_k\in \mR^{m\x n_k}$ in (\ref{Ak})--(\ref{Ck}) specify the system Hamiltonian $H_k$ and the vector $L_k$ of $m$ operators of system-field  coupling of the $k$th oscillator $O_k$ to the output $Y_{k-1}$ of the preceding oscillator $O_{k-1}$ in the cascade (see Fig.~\ref{fig:casc}):
\begin{equation}
\label{HLk}
        H_k
        :=\frac{1}{2}X_k^{\rT}R_k X_k,
        \qquad
        L_k := M_k X_k.
\end{equation}
In accordance with (\ref{APR}) and (\ref{BCPR}), the state-space matrices $A_k$, $B_k$, $C_k$ in (\ref{Ak})--(\ref{Ck}) satisfy the PR conditions 
\begin{align}
\label{APRk}
    A_k \Theta_k + \Theta_k A_k^{\rT} + B_kJB_k^{\rT} & = 0,\\
\label{BCPRk}
    \Theta_k C_k^{\rT} + B_kJ & = 0,
\end{align}
which pertain to the preservation of the CCRs (\ref{XXk}) and the commutativity between the future system variables and the past output variables of the cascaded oscillators:
\begin{equation}
\label{XYk}
        [X_j(t),Y_k(\tau)^{\rT}]
     =
    0,
    \qquad
    j,k=1,\ldots, N,
    \quad
    t\> \tau\> 0.
\end{equation}
The above cascade can be regarded as a composite OQHO with the augmented vector of $n:= n_1+\ldots +n_N$ system variables
\begin{equation}
\label{cX}
    \cX
    :=
    \begin{bmatrix}
    X_1\\
    \vdots\\
    X_N
    \end{bmatrix}
\end{equation}
which satisfies a linear QSDE obtained by combining the QSDEs (\ref{dXk}) and (\ref{dYk}):
\begin{align}
\label{dcX}
    \rd \cX
    & =
    \cA \cX \rd t + \cB \rd W,\\
\label{dcY}
    \rd Y_N
    & = \cC \cX \rd t + \rd W.
\end{align}
The output of this composite system is the output $Y_N$ of the $N$th oscillator $O_N$ in the cascade. Also, the matrices $\cA\in \mR^{n\x n}$, $\cB\in \mR^{n\x m}$, $\cC\in \mR^{m\x n}$ in (\ref{dcX}) and (\ref{dcY}) are the last elements
\begin{equation}
\label{cABCN}
    \cA:= \cA_N,
    \qquad
    \cB:=\cB_N,
    \qquad
    \cC:=\cC_N
\end{equation}
of a matrix sequence $(\cA_k,\cB_k,\cC_k)_{1\< k\< N}$ computed recursively as
\begin{align}
\nonumber
    \cA_k
    & =
    \begin{bmatrix}
    A_{11} & \cdots & A_{1k}\\
    \vdots & \ddots & \vdots\\
    A_{k1} & \cdots & A_{kk}
    \end{bmatrix}\\
\nonumber
    & =
    \begin{bmatrix}
        A_1 & 0 & 0 & \cdots& 0\\
        B_2 C_1 & A_2 & 0 & \cdots & 0 \\
%        B_3 C_1 & B_3 C_2 & A_3 &\cdots &  0 \\
        \cdots &\cdots  & \cdots &\cdots & \vdots \\
        B_k C_1 & B_k C_2 & \cdots & B_k C_{k-1} & A_k
    \end{bmatrix}\\
\label{cAk}
    & =
    \begin{bmatrix}
        \cA_{k-1} & 0\\
        B_k \cC_{k-1} & A_k
    \end{bmatrix},\\
\label{cBk}
    \cB_k
    & =
    \begin{bmatrix}
        B_1\\
        \vdots\\
        B_k
    \end{bmatrix}
    =
    \begin{bmatrix}
        \cB_{k-1}\\
        B_k
    \end{bmatrix},\\
\label{cCk}
    \cC_k
    & =
    \begin{bmatrix}
        C_1 & \cdots & C_k
    \end{bmatrix}
    =
    \begin{bmatrix}
        \cC_{k-1} & C_k
    \end{bmatrix},
\end{align}
with the initial conditions being the matrices of the first oscillator $O_1$ in the cascade:
\begin{equation}
\label{cABC1}
    \cA_1:= A_1,
    \qquad
    \cB_1:= B_1,
    \qquad
    \cC_1:= C_1,
\end{equation}
see also \cite{MWPY_2014}. In view of (\ref{Ak})--(\ref{Ck}) and (\ref{cAk}), the blocks $A_{jk} \in \mR^{n_j \x n_k}$ of the matrix $\cA$ are given by
\begin{equation}
\label{Ajk}
    A_{jk}
    :=
    2\Theta_j
    \left\{
    \begin{matrix}
        R_k + M_k^{\rT} J M_k & {\rm if} & j=k\\
        2 M_j^{\rT} J M_k & {\rm if} & j>k\\
        0& {\rm if} & j<k
    \end{matrix}
    \right.
\end{equation}
for all $j,k=1, \ldots, N$. The corresponding augmented version of the PR conditions (\ref{APRk}) and (\ref{BCPRk}) takes the form
\begin{align}
\label{cAPR}
    \cA \bTheta + \bTheta \cA^{\rT} + \cB J\cB^{\rT} & = 0,\\
\label{cBCPR}
    \bTheta \cC^{\rT} + \cB J & = 0,
\end{align}
where
\begin{equation}
\label{Theta}
  \bTheta
  :=
  \diag_{1\< k\< N}(\Theta_k)
\end{equation}
is a block-diagonal CCR matrix of order $n$ for the augmented vector $\cX$ in (\ref{cX}) which represents the CCRs (\ref{XXk}) as
\begin{equation}
\label{cXX}
    [\cX, \cX^{\rT}]
    =
     2i \bTheta.
\end{equation}
Similarly to (\ref{APRk}) and (\ref{BCPRk}), the PR conditions (\ref{cAPR}) and (\ref{cBCPR}) are equivalent to the preservation of the CCRs (\ref{cXX}) and the commutativity (\ref{XYk}), respectively.

In accordance with \cite[Theorem 5.1]{NJD_2009} and \cite[Lemma 3]{N_2010}, a combination of (\ref{Ak})--(\ref{Ck}) with (\ref{cABCN})--(\ref{Ajk}) and (\ref{Theta}) leads to the following energy and coupling matrices $\cR \in \mS_n$ and $\cM \in \mR^{m\x n}$ of the composite OQHO:
\begin{equation}
\label{cRM}
  \cR
  =
    \begin{bmatrix}
    R_{11} & \cdots & R_{1N}\\
    \vdots & \ddots & \vdots\\
    R_{N1} & \cdots & R_{NN}
    \end{bmatrix},
    \qquad
    \cM
    =
    \begin{bmatrix}
    M_1 & \cdots & M_N
    \end{bmatrix},
\end{equation}
where the blocks $R_{jk} \in \mR^{n_j\x n_k}$ are computed for all $j,k=1, \ldots, N$ as
\begin{equation}
\label{Rjk}
    R_{jk}
    :=
    \left\{
    \begin{matrix}
        R_k & {\rm if} & j=k\\
        M_j^{\rT} J M_k & {\rm if} & j>k\\
        -M_j^{\rT} J M_k & {\rm if} & j<k
    \end{matrix}
    \right..
\end{equation}
The energy matrices $R_1, \ldots, R_N$ of the component oscillators enter the energy matrix $\cR$ of the composite system only through its diagonal part
\begin{equation}
\label{bitR}
  \bitR
  :=
  \diag_{1\< k\< N}
  (R_k)
\end{equation}
which takes values in a Hilbert space $\fR$, isomorphic  to the Hilbert space $\mS_{n_1}\x \ldots \x \mS_{n_N}$ with the direct-sum inner product (associated with the Frobenius  inner product of real matrices $\bra L, M\ket:= \Tr(L^{\rT} M)$ and the corresponding norm $\|M\|:=\sqrt{\bra M, M\ket}$).
The specific quadratic dependence of the energy matrix $\cR$ on the coupling matrices $M_1, \ldots, M_N$ in (\ref{cRM}) and  (\ref{Rjk}) is closely related to the block lower triangularity  of the matrix $\cA$ of the augmented QSDE (\ref{dcX}) due to the cascade structure of the quantum system being considered.
While the matrix $\cA$,   which can be  represented as
\begin{equation}
\label{cARM}
    \cA = 2\bTheta (\cR + \cM^{\rT} J \cM),
\end{equation}
depends linearly on the energy matrices $R_1, \ldots, R_N$ and quadratically on the coupling matrices $M_1, \ldots, M_N$ of the constituent OQHOs, the dependence of the matrices $\cB$ and $\cC$ in (\ref{dcX}) and (\ref{dcY}) on the coupling matrices is linear, similarly to (\ref{Bk}) and (\ref{Ck}):
\begin{align}
\label{cBM}
    \cB  &= 2\bTheta \cM^{\rT},\\
\nonumber
    \cC & = 2J \cM.
\end{align}
The Hamiltonian $\frac{1}{2}\cX^{\rT}\cR \cX$ and the vector $\cM \cX$  of operators of coupling with the external input $W$, which are computed  for the composite OQHO  in terms of the component parameters according to  (\ref{cRM}) and (\ref{Rjk}),  can also be obtained by using the quantum feedback network formalism \cite{GJ_2009,JG_2010} and the individual energy operators in (\ref{HLk}).

%%%%%%%%%%%%%%%%%%%%%%%%%%%%%%%%%%%%%%%%%%%%%%%%%%%%%%%%%%%%%%%%%%%%%%%%%%%%%%%%%%%%%%%%%%%%%%%%%%%
\section{Gaussian invariant quantum state}
\label{sec:inv}
%%%%%%%%%%%%%%%%%%%%%%%%%%%%%%%%%%%%%%%%%%%%%%%%%%%%%%%%%%%%%%%%%%%%%%%%%%%%%%%%%%%%%%%%%%%%%%%%%%%

If the matrices $A_1, \ldots, A_N$ in (\ref{Ak}) are Hurwitz, then so is the block lower triangular matrix $\cA$ in (\ref{cABCN}). In this case, in accordance with Section~\ref{sec:LQSS}, the composite OQHO has a unique invariant state which is Gaussian with zero mean vector and the quantum covariance matrix $\cP + i\bTheta\succcurlyeq 0$, whose real part
\begin{equation}
\label{cP}
  \cP := \int_0^{+\infty}\re^{t\cA}\cB\cB^{\rT}\re^{t\cA^{\rT}}\rd t
\end{equation}
is the controllability Gramian of the pair $(\cA, \cB)$, similarly to (\ref{P}) and (\ref{PALE}), and is found uniquely by solving the ALE
\begin{equation}
\label{cPALE}
    \cA \cP + \cP\cA^{\rT} + \cB \cB^{\rT} = 0.
\end{equation}
Similarly to (\ref{Ppos}) and (\ref{Ppos1}), the state-space and frequency-domain representations of the invariant quantum covariance matrix are  given by
\begin{align}
\nonumber
    \cP+i\bTheta
    & =
    \int_0^{+\infty}
    \re^{t\cA}\cB\Omega
    \cB^{\rT}\re^{t\cA^{\rT}}\rd t    \\
\label{cPpos1}
    & =
    \frac{1}{2\pi}
    \int_{-\infty}^{+\infty}
    \cF(i\lambda)\Omega \cF(i\lambda)^*
    \rd \lambda.
\end{align}
Here, $\cF$ is the $\mC^{n\x m}$-valued rational transfer function from the input quantum Wiener process $W$ to the system variables in $\cX$ computed as
\begin{equation}
\label{FG}
    \cF
    :=
    \begin{bmatrix}
        F_1\\
        F_2 G_1\\
        F_3 G_2G_1\\
        \vdots\\
        F_N G_{N-1}\x \ldots \x G_1
    \end{bmatrix},
\end{equation}
where
\begin{align}
\label{Fk}
    F_k(s)
    & := (sI_{n_k} - A_k)^{-1}B_k,\\
\label{Gk}
    G_k(s)
    & := C_k F_k(s) + I_m
\end{align}
are  the transfer functions from $Y_{k-1}$ to $X_k$ and $Y_k$, respectively, associated with the QSDEs (\ref{dXk}) and (\ref{dYk}) for the $k$th OQHO. 
Similarly to (\ref{Xtilde}) and (\ref{Ytilde}), these transfer functions  relate the Laplace transforms of $X_k$ and $Y_k$ to that of $W$ in (\ref{Wtilde}) in the right-half plane $\Re s>0$ by
\begin{align}
\nonumber
    \wt{X}_k(s)
    := &
    \int_{0}^{+\infty}
    \re^{-st}
    X_k(t)\rd t \\
\label{Xktilde}
    = &
    F_k(s)\wt{Y}_{k-1}(s)
    +
    (sI_{n_k} - A_k)^{-1}X_k(0),\\
\nonumber
    \wt{Y}_k(s)
    := &
    \int_{0}^{+\infty}
    \re^{-st}
    \rd Y_k(t) \\
\nonumber
    = &
    C_k \wt{X}_k(s) + \wt{Y}_{k-1}(s)\\
\nonumber
    = &
    G_k(s)\wt{Y}_{k-1}(s)
    +
    C_k(sI_{n_k} - A_k)^{-1}X_k(0)\\
\nonumber
    = &
    G_k(s)\x \ldots \x G_1(s)\wt{W}(s)\\
\label{Yktilde}
    & +
    \sum_{j=1}^k
    G_k(s)\x \ldots \x G_{j+1}(s)
    C_j(sI_{n_j} - A_j)^{-1}X_j(0).
\end{align}
In view of the PR conditions (\ref{APRk}) and (\ref{BCPRk}), each of the functions $G_1, \ldots, G_N$ in (\ref{Gk}) is $(J,J)$-unitary in the sense of (\ref{GJJ}). 
The products of the transfer matrices in (\ref{FG}) are closely related to the concatenation of quantum systems \cite{GJ_2009,JG_2010} in the case of linear QSDEs being considered \cite{YK_2003}. Since the transfer functions $G_k$  in (\ref{Gk}) have an identity feedthrough matrix $I_m$, so also do their products.  Hence, their $\cH_{\infty}$-norms satisfy
\begin{equation}
\label{Gprodnorm}
    \|G_k\x \ldots \x G_1\|_{\infty}\> 1,
    \qquad
    k=1,\ldots, N,
\end{equation}
which also follows from the $(J,J)$-unitarity\footnote{and the fact that such matrices form a group} of $G_k$, 
whereby no contraction can be guaranteed in (\ref{FG}). This can potentially facilitate the propagation of modelling errors over the cascade (especially from the first oscillators towards the end in long cascades), which will be discussed for the translation invariant case in  Appendix~\ref{sec:trans}.

Now, a given admissible real part $\cP$ of the quantum covariance matrix  $\cP+i\bTheta \succcurlyeq 0$ for the invariant Gaussian state of the cascaded OQHOs in (\ref{XXk})--(\ref{dYk}) can be achieved by choosing the energy and coupling matrices $R_1, \ldots, R_N$ and $M_1, \ldots, M_N$ so as to make the matrices $A_1, \ldots, A_N$ in (\ref{Ak}) Hurwitz and to satisfy (\ref{cP}); see, for example, \cite{MWPY_2014,Y_2012}. 
Inaccuracies in these matrices leave the CCRs (\ref{cXX}) intact and can
only perturb the matrix $\cP$ without destroying the Gaussian nature of the invariant state. 
In application to the invariant Gaussian state of the composite OQHO (with the reduced density operator $r$), the purity functional (\ref{puritylin}) takes the form
\begin{equation}
\label{purity}
    \Tr (r^2)
    =
    \sqrt{\frac{\det \bTheta}{\det \cP}}
    \< 1,
\end{equation}
where only the denominator depends on the matrices $R_1, \ldots, R_N$ and $M_1, \ldots, M_N$ of the oscillators. 
Note that the purity functional (\ref{purity}) is invariant with respect to the transformations
\begin{equation}
\label{SRS}
    R_k \mapsto S_k^{-\rT} R_k S_k^{-1},
    \qquad
    M_k \mapsto M_k S_k^{-1}
\end{equation}
for arbitrary matrices $S_k \in \mR^{n_k\x n_k}$, satisfying $S_k\Theta_k S_k^{\rT} = \Theta_k$ and forming the symplectic group $\Sp(\Theta_k)$ (which agrees with the usual definition of the symplectic group \cite{D_2006,O_1993,V_1984} up to the matrix transpose). With the transfer functions $G_k$ in (\ref{Gk}) remaining  unchanged,  (\ref{SRS}) corresponds to the
symplectic transformations
\begin{equation}
\label{SXk}
    X_k\mapsto S_k X_k
\end{equation}
of the system variables of the component oscillators which preserve the CCRs (\ref{cXX}) since
the matrix
\begin{equation}
\label{diagS}
    S:= \diag_{1\< k\< N} (S_k)
\end{equation}
belongs to $\Sp(\bTheta)$, where $\bTheta$ is given by (\ref{Theta}).  From the property $\det S = 1$ for any such matrix $S$, it follows that the corresponding transformation
\begin{equation}
\label{SPS}
    \cP \mapsto S \cP S^{\rT}
\end{equation}
leaves $\det(S \cP S^{\rT}) = \det \cP$ unchanged, and hence, the purity functional in (\ref{purity}) is indeed invariant under such transformations. Since (\ref{purity}) can be  expressed in terms of the log determinant of the matrix $\cP$ as
\begin{equation}
\label{V}
    \Tr (r^2)
    =
    \sqrt{\det\bTheta}
    \re^{-\frac{1}{2}V},
    \qquad
    V:= \ln\det \cP,
\end{equation}
the minimization of the functional $V$ and its sensitivity to the perturbations in the energy and coupling matrices is relevant to the  purity  of Gaussian quantum states generated through cascaded oscillators \cite{MWPY_2014,N_2010,Y_2012}.

%%%%%%%%%%%%%%%%%%%%%%%%%%%%%%%%%%%%%%%%%%%%%%%%%%%%%%%%%%%%%%%%%%%%%%%%%%%%%%%%%%%%%%%%%%%%%%%%%%%
\section{Recursive computation of the steady-state covariances}
\label{sec:reccov}
%%%%%%%%%%%%%%%%%%%%%%%%%%%%%%%%%%%%%%%%%%%%%%%%%%%%%%%%%%%%%%%%%%%%%%%%%%%%%%%%%%%%%%%%%%%%%%%%%%%

The cascade structure of the composite quantum system under consideration (see Fig.~\ref{fig:casc}) leads to ``spatial causality'' in the dependence of the matrix $\cP$ in (\ref{cP}) on the energy and coupling matrices of the component oscillators. In accordance with the partitioning of the vector $\cX$ in (\ref{cX}), the matrix $\cP$ can be split into blocks $P_{jk}\in \mR^{n_j\x n_k}$ as
\begin{equation}
\label{cPN}
    \cP
    =
    \begin{bmatrix}
    P_{11} & \cdots & P_{1N}\\
    \vdots & \ddots & \vdots\\
    P_{N1} & \cdots & P_{NN}
    \end{bmatrix}
    =
    \cP_N
\end{equation}
and regarded as the last element in the sequence of matrices $\cP_1, \ldots, \cP_N$ given by
\begin{equation}
\label{cPk}
    \cP_k
    :=
    \begin{bmatrix}
    P_{11} & \cdots & P_{1k}\\
    \vdots & \ddots & \vdots\\
    P_{k1} & \cdots & P_{kk}
    \end{bmatrix}
    =
    \begin{bmatrix}
    \cP_{k-1}& Q_k^{\rT}\\
    Q_k & P_{kk}
    \end{bmatrix},
\end{equation}
where
\begin{equation}
\label{Qk}
    Q_k
    :=
    \begin{bmatrix}
    P_{k1} & \cdots & P_{k,k-1}
    \end{bmatrix}    .
\end{equation}
The matrix $\cP_k$ in (\ref{cPk}) is the real part of the quantum covariance matrix $\cP_k + i\bTheta_k$
of the reduced Gaussian invariant state for the oscillators $O_1, \ldots, O_k$, where
$
    \bTheta_k
  :=
  \diag_{1\< j\< k}(\Theta_j)
$
is the CCR matrix of their system variables constituting the vectors  $X_1, \ldots, X_k$, which corresponds to (\ref{Theta}).
The above mentioned spatial causality is the property that $\cP_j$ is  independent of $R_k$ and $M_k$ for all $1\< j<k\< N$. This property is closely related to the following recursive computation of the matrix $\cP$.

%%%%%%%%%%%%%%%%%%%%%%%%%%%%%%%%%%%%%%%%%%%%%%%%%%%%%%%%%%%%%%%%%%%%%%%%%%%%%%%%%%%%%%%%%%%%%%%%%%%
\begin{lemma}
\label{lem:cP}
Suppose the matrices $A_1, \ldots, A_N$ in (\ref{Ak}) are Hurwitz. Then the matrix $\cP$ in (\ref{cPN}) can be found by recursively computing the matrices $\cP_1, \ldots, \cP_N$ in (\ref{cPk}) through solving the ALEs
\begin{align}
\nonumber
    A_k Q_k  & +Q_k\cA_{k-1}^{\rT}\\
\label{QkALE}
    & +  B_k (\cC_{k-1}\cP_{k-1}   + \cB_{k-1}^{\rT}) = 0,\\
\nonumber
    A_k P_{kk} & + P_{kk} A_k^{\rT}\\
\label{PkkALE}
    & +
    B_k \cC_{k-1} Q_k^{\rT} + Q_k \cC_{k-1}^{\rT}B_k^{\rT} + B_k B_k^{\rT} = 0
\end{align}
with respect to  the matrices $Q_k$ and $P_{kk}$ in (\ref{Qk}) for all $k=2, \ldots, N$. Here, the initial condition $\cP_1=P_{11}$ is obtained by solving the ALE
\begin{equation}
\label{P11ALE}
  A_1 P_{11} + P_{11} A_1^{\rT} + B_1 B_1^{\rT} = 0.
\end{equation}
\hfill$\square$
\end{lemma}
%%%%%%%%%%%%%%%%%%%%%%%%%%%%%%%%%%%%%%%%%%%%%%%%%%%%%%%%%%%%%%%%%%%%%%%%%%%%%%%%%%%%%%%%%%%%%%%%%%%
\begin{proof}
For any $k=1, \ldots, N$, the matrix $\cP_k$, which pertains to the steady-state covariance dynamics of the system variables for the first $k$ oscillators in the cascade,  satisfies the ALE
\begin{equation}
\label{cPkALE}
    \cA_k \cP_k + \cP_k \cA_k^{\rT} + \cB_k\cB_k^{\rT} = 0,
\end{equation}
where the matrices $\cA_k$ and $\cB_k$ are given by (\ref{cAk}) and (\ref{cBk}).
The ALE (\ref{P11ALE}) is obtained by letting $k=1$ in (\ref{cPkALE}). For any $k=2, \ldots, N$, the structure of the matrices $\cA_k$ and $\cB_k$ in (\ref{cAk}), (\ref{cBk}) and $\cP_k$ in (\ref{cPk}) leads to
\begin{align*}
    \cA_k
    \cP_k
    & =
    \begin{bmatrix}
        \cA_{k-1} \cP_{k-1} & \cA_{k-1} Q_k^{\rT}\\
        B_k \cC_{k-1} \cP_{k-1} + A_k Q_k & B_k \cC_{k-1} Q_k^{\rT} + A_k P_{kk}
    \end{bmatrix},\\
    \cB_k
    \cB_k^{\rT}
    & =
    \begin{bmatrix}
        \cB_{k-1} \cB_{k-1}^{\rT} & \cB_{k-1} B_k^{\rT}\\
        B_k \cB_{k-1}^{\rT} & B_k B_k^{\rT}
    \end{bmatrix}.
\end{align*}
Therefore, since the left-hand side of (\ref{cPkALE}) is a symmetric matrix, this ALE splits
into three equations consisting of (\ref{QkALE}), (\ref{PkkALE}) and
\begin{equation*}
    \cA_{k-1} \cP_{k-1}
    +
    \cP_{k-1} \cA_{k-1}^{\rT}
    +
    \cB_{k-1}\cB_{k-1}^{\rT} = 0.
\end{equation*}
The last equation reproduces (\ref{cPkALE}) with the preceding value of $k$, thus completing the proof by a standard  induction argument.
\end{proof}
%%%%%%%%%%%%%%%%%%%%%%%%%%%%%%%%%%%%%%%%%%%%%%%%%%%%%%%%%%%%%%%%%%%%%%%%%%%%%%%%%%%%%%%%%%%%%%%%%%%

Note that for any given $k=1, \ldots, N$, the energy and coupling matrices $R_k$ and $M_k$ of the $k$th oscillator influence the matrix $\cP$ through the $k$th block rows and the $k$th block columns of the matrices $\cA$ and $\cB\cB^{\rT}$ in the ALE (\ref{cPALE}). The corresponding ``cross-shaped'' fragments of $\cA$ and $\cB\cB^{\rT}$ are described by
$$
    \begin{bmatrix}
    * & * & * \\
    B_k \cC_{k-1} & A_k & *\\
     * & \cB_{>k} C_k & *
    \end{bmatrix},
    \quad
    \begin{bmatrix}
    * & \cB_{k-1} B_k^{\rT} & * \\
    B_k\cB_{k-1}^{\rT} & B_kB_k^{\rT} & B_k \cB_{>k}^{\rT}\\
    * & \cB_{>k} B_k^{\rT} & *
    \end{bmatrix},
$$
where use is made of (\ref{cAk}) and (\ref{cBk}) together with  an auxiliary matrix
\begin{equation}
\label{cB>k}
    \cB_{>k}
    =
    \cB_{\> k+1}
    =
    \begin{bmatrix}
        B_{k+1}\\
        \vdots\\
        B_N
    \end{bmatrix}
    =
    \begin{bmatrix}
        B_{k+1}\\
        \cB_{>k+1}
    \end{bmatrix}.
\end{equation}
At the level of covariances,  there is a connection between the quantum system variables\footnote{which, as mentioned before,  are not accessible to simultaneous measurement and conditional averaging because of their noncommutativity \cite{H_2001}} and classical random variables.
More precisely, the real parts of the  steady-state covariances of the system variables in the cascaded OQHOs can be reproduced by considering a sequence of jointly Gaussian classical random vectors $\xi_1, \ldots, \xi_N$ of dimensions $n_1, \ldots, n_N$ in the form
\begin{equation}
\label{Xik}
    \Xi_k
    :=
    \begin{bmatrix}
        \xi_1\\
        \vdots\\
        \xi_k
    \end{bmatrix}=
    \begin{bmatrix}
        \Xi_{k-1}\\
        \xi_k
    \end{bmatrix}\\
    =
    \int_0^{+\infty}
    \re^{t\cA_k}
    \cB_k
    \rd \omega(t),
\end{equation}
where $\omega$ is a standard Wiener process in $\mR^m$,  and the improper integral is convergent (in particular, in the mean square sense) since the matrix $\cA_k$ is Hurwitz. Indeed, the covariance  matrices of these auxiliary random vectors coincide with the blocks of the matrix $\cP$ in (\ref{cP}) and (\ref{cPN})--(\ref{Qk}):
\begin{equation}
\label{PQ}
    \cov(\xi_j,\xi_k) = P_{jk},
    \qquad
    \cov(\xi_k,\Xi_{k-1}) = Q_k
\end{equation}
for all $j,k=1,\ldots, N$, in terms of which the functional $V$ in (\ref{V}) admits the decomposition
\begin{equation}
\label{VV}
    V = \sum_{k=1}^N V_k,
    \qquad
        V_k
    :=
    \ln \det \Pi_k.
\end{equation}
Here, the matrix $\Pi_k$ is the Schur complement \cite{HJ_2007} of the block $\cP_{k-1}$ in the matrix $\cP_k$ in (\ref{cPk}), which is given by
\begin{align}
\nonumber
    \Pi_k
    & :=
    P_{kk} - Q_k\cP_{k-1}^{-1} Q_k^{\rT}    \\
\nonumber
    & =
    \cov(\xi_k \mid \Xi_{k-1})\\
\label{Pik}
    & =
    \bE\big((\xi_k-\wh{\xi}_k)(\xi_k-\wh{\xi}_k)^{\rT} \mid \Xi_{k-1}\big)
\end{align}
and
coincides  with the conditional covariance matrix  of the random vector $\xi_k$ given the ``past history'' $\Xi_{k-1}$ (in the sense of the spatial parameter $k$). Also,
\begin{equation}
\label{xihat}
    \wh{\xi}_k
    :=
    \bE(\xi_k\mid \Xi_{k-1})
    =
    Q_k\cP_{k-1}^{-1} \Xi_{k-1}
\end{equation}
is the corresponding predictor.
In (\ref{Pik}) and (\ref{xihat}),
use is also made of (\ref{Xik}),  (\ref{PQ}) and the structure of conditional distributions for jointly Gaussian random vectors, which plays an important role in linear stochastic filtering \cite{AM_1979,LS_2001}.

The above mentioned spatial causality manifests itself in the  fact that, for all $1\< j<k\< N$, the Schur complement $\Pi_j$ in (\ref{Pik}) is  independent of the energy and coupling matrices $R_k$ and $M_k$ of the $k$th oscillator,  and hence, so is the quantity $V_j$ in (\ref{VV}).

%%%%%%%%%%%%%%%%%%%%%%%%%%%%%%%%%%%%%%%%%%%%%%%%%%%%%%%%%%%%%%%%%%%%%%%%%%%%%%%%%%%%%%%%%%%%%%%%%%%
\section{Infinitesimal perturbation analysis of the purity functional}
\label{sec:diff}
%%%%%%%%%%%%%%%%%%%%%%%%%%%%%%%%%%%%%%%%%%%%%%%%%%%%%%%%%%%%%%%%%%%%%%%%%%%%%%%%%%%%%%%%%%%%%%%%%%%

As mentioned in Section~\ref{sec:inv}, inaccuracies in the energy and coupling matrices of the cascaded OQHOs can affect the purity functional (\ref{purity}) for the Gaussian invariant   state of the system. Its sensitivity to infinitely small perturbations can be described in terms of the Frechet derivatives of the functional $V$ in (\ref{V})  with respect to the matrices $R_1, \ldots, R_N$ and $M_1, \ldots, M_N$ which constitute the matrices $\bitR$ and $\cM$ in (\ref{bitR}) and (\ref{cRM}). To this end, we will first establish a preliminary lemma which computes such derivatives for the invariant covariance matrix of the composite system. More precisely, the condition that the matrix $\cA$ of the composite quantum system (\ref{dcX}) is Hurwitz ensures smooth dependence of the matrix $\cP$ in (\ref{cP}) on the $\fR\x \mR^{m\x n}$-valued pair $(\bitR, \cM)$ comprising the  energy and coupling matrices of the cascaded oscillators. This smoothness leads to a well-defined Frechet derivative
\begin{equation}
\label{dPdRM0}
    \Lambda:= \Lambda_{\bitR,\cM} := \d_{\bitR, \cM} \cP
\end{equation}
which (for a given pair $(\bitR, \cM)$) is a linear operator acting from the space $\fR\x \mR^{m\x n}$ to $\mS_n$. The first variation of the matrix $\cP$ is expressed in terms of $\Lambda$ as
\begin{align}
\nonumber
    \delta\cP
    & =
    \Lambda(\delta\bitR, \delta\cM)\\
\label{dPdRM}
    & =
    \d_{\bitR}\cP(\delta \bitR)
    +
    \d_{\cM}\cP(\delta \cM),
\end{align}
where $\d_{\bitR}\cP$ and $\d_{\cM}\cP$ denote the corresponding partial Frechet derivatives.
For what follows, we associate with arbitrary Hurwitz matrices $\alpha$ and $\beta$ a linear operator  $\bL_{\alpha,\beta}$ which maps an appropriately dimensioned matrix $\gamma$ to
\begin{equation}
\label{bL}
    \bL_{\alpha,\beta}(\gamma) := \int_0^{+\infty}\re^{t\alpha}\gamma\re^{t\beta^{\rT}}\rd t,
\end{equation}
which is a unique solution $\sigma$ of an algebraic  Sylvester equation (ASE), or a generalized ALE $\alpha \sigma + \sigma \beta^{\rT} + \gamma = 0$; see, for example, \cite{GLAM_1992}.
In the case when $\alpha=\beta$, which corresponds to standard ALEs, an abbreviated notation
\begin{equation}
\label{bLbL}
    \bL_{\alpha}:= \bL_{\alpha,\alpha}
\end{equation}
will be used. 
Also, we denote by $\[[[\alpha, \beta\]]]$ a ``sandwich'' operator which acts on an appropriately dimensioned matrix $\gamma$ as
\begin{equation}
\label{sand}
    \[[[\alpha,\beta\]]](\gamma) := \alpha \gamma \beta.
\end{equation}
Furthermore, it will be convenient to denote the symmetrizer and antisymmetrizer of square matrices by
\begin{align}
\label{bSbA}
    \bS(M)
    :=
    \frac{1}{2}(M+M^{\rT}),
    \qquad
    \bA(M)
    :=
    \frac{1}{2}(M-M^{\rT}).
\end{align}
Also, we will need an auxiliary linear operator $\Gamma:= \Gamma_{\cM}$, which is associated with the coupling matrix $\cM$ in (\ref{cRM}) and maps   its variation $\delta \cM$ to a block lower triangular matrix $\Gamma(\delta \cM ) := (\Gamma_{jk}(\delta \cM ))_{1\< j,k\< N}$ whose blocks are given by
\begin{equation}
\label{Gammajk}
    \Gamma_{jk}(\delta \cM ) =
    \left\{
    \begin{matrix}
        \bA(M_k^{\rT} J \delta M_k) & {\rm if} & j=k\\
        (\delta M_j)^{\rT} J M_k + M_j^{\rT} J \delta M_k & {\rm if} & j>k\\
        0& {\rm if} & j<k
    \end{matrix}
    \right..
\end{equation}
The first variation of the block $A_{jk}$ of the matrix  $\cA$ in (\ref{Ajk}) with respect to the matrix $\cM$ is expressed in terms of (\ref{Gammajk}) as
$$
    \delta_{\cM} A_{jk} =  4\Theta_j \Gamma_{jk}(\delta \cM)
$$
for all $j,k=1,\ldots, N$.
Therefore, the operator $\Gamma$ allows the partial Frechet derivative of $\cA$ with respect to $\cM$ to be represented as the composition
\begin{equation}
\label{dAdM}
    \d_{\cM} \cA = 4\[[[\bTheta, I_n\]]] \Gamma,
\end{equation}
where (\ref{Theta}) and (\ref{sand}) are used.
The following lemma carries out an infinitesimal perturbation analysis for the steady-state covariances of the quantum system.

%%%%%%%%%%%%%%%%%%%%%%%%%%%%%%%%%%%%%%%%%%%%%%%%%%%%%%%%%%%%%%%%%%%%%%%%%%%%%%%%%%%%%%%%%%%%%%%%%%%
\begin{lemma}
\label{lem:dPdRM}
Suppose the matrices $A_1, \ldots, A_N$ in (\ref{Ak}) are Hurwitz.
Then the partial Frechet derivatives  $\d_{\bitR}\cP$ and $\d_{\cM}\cP$ in (\ref{dPdRM}) are computed as
\begin{align}
\label{dPdR}
    \d_{\bitR}\cP
    & =
    4\bL_{\cA}\bS\[[[\bTheta, \cP\]]],\\
\label{dPdM}
    \d_{\cM}\cP
    & =
    4\bL_{\cA}\bS
    (2\[[[\bTheta, \cP\]]]\Gamma
    -
    \[[[\cB,\bTheta\]]])
\end{align}
in terms of the operators (\ref{bL})--(\ref{dAdM}).
\hfill$\square$
\end{lemma}
%%%%%%%%%%%%%%%%%%%%%%%%%%%%%%%%%%%%%%%%%%%%%%%%%%%%%%%%%%%%%%%%%%%%%%%%%%%%%%%%%%%%%%%%%%%%%%%%%%%
\begin{proof}
The proof is carried out by using the algebraic techniques of \cite{SIG_1998,VP_2013a} justified by the smooth dependence of $\cP$ on the energy and coupling matrices under the condition that $\cA$ is Hurwitz. More precisely, the first variation of the ALE (\ref{cPALE}) leads to
\begin{align}
\nonumber
    \cA \delta \cP & + (\delta \cP)\cA^{\rT}\\
\nonumber
    & + (\delta \cA) \cP + \cP\delta \cA^{\rT}\\
\label{dPALE1}
    & +  (\delta \cB) \cB^{\rT}+ \cB \delta \cB^{\rT}= 0.
\end{align}
Since the matrix $\cP$ is symmetric, (\ref{dPALE1}) can be represented as
\begin{align}
\nonumber
    \cA \delta \cP  & + (\delta \cP)\cA^{\rT}\\
\label{dPALE2}
     & + 2\bS((\delta \cA) \cP + \cB\delta \cB^{\rT})= 0,
\end{align}
where use is made of the symmetrizer $\bS$ from (\ref{bSbA}). The relation (\ref{dPALE2}) is an ALE with respect to the matrix $\delta \cP$, whose solution can be expressed as
\begin{equation}
\label{cPbL}
    \delta \cP
    =
    2\bL_{\cA}(\bS((\delta \cA) \cP + \cB \delta \cB^{\rT}))
\end{equation}
in terms of the operator $\bL_{\cA}$ given by  (\ref{bL}) and (\ref{bLbL}). In view of (\ref{Ajk}) and (\ref{cARM}), the first variation of the matrix $\cA$,  as a function of the matrices $\bitR$ and $\cM$ from (\ref{bitR}) and (\ref{cRM}), is computed as
\begin{equation}
\label{dcA}
    \delta \cA
    =
    2\bTheta (\delta\bitR  + 2\Gamma(\delta \cM)) ,
\end{equation}
where use is made of the operator $\Gamma$ from (\ref{Gammajk}) and (\ref{dAdM}). By a similar reasoning, the first variation of the matrix $\cB$ in (\ref{cBM}) with respect to $\cM$ is  given by
\begin{equation}
\label{dcB}
    \delta \cB
    =
    2\bTheta \delta \cM^{\rT}.
\end{equation}
Substitution of (\ref{dcA}) and (\ref{dcB}) into (\ref{cPbL}) relates $\delta\cP$ to $\delta \bitR$ and $\delta\cM$ by
\begin{equation}
\label{cPbL1}
    \delta \cP
    =
    4\bL_{\cA}(\bS((\bTheta (\delta\bitR  + 2\Gamma(\delta \cM))) \cP -\cB \delta \cM \bTheta)),
\end{equation}
where use is also made of the antisymmetry of the CCR matrix $\bTheta$ in (\ref{Theta}). By comparing
(\ref{cPbL1}) with (\ref{dPdRM}), it follows  that the partial Frechet derivatives $\cP$ with respect to $\bitR$ and $\cM$ take the form (\ref{dPdR}) and (\ref{dPdM}).
\end{proof}
%%%%%%%%%%%%%%%%%%%%%%%%%%%%%%%%%%%%%%%%%%%%%%%%%%%%%%%%%%%%%%%%%%%%%%%%%%%%%%%%%%%%%%%%%%%%%%%%%%%%%%

The proof of Lemma~\ref{lem:dPdRM} leads to the following group theoretic property of the linear operator $\Lambda$ in (\ref{dPdRM0}) with respect to the symplectic similarity transformations of the cascaded oscillators described by  (\ref{SRS}) and (\ref{SXk}).

%%%%%%%%%%%%%%%%%%%%%%%%%%%%%%%%%%%%%%%%%%%%%%%%%%%%%%%%%%%%%%%%%%%%%%%%%%%%%%%%%%%%%%%%%%%%%%%%%%%
\begin{lemma}
\label{lem:dPdRMnew}
Suppose the matrices $A_1, \ldots, A_N$ in (\ref{Ak}) are Hurwitz.
Then, for any symplectic matrix $S\in \Sp(\bTheta)$ in (\ref{diagS}),  the Frechet derivatives $\Lambda$ in (\ref{dPdRM0}), evaluated  at the original and  transformed pairs $(\bitR,\cM)$ and  $(S^{-\rT} \bitR S^{-1}, \cM S^{-1})$, are related by
\begin{equation}
\label{dPdRMnew}
        \Lambda_{S^{-\rT} \bitR S^{-1}, \cM S^{-1}}
        (\delta \bitR, \delta\cM)
        =
        S
        \Lambda_{\bitR, \cM}
        \big(
            S^{\rT} (\delta \bitR) S, \,
            (\delta\cM) S
        \big)
        S^{\rT}.
\end{equation}
\hfill$\square$
\end{lemma}
%%%%%%%%%%%%%%%%%%%%%%%%%%%%%%%%%%%%%%%%%%%%%%%%%%%%%%%%%%%%%%%%%%%%%%%%%%%%%%%%%%%%%%%%%%%%%%%%%%%
\begin{proof}
The property (\ref{dPdRMnew}) is inherited by $\Lambda$ from solutions of the ALE (\ref{dPALE1}). More precisely, if the matrices $\bitR$ and $\cM$ and their variations $\delta\bitR$ and $\delta\cM$ are transformed as
\begin{align}
\label{bitRnew}
    \bitR
    & \mapsto
    S^{-\rT} \bitR S^{-1},
    \qquad
    \delta \bitR
    \mapsto
    S^{-\rT} (\delta\bitR) S^{-1},\\
\label{cMnew}
    \cM
    & \mapsto
    \cM S^{-1},
    \qquad\ \ \ \,
    \delta \cM
    \mapsto
    (\delta \cM) S^{-1},
\end{align}
then the corresponding matrices $\cA$, $\cB$, $\cP$ and their first variations $\delta \cA$, $\delta\cB$, $\delta\cP$ in (\ref{dcA}), (\ref{dcB}),  (\ref{dPALE1}) are transformed as
\begin{align}
\label{cAnew}
    \cA
    & \mapsto
    S\cA S^{-1},
    \qquad
    \delta\cA
    \mapsto
    S(\delta\cA) S^{-1},\\
\label{cBnew}
    \cB
    & \mapsto
    S\cB,
    \qquad\quad\ \ \,
    \delta\cB
    \mapsto
    S\delta \cB,\\
\label{cPnew}
    \cP
    & \mapsto
    S\cP S^{\rT},
    \qquad\
    \delta\cP
    \mapsto
    S(\delta\cP) S^{\rT}
\end{align}
due to the symplectic property of the  matrix $S\in \Sp(\bTheta)$. A combination of (\ref{bitRnew})--(\ref{cPnew}) with (\ref{dPdRM}) implies that the operator $\Lambda$ in (\ref{dPdRM0}) satisfies
\begin{equation}
\label{dPdRMnew1}
        \Lambda_{S^{-\rT} \bitR S^{-1}, \cM S^{-1}}
        \big(
            S^{-\rT}(\delta \bitR) S^{-1}, \,
            (\delta\cM) S^{-1}
        \big)
        =
        S
        \Lambda_{\bitR, \cM}
        (
            \delta \bitR,
            \delta\cM
        )
        S^{\rT}.
\end{equation}
The relation (\ref{dPdRMnew}) now follows from (\ref{dPdRMnew1}) by applying an appropriate inverse transformation to the variations $\delta\bitR$ and $\delta\cM$ of the independent variables. An alternative (and somewhat less  intuitive) way to establish (\ref{dPdRMnew}) is to use the operator identities
\begin{align}
\label{opid1}
    \bL_{S\cA S^{-1} }
    \[[[S,S^{\rT}\]]]
        & =
    \[[[S,S^{\rT}\]]]
    \bL_{\cA},\\
\label{opid2}
    \[[[\bTheta,S\cP S^{\rT}\]]]
    & =
    \[[[S,S^{\rT}\]]]
    \[[[\bTheta,\cP\]]]
    \[[[S^{\rT},S\]]],\\
\label{opid3}
    \[[[S\cB,\bTheta\]]]
    & =
    \[[[ S,S^{\rT}\]]]
    \[[[\cB,\bTheta\]]]
    \[[[I_m, S\]]],\\
\label{opid4}
    \Gamma_{\cM S^{-1}}
    \[[[I_m,S^{-1}\]]]
    & =
    \[[[S^{-\rT}, S^{-1}\]]]
    \Gamma_{\cM},\\
\label{opid5}
    \bS \[[[S,S^{\rT}\]]]
    & =
    \[[[S,S^{\rT}\]]] \bS,
\end{align}
which hold for any matrix $S\in \Sp(\bTheta)$ in (\ref{diagS}). Indeed, by applying (\ref{opid1})--(\ref{opid5})   to the right-hand sides of (\ref{dPdR}) and (\ref{dPdM}), it follows that the operators $\d_{\bitR} \cP$ and $\d_{\cM}\cP$ are transformed as
\begin{align}
\nonumber
    \d_{\bitR}\cP
    & \mapsto
    4\bL_{S\cA S^{-1}}\bS\[[[\bTheta, S\cP S^{\rT}\]]]\\
\nonumber
    & =
    4\bL_{S\cA S^{-1}}\bS    \[[[S,S^{\rT}\]]]
    \[[[\bTheta,\cP\]]]
    \[[[S^{\rT},S\]]]\\
\nonumber
    & =
    4\bL_{S\cA S^{-1}}\[[[S,S^{\rT}\]]]\bS
    \[[[\bTheta,\cP\]]]
    \[[[S^{\rT},S\]]]\\
\nonumber
    & =
    4\[[[S,S^{\rT}\]]]
    \bL_{\cA}\bS
    \[[[\bTheta,\cP\]]]
    \[[[S^{\rT},S\]]]\\
\label{dPdRnew}
    & =
    \[[[S,S^{\rT}\]]]
    \d_{\bitR}\cP
    \[[[S^{\rT},S\]]],\\
\nonumber
    \d_{\cM}\cP
    & \mapsto
    4\bL_{S\cA S^{-1}}\bS
    (2\[[[\bTheta, S\cP S^{\rT}\]]]\Gamma_{\cM S^{-1}}
    -
    \[[[S \cB,\bTheta\]]])\\
\nonumber
    & =
    4\bL_{S\cA S^{-1}}\bS
    (2
    \[[[S,S^{\rT}\]]]
    \[[[\bTheta,\cP\]]]
    \[[[S^{\rT},S\]]]
    \Gamma_{\cM S^{-1}}
    -
    \[[[ S,S^{\rT}\]]]
    \[[[\cB,\bTheta\]]]
    \[[[I_m, S\]]])\\
\nonumber
    & =
    4\bL_{S\cA S^{-1}}\bS
    \[[[S,S^{\rT}\]]]
    (2
    \[[[\bTheta,\cP\]]]
    \[[[S^{\rT},S\]]]
    \Gamma_{\cM S^{-1}}
    \[[[I_m, S^{-1}\]]]
    -
    \[[[\cB,\bTheta\]]]
    )
    \[[[I_m, S\]]]\\
\nonumber
    & =
    4\bL_{S\cA S^{-1}}
    \[[[S,S^{\rT}\]]]\bS
    (2
    \[[[\bTheta,\cP\]]]
    \Gamma_{\cM}
    -
    \[[[\cB,\bTheta\]]]
    )
    \[[[I_m, S\]]]\\
\nonumber
    & =
    4
    \[[[S,S^{\rT}\]]]
    \bL_{\cA}
    \bS
    (2
    \[[[\bTheta,\cP\]]]
    \Gamma_{\cM}
    -
    \[[[\cB,\bTheta\]]]
    )
    \[[[I_m, S\]]]\\
\label{dPdMnew}
    & =
    \[[[S,S^{\rT}\]]]
    \d_{\cM}\cP
    \[[[I_m, S\]]],
\end{align}
which is an equivalent form of (\ref{dPdRMnew}).
\end{proof}
%%%%%%%%%%%%%%%%%%%%%%%%%%%%%%%%%%%%%%%%%%%%%%%%%%%%%%%%%%%%%%%%%%%%%%%%%%%%%%%%%%%%%%%%%%%%%%%%%%%

Lemma~\ref{lem:dPdRMnew} simplifies the ``recalculation'' of the Frechet derivative $\Lambda$ under the symplectic similarity transformations of a particular realization of the cascaded oscillators. Indeed, the linear operator $\Lambda$ is modified by applying appropriate linear transformations $\fS_{S^{-1}}$ and $\fP_S$ to its domain $\fR\x \mR^{m\x n}$ and range $\mS_n$ as
\begin{equation*}
    \Lambda_{\fS_S(\bitR, \cM)}
    =
    \fP_S\Lambda_{\bitR, \cM}\fS_{S^{-1}},
\end{equation*}
which is merely a concise representation of (\ref{dPdRnew}) and (\ref{dPdMnew}) or their equivalent form (\ref{dPdRMnew}). Here,
the transformations
\begin{align*}
    \fS_S(\bitR, \cM)
    & :=
    (S^{-\rT}\bitR S^{-1}, \cM S^{-1}),\\
    \fP_S
    & :=
    \[[[S, S^{\rT}\]]]
\end{align*}
are parameterized by the symplectic matrix $S \in \Sp(\bTheta)$ in (\ref{diagS}) and describe group homomorphisms (whereby $\fS_{S^{-1}} = \fS_S^{-1}$).

The following theorem, which is concerned with infinitesimal perturbation analysis of the purity functional (\ref{purity}), applies Lemma~\ref{lem:dPdRMnew} in order to obtain the corresponding   transformations for the partial Frechet derivatives of $V$ in (\ref{V}) with respect to the energy and coupling matrices $R_1, \ldots, R_N$ and $M_1, \ldots, M_N$ of the component oscillators:
\begin{equation}
\label{rhokmuk}
    \rho_k
    :=
    \d_{R_k}V,
    \qquad
    \mu_k
    :=
    \d_{M_k}V.
\end{equation}
Up to a constant factor of $-\frac{1}{2}$,  the matrices $\rho_k$ and $\mu_k$  describe the corresponding logarithmic Frechet derivatives of the purity functional.
Before actually computing $\rho_k$ and $\mu_k$ (which are fairly complicated rational functions of the  entries of the energy and coupling matrices), the theorem shows that, under the symplectic similarity transformations, these derivatives   are at most quadratic polynomials of the transformation matrices.

%%%%%%%%%%%%%%%%%%%%%%%%%%%%%%%%%%%%%%%%%%%%%%%%%%%%%%%%%%%%%%%%%%%%%%%%%%%%%%%%%%%%%%%%%%%%%%%%%%%
\begin{theorem}
\label{th:dVdRMnew}
Suppose the matrices $A_1, \ldots, A_N$ in (\ref{Ak}) are Hurwitz.
Then, for any symplectic matrix $S\in \Sp(\bTheta)$ in (\ref{diagS}),  the Frechet derivatives of the functional $V$ with respect to the energy and coupling matrices $R_1, \ldots, R_N$ and $M_1,\ldots, M_N$ in (\ref{rhokmuk})  are transformed as
\begin{equation}
\label{rhokmuknew}
    \rho_k
    \mapsto
    S_k \rho_k S_k^{\rT},
    \qquad
    \mu_k
    \mapsto
    \mu_k S_k^{\rT}
\end{equation}
for all $k=1,\ldots, N$.
\hfill$\square$
\end{theorem}
%%%%%%%%%%%%%%%%%%%%%%%%%%%%%%%%%%%%%%%%%%%%%%%%%%%%%%%%%%%%%%%%%%%%%%%%%%%%%%%%%%%%%%%%%%%%%%%%%%%
\begin{proof}
We will combine Lemma~\ref{lem:dPdRMnew}  with  the following Frechet derivative on the Hilbert space $\mS_n$ (see, for example, \cite{H_2008,HJ_2007}):
\begin{equation}
\label{dlndet}
    \d_{\chi}
    \ln
    \det \chi
    =
    \chi^{-1}.
\end{equation}
This allows the first variation of the functional $V$ in (\ref{V}) (as a composite function $(\bitR, \cM)\mapsto \cP \mapsto \ln\det \cP$) to be computed as
\begin{align*}
\nonumber
    \delta V
    & =
    \bra \cP^{-1}, \delta \cP\ket\\
\nonumber
    & =
    \bra \cP^{-1}, \Lambda(\delta\bitR, \delta \cM)\ket\\
\nonumber
    & =
    \bra \cP^{-1},     \d_{\bitR}\cP(\delta \bitR)
    +
    \d_{\cM}\cP(\delta \cM)\ket\\
    & =
    \bra
        \d_{\bitR}\cP^{\dagger}(\cP^{-1}),
        \delta \bitR
    \ket
    +
    \bra
        \d_{\cM}\cP^{\dagger}(\cP^{-1}),
        \delta \cM
    \ket,
\end{align*}
whence
\begin{align}
\label{dVdR}
    \d_{\bitR} V
     & =
     \diag_{1\< k\< N}(\rho_k)
     =
        \d_{\bitR}\cP^{\dagger}(\cP^{-1}),\\
\label{dVdM}
    \d_{\cM} V
     & =
     \begin{bmatrix}
     \mu_1 & \ldots & \mu_N
     \end{bmatrix}
     =
        \d_{\cM}\cP^{\dagger}(\cP^{-1})
\end{align}
in view of (\ref{rhokmuk}).
From the transformations (\ref{dPdRnew}) and (\ref{dPdMnew}), it follows that the corresponding adjoint operators $\d_{\bitR}\cP^{\dagger}$ and $\d_{\cM}\cP^{\dagger}$ are modified as
\begin{align}
\label{dPdR+new}
    \d_{\bitR}\cP^{\dagger}
    & \mapsto
    \[[[S,S^{\rT}\]]]
    \d_{\bitR}\cP^{\dagger}
    \[[[S^{\rT},S\]]],\\
\label{dPdM+new}
    \d_{\cM}\cP^{\dagger}
    & \mapsto
    \[[[I_m, S^{\rT}\]]]
    \d_{\cM}\cP^{\dagger}
    \[[[S^{\rT},S\]]],
\end{align}
where use is also made of the  relation
\begin{equation}
\label{sanddagger}
    \[[[\alpha, \beta\]]]^{\dagger} = \[[[\alpha^{\rT},\beta^{\rT}\]]]
\end{equation}
for the operator (\ref{sand}) which yields $\[[[S,S^{\rT}\]]]^{\dagger} = \[[[S^{\rT},S\]]]$ and $\[[[I_m,S\]]]^{\dagger} = \[[[I_m,S^{\rT}\]]]$.
A combination of (\ref{dPdR+new}) and (\ref{dPdM+new}) with the transformation (\ref{SPS}) implies that the Frechet derivatives in (\ref{dVdR}) and (\ref{dVdM}) are transformed as
\begin{align}
\nonumber
  \d_{\bitR}V
  & \mapsto
    (\[[[S,S^{\rT}\]]]
    \d_{\bitR}\cP^{\dagger}
    \[[[S^{\rT},S\]]])  ((S\cP S^{\rT})^{-1})\\
\nonumber
  & =
    S
    \d_{\bitR}\cP^{\dagger}
    (\cP^{-1}) S^{\rT}\\
\label{dVdRnew}
  & =
  S\d_{\bitR}V S^{\rT},\\
\nonumber
  \d_{\cM}V
  & \mapsto
    (\[[[I_m, S^{\rT}\]]]
    \d_{\cM}\cP^{\dagger}
    \[[[S^{\rT},S\]]])
      ((S\cP S^{\rT})^{-1})\\
\nonumber
  & =
    \d_{\cM}\cP^{\dagger}(\cP^{-1})S^{\rT}\\
\label{dVdMnew}
  & =
  \d_{\cM}V S^{\rT}.
\end{align}
The relations (\ref{rhokmuknew}) can now be obtained as a block-wise form of (\ref{dVdRnew}) and (\ref{dVdMnew}) due to the block diagonal structure of the matrices $S$ in (\ref{diagS}) and  $\d_{\bitR}V$ in (\ref{dVdR}) and the block partitioning of the matrix $\d_{\cM}V$ in (\ref{dVdM}).
\end{proof}
%%%%%%%%%%%%%%%%%%%%%%%%%%%%%%%%%%%%%%%%%%%%%%%%%%%%%%%%%%%%%%%%%%%%%%%%%%%%%%%%%%%%%%%%%%%%%%%%%%%

Note that the proof of Theorem~\ref{th:dVdRMnew} employs not only Lemma~\ref{lem:dPdRMnew} but also the presence of the inverse matrix $\cP^{-1}$ in (\ref{dVdRnew}) and (\ref{dVdMnew}) as a consequence of (\ref{dlndet}) due to the specific structure of the purity functional, leading to the self-cancellation of the matrix $S$ in
$$
    \[[[S^{\rT},S\]]]
    ((S\cP S^{\rT})^{-1}) = \cP^{-1}.
$$
Similarly to the steady-state covariances themselves, the relations (\ref{rhokmuknew})  allow the infinitesimal perturbation analysis of the purity functional for a particular realization of the quantum system to be easily modified for an equivalent realization.

In view of the symplectic property of the matrices $S_k \in \Sp(\Theta_k)$, the transformation (\ref{rhokmuknew}) also shows that the partial Frechet derivatives $\rho_k$ and $\mu_k$ of the functional $V$ can not be made arbitrarily small. Moreover, since the norm  of such matrices $S_k$ can be arbitrarily large, their inappropriate choice can make the purity functional highly sensitive to the parameters of the cascaded oscillators. In Section~\ref{sec:opt}, a criterion will be specified for the  minimization of this sensitivity, thus putting the choice of the symplectic matrices $S_1, \ldots, S_N$ on a rational footing.

We will now apply Lemma~\ref{lem:dPdRM} and elements of the proof of Theorem~\ref{th:dVdRMnew} to computing the logarithmic Frechet derivatives of the purity functional. The formulation of the following theorem employs the observability Gramian $\cQ:= (\cQ_{jk})_{1\< j,k\< N}$ of the pair $(\cA,\cP^{-1/2})$, which is split into blocks  $\cQ_{jk} \in \mR^{n_j\x n_k}$ and satisfies the  ALE
\begin{equation}
\label{cQALE}
    \cA^{\rT} \cQ + \cQ\cA + \cP^{-1} = 0,
\end{equation}
where, as before, $\cP$ is the controllability Gramian of $(\cA,\cB)$ from (\ref{cPALE}). Also, use will be made of the \textit{Hankelian} $\cH := (\cH_{jk})_{1\< j,k\< N}$ for the triple $(\cA,\cB,\cP^{-1/2})$, consisting of the blocks $\cH_{jk} \in \mR^{n_j\x n_k}$ and defined by
\begin{equation}
\label{cH}
    \cH := \cQ\cP.
\end{equation}
The matrix $\cH$ is diagonalizable and its eigenvalues are the squared Hankel singular values of the triple $(\cA,\cB,\cP^{-1/2})$.

%%%%%%%%%%%%%%%%%%%%%%%%%%%%%%%%%%%%%%%%%%%%%%%%%%%%%%%%%%%%%%%%%%%%%%%%%%%%%%%%%%%%%%%%%%%%%%%%%%%%
\begin{theorem}
\label{th:dVdRM}
Suppose the matrices $A_1, \ldots, A_N$ in (\ref{Ak}) are Hurwitz. Then the Frechet derivatives (\ref{rhokmuk}) can be  computed as
\begin{align}
\label{dVdRk}
    \rho_k
    & =
    -4 \bS(\Theta_k \cH_{kk}),\\
\label{dVdMk}
    \mu_k
    & =
         4
    \left(\cB^{\rT} \cQ_{\bullet k} \Theta_k
    +
    2J
    \left(
    M_k \bA(\Theta_k \cH_{kk})
    +
    \sum_{j=k+1}^N
    M_j \Theta_j \cH_{jk}
    +
    \sum_{j=1}^{k-1}
    M_j \cH_{kj}^{\rT}\Theta_k
    \right)
    \right)
\end{align}
for all $k=1,\ldots, N$.
Here, $\cQ_{\bullet k}:= {\small\begin{bmatrix}\cQ_{1k}\\ \vdots\\ \cQ_{Nk}\end{bmatrix}} \in \mR^{n\x n_k}$ denotes the $k$th block-column of the observability Gramian $\cQ$ from (\ref{cQALE}), and  $\cH_{kk}$ is the $k$th diagonal block of the Hankelian $\cH$ from (\ref{cH}).
\hfill$\square$
\end{theorem}
%%%%%%%%%%%%%%%%%%%%%%%%%%%%%%%%%%%%%%%%%%%%%%%%%%%%%%%%%%%%%%%%%%%%%%%%%%%%%%%%%%%%%%%%%%%%%%%%%%%%
\begin{proof}
The representations of the partial Frechet derivatives in (\ref{dPdR}) and (\ref{dPdM}) allow
their adjoint operators $\d_{\bitR}\cP^{\dagger}$ and $\d_{\cM}\cP^{\dagger}$ to be computed as
\begin{align}
\nonumber
    \d_{\bitR}\cP^{\dagger}
    & =
    4\bitSigma\[[[\bTheta, \cP\]]]^{\dagger}\bS^{\dagger}\bL_{\cA}^{\dagger}\\
\label{dPdR+}
    & =
    -4\bitSigma\[[[\bTheta, \cP\]]]\bS\bL_{\cA^{\rT}},\\
\nonumber
    \d_{\cM}\cP^{\dagger}
    & =
    4
    (2\Gamma^{\dagger}\[[[\bTheta, \cP\]]]^{\dagger}
    -
    \[[[\cB,\bTheta\]]]^{\dagger})
    \bS^{\dagger}
    \bL_{\cA}^{\dagger}\\
\label{dPdM+}
    & =
    4
    (\[[[\cB^{\rT},\bTheta\]]]
    -
    2\Gamma^{\dagger}\[[[\bTheta, \cP\]]]
    )
    \bS
    \bL_{\cA^{\rT}},
\end{align}
where use is made of self-adjointness of the symmetrizer $\bS$ (which is the orthogonal projection onto $\mS_n$) together with the relation (\ref{sanddagger}) which leads to $\bL_{\alpha}^{\dagger} = \bL_{\alpha^{\rT}}$ and is combined with the symmetry of $\cP$ and antisymmetry of $\bTheta$. Here, $\bitSigma$ denotes the orthogonal projection of $\mR^{n\x n}$ onto the subspace $\fR$ of block-diagonal matrices $\bitR$ in (\ref{bitR}), which maps an arbitrary matrix $\zeta:= (\zeta_{jk})_{1\< j,k\< N} \in \mR^{n\x n}$ with blocks $\zeta_{jk} \in \mR^{n_j\x n_k}$ to the matrix
\begin{equation}
\label{bitSigma}
    \bitSigma(\zeta):= \diag_{1\< k\< N}(\bS(\zeta_{kk})).
\end{equation}
Substitution of (\ref{dPdR+}) and (\ref{bitSigma}) into (\ref{dVdR}) leads to
\begin{align}
\nonumber
    \d_{\bitR} V
     & =
     -4\bitSigma(\[[[\bTheta, \cP\]]](\bS(\bL_{\cA^{\rT}}(\cP^{-1}))))\\
\nonumber
     & =
     -4\bitSigma(\[[[\bTheta, \cP\]]](\cQ))\\
\nonumber
     & =
     -4\bitSigma(\bTheta \cH)\\
\label{dVdR1}
     & =
     -4\diag_{1\< k \< N}(\bS(\Theta_k \cH_{kk})),
\end{align}
where use is made of the symmetric matrix $\cQ = \bL_{\cA^{\rT}}(\cP^{-1})$ from (\ref{cQALE}) along with the Hankelian $\cH$ from (\ref{cH}) and the block diagonal structure of the CCR matrix $\bTheta$ in (\ref{Theta}). The representation (\ref{dVdRk}) is now obtained by considering the diagonal blocks of (\ref{dVdR1}). By a similar reasoning, substitution of (\ref{dPdM+}) into (\ref{dVdM}) leads to
\begin{align}
\nonumber
    \d_{\cM} V
     & =
         4
    (\[[[\cB^{\rT},\bTheta\]]]
    -
    2\Gamma^{\dagger}\[[[\bTheta, \cP\]]]
    )
    (\bS
    (\bL_{\cA^{\rT}}(\cP^{-1})))\\
\nonumber
    & =
         4
    (\[[[\cB^{\rT},\bTheta\]]]
    -
    2\Gamma^{\dagger}\[[[\bTheta, \cP\]]]
    )
    (\cQ)    \\
\label{dVdM1}
    & =
         4
    (\cB^{\rT} \cQ \bTheta
    -
    2\Gamma^{\dagger}(\bTheta \cH)).
\end{align}
The structure of the operator $\Gamma$ in (\ref{Gammajk}) implies that its adjoint $\Gamma^{\dagger}$ acts on a matrix  $\alpha:= (\alpha_{jk})_{1\< j,k\< N} \in \mR^{n\x n}$ with blocks $\alpha_{jk} \in \mR^{n_j\x n_k}$ as
\begin{align}
\nonumber
    \Gamma^{\dagger}(\alpha)
    = &
    \sum_{j,k=1}^N
    \Gamma_{jk}^{\dagger}(\alpha_{jk})\\
\nonumber
    = &
    -
    \sum_{k=1}^N
    \[[[JM_k, \begin{bmatrix}0 & I_{n_k} & 0\end{bmatrix}\]]]    (\bA(\alpha_{kk}))\\
\nonumber
    &
    -
    \sum_{N\> j>k\> 1}
    \left(
        \[[[I_m, \begin{bmatrix}0 & I_{n_j} & 0\end{bmatrix}\]]]
        \bT
        \[[[I_{n_j}, M_k^{\rT}J\]]]
        +
        \[[[JM_j, \begin{bmatrix}0 & I_{n_k} & 0\end{bmatrix}\]]]
    \right)(\alpha_{jk})\\
\label{Gammadagger}
    = &
    -J\cM \bitPhi(\alpha),
\end{align}
where $\bT$ denotes the operator of matrix transpose. Here, $\bitPhi$ is a linear operator which maps the matrix $\alpha$ to $\bitPhi(\alpha):= (\bitPhi_{jk}(\alpha))_{1\< j,k\< N} \in \mA_n$ whose blocks are given by
\begin{equation}
\label{bitPhijk}
    \bitPhi_{jk}(\alpha) =
    \left\{
    \begin{matrix}
        \bA(\alpha_{kk}) & {\rm if} & j=k\\
        \alpha_{jk} & {\rm if} & j>k\\
        -\alpha_{kj}^{\rT}& {\rm if} & j<k
    \end{matrix}
    \right.
\end{equation}
and depend only on the block lower triangular part of $\alpha$ (including its diagonal blocks). Also, use has been made of the following representation of the operators in (\ref{Gammajk}):
\begin{equation*}
\label{Gammajksand}
    \Gamma_{jk} =
    \left\{
    \begin{matrix}
        \bA \[[[M_k^{\rT} J, {\small\begin{bmatrix}0 \\ I_{n_k} \\ 0\end{bmatrix}}\]]]
        & {\rm if} & j=k\\
        \[[[I_{n_j}, JM_k\]]]\bT\[[[I_m, {\small\begin{bmatrix}0 \\ I_{n_j} \\ 0\end{bmatrix}}\]]]
        + \[[[M_j^{\rT} J, {\small\begin{bmatrix}0 \\ I_{n_k} \\ 0\end{bmatrix}}\]]] & {\rm if} & j>k\\
        0& {\rm if} & j<k
    \end{matrix}
    \right.,
\end{equation*}
in combination with the antisymmetry of the matrix $J$ and the fact that both the antisymmetrizer $\bA$ (which is an orthogonal projection) and $\bT$ are   self-adjoint operators.\footnote{More precisely, the transpose $\bT_{r,s}$, applied to $(r\x s)$-matrices, satisfies $\bT_{r,s}^{\dagger} = \bT_{s,r}$.} According to (\ref{bitPhijk}), the $k$th block of the image matrix in (\ref{Gammadagger}) is computed as
\begin{equation}
\label{Gammadaggerk}
    (\Gamma^{\dagger}(\alpha))_k
    =
    -J
    \left(
    M_k \bA(\alpha_{kk})
    +
    \sum_{j=k+1}^N
    M_j \alpha_{jk}
    -
    \sum_{j=1}^{k-1}
    M_j \alpha_{kj}^{\rT}
    \right)
\end{equation}
for all $k=1,\ldots, N$.  Therefore, by applying (\ref{Gammadaggerk}) to the matrix $\bTheta\cH$ in (\ref{dVdM1}) and recalling the block diagonal structure and antisymmetry of $\bTheta$ in (\ref{Theta}), it follows that
\begin{align*}
\nonumber
    \mu_k
     & =
         4
    (\cB^{\rT} \cQ_{\bullet k} \Theta_k
    -
    2(\Gamma^{\dagger}(\bTheta \cH))_k)\\
     & =
         4
    \left(\cB^{\rT} \cQ_{\bullet k} \Theta_k
    +
    2J
    \left(
    M_k \bA(\Theta_k \cH_{kk})
    +
    \sum_{j=k+1}^N
    M_j \Theta_j \cH_{jk}
    +
    \sum_{j=1}^{k-1}
    M_j \cH_{kj}^{\rT}\Theta_k
    \right)
    \right),
\end{align*}
which establishes (\ref{dVdMk}).
\end{proof}
%%%%%%%%%%%%%%%%%%%%%%%%%%%%%%%%%%%%%%%%%%%%%%%%%%%%%%%%%%%%%%%%%%%%%%%%%%%%%%%%%%%%%%%%%%%%%%%%%%%%

Due to the block lower triangular structure of the matrix $\cA$ (that is, block upper triangularity of $\cA^{\rT}$),  the ALE (\ref{cQALE}) of Theorem~\ref{th:dVdRM} for the observability Gramian $\cQ$ can also be solved in a recursive fashion, similar to the computation of the controllability Gramian $\cP$ in Lemma~\ref{lem:cP}. However, in contrast to $\cP$,  the blocks of the matrix $\cQ$ satisfy recurrence equations which unfold backwards (that is, from the last oscillator in the cascade towards the first one), thus being reminiscent of Bellman's dynamic programming equations. Note that  here  the role of time is played by the spatial parameter which numbers the oscillators in the cascade.

%%%%%%%%%%%%%%%%%%%%%%%%%%%%%%%%%%%%%%%%%%%%%%%%%%%%%%%%%%%%%%%%%%%%%%%%%%%%%%%%%%%%%%%%%%%%%%%%%%%
\section{Recursive computation of the Frechet derivatives}
\label{sec:recder}
%%%%%%%%%%%%%%%%%%%%%%%%%%%%%%%%%%%%%%%%%%%%%%%%%%%%%%%%%%%%%%%%%%%%%%%%%%%%%%%%%%%%%%%%%%%%%%%%%%%

Although Theorem~\ref{th:dVdRM} provides a complete set of equations for computing the Frechet derivatives $\rho_1, \ldots, \rho_N$ and $\mu_1, \ldots, \mu_N$ of the purity functional in (\ref{rhokmuk}), we will outline its  recursive version which takes into account the cascade structure of the system to a fuller extent.
For what follows,  the matrix pairs
\begin{equation}
\label{Ek}
  E_k:=(R_k, M_k),
  \qquad
  k=1,\ldots, N,
\end{equation}
which specify  the energetics of individual oscillators and take values in the corresponding spaces
$\fE_k:= \mS_{n_k}\x \mR^{m\x n_k}$, are assembled into an $N$-tuple
\begin{equation}
\label{E}
  E:= (E_1, \ldots, E_N).
\end{equation}
Accordingly, $E$ takes values in the Hilbert space
$
    \fE:= \fE_1\x \ldots \x \fE_N
$
with the direct-sum inner product $\bra\cdot, \cdot\ket_{\fE}$ (generated from the Frobenius inner products of matrices). The Frechet derivative  of $V$ (as a composite function of $E$) takes the form
\begin{equation}
\label{dVdE}
    \d_E V
    =
    (\d_{E_1}V, \ldots, \d_{E_N}V)
    =
    \d_E\cP^{\dagger} (\cP^{-1})
\end{equation}
which combines (\ref{rhokmuk}), (\ref{dVdR}) and (\ref{dVdM}). Note that (\ref{dVdE}) is obtained without using the spatial causality mentioned in Section~\ref{sec:reccov}. This property can be taken into account in order to gain a computational advantage for long cascades.
More precisely, the decomposition (\ref{VV}) of the functional $V$ and the fact that $V_1, \ldots, V_{k-1}$ are independent of the matrix pair $E_k$ in (\ref{Ek}) imply that
\begin{align}
\nonumber
    \d_{E_k} V
    & =
    (\rho_k, \mu_k)\\
\nonumber
    & =
    \d_{E_k}
    \left(
        V_{\> k}
        +
        \sum_{j=1}^{k-1}
        V_j
    \right)\\
\label{dVdEk}
    & =
    \d_{E_k}
        V_{\> k}
        =
    (\d_{R_k}
        V_{\> k},
    \d_{M_k}
        V_{\> k})
\end{align}
for all $k=1, \ldots, N$, where
\begin{equation}
\label{Vtail}
    V_{\> k}
    :=
    \sum_{j=k}^N V_j
    =
    \ln\det \Pi_{\> k}
\end{equation}
is the ``tail'' part of $V$. Here, $\Pi_{\> k}$ denotes the Schur complement of the block $\cP_{k-1}$ in the matrix $\cP$ in (\ref{cPN}) computed as
\begin{equation}
\label{Pitail}
    \Pi_{\> k}
    :=
    \cP_{\> k} - T_kQ_{\> k}^{\rT}
     =
    \cov(\Xi_{\>k} \mid \Xi_{k-1}),
\end{equation}
with
\begin{equation}
\label{Tk}
    T_k
    :=
    Q_{\> k}\cP_{k-1}^{-1},
\end{equation}
according to the block partitioning
\begin{equation}
\label{cPtail}
    \cP
    =
    \begin{bmatrix}
    \cP_{k-1}& Q_{\>k}^{\rT}\\
    Q_{\>k} & \cP_{\> k}
    \end{bmatrix}
\end{equation}
which is similar to (\ref{cPk}).
The second equality in (\ref{Pitail}) provides a probabilistic interpretation of $\Pi_{\> k}$ in terms of the auxiliary classical Gaussian random vectors $\xi_1, \ldots, \xi_N$ of Section~\ref{sec:reccov}, with
\begin{equation}
\label{Xi>=k}
    \Xi_{\>k}
    :=
    \begin{bmatrix}
        \xi_k\\
        \vdots\\
        \xi_N
    \end{bmatrix}
    =
    \begin{bmatrix}
        \xi_k\\
        \Xi_{\> k+1}
    \end{bmatrix}.
\end{equation}
Accordingly, the blocks $Q_{\>k}$ and $\cP_{\> k}$ in (\ref{cPtail}) are the covariance matrices for the random vectors $\Xi_{k-1}$ in (\ref{Xik}) and $\Xi_{\>k}$ in (\ref{Xi>=k}):
\begin{align}
\label{Q>=k}
    Q_{\>k}
    & :=
    \begin{bmatrix}
    P_{k1} & \cdots & P_{k,k-1}\\
    \vdots & \ddots & \vdots\\
    P_{N1} & \cdots & P_{N,k-1}
    \end{bmatrix}
    =
    \cov(\Xi_{\> k}, \Xi_{k-1}),\\
\label{P>=k}
    \cP_{\>k}
    & :=
    \begin{bmatrix}
    P_{kk} & \cdots & P_{kN}\\
    \vdots & \ddots & \vdots\\
    P_{Nk} & \cdots & P_{NN}
    \end{bmatrix}
    =
    \cov(\Xi_{\> k}).
\end{align}
Note that the first block-row of the matrix $Q_{\> k}$ in (\ref{Q>=k}) is the matrix $Q_k$ given by (\ref{Qk}).
Similarly to (\ref{xihat}),  the probabilistic meaning of the matrix $T_k$ in (\ref{Tk}) is described in terms of classical conditional expectations as
\begin{equation}
\label{Xihat}
    \wh{\Xi}_{\> k}
    :=
    \bE(\Xi_{\>k}\mid \Xi_{k-1})
    =
    T_k \Xi_{k-1}.
\end{equation}
In particular, in view of (\ref{P>=k}),
\begin{equation}
\label{Pitail1}
    \Pi_{\>1} = \cov(\Xi_{\> 1}) = \cP
\end{equation}
is the unconditional covariance matrix of the vector $\Xi_{\> 1}=\Xi_N$. The subsequent Schur complements $\Pi_{\>2}, \ldots, \Pi_{\> N} = \Pi_N$ are amenable to a recursive computation as follows.

%%%%%%%%%%%%%%%%%%%%%%%%%%%%%%%%%%%%%%%%%%%%%%%%%%%%%%%%%%%%%%%%%%%%%%%%%%%%%%%%%%%%%%%%%%%%%%%%%%%%%%%%%%%
\begin{lemma}
\label{lem:Pinext}
Suppose the matrices $A_1, \ldots, A_N$ in (\ref{Ak}) are Hurwitz. Then
the Schur complement in (\ref{Pitail}) satisfies the recurrence equation
\begin{equation}
\label{Pitailnext}
    \Pi_{\> k}
    =
    \alpha_k - \beta_k \gamma_k^{-1} \beta_k^{\rT}
\end{equation}
for   all $k = 2, \ldots, N$, where the matrices
\begin{align}
\label{alphak}
    \alpha_k
    & :=
    \cov(\Xi_{\> k} \mid \Xi_{k-2}),\\
\label{betak}
    \beta_k
    & :=
    \cov(\Xi_{\> k}, \xi_{k-1} \mid \Xi_{k-2}),\\
\label{gammak}
    \gamma_k
    & :=
    \cov(\xi_{k-1} \mid \Xi_{k-2})
\end{align}
are submatrices of the preceding Schur complement
\begin{equation}
\label{Piprec}
    \Pi_{\> k-1}
    =
    \cov(\Xi_{\> k-1} \mid \Xi_{k-2})
    =
    \begin{bmatrix}
        \gamma_k & \beta_k^{\rT}\\
        \beta_k & \alpha_k
    \end{bmatrix},
\end{equation}
and the initial condition is given by (\ref{Pitail1}). \hfill$\square$
\end{lemma}
%%%%%%%%%%%%%%%%%%%%%%%%%%%%%%%%%%%%%%%%%%%%%%%%%%%%%%%%%%%%%%%%%%%%%%%%%%%%%%%%%%%%%%%%%%%%%%%%%%%%%%%%%%%
\begin{proof}
The relations (\ref{Pitailnext})--(\ref{Piprec}) follow directly from the representation of the matrices in terms of the classical Gaussian random vectors $\xi_1, \ldots, \xi_N$ of Section~\ref{sec:reccov} and are similar to the conditional covariance matrix update in the discrete-time Kalman filter \cite{AM_1979,LS_2001}.
\end{proof}
%%%%%%%%%%%%%%%%%%%%%%%%%%%%%%%%%%%%%%%%%%%%%%%%%%%%%%%%%%%%%%%%%%%%%%%%%%%%%%%%%%%%%%%%%%%%%%%%%%%%%%%%%%%

By a reasoning,  similar to (\ref{dlndet})--(\ref{dVdM}),
the partial Frechet derivative $\d_{E_k}V_{\>k}$  of the tail functional $V_{\> k}$ with respect to the matrix pair $E_k$ in (\ref{Ek}) can be computed in terms of the Schur complement $\Pi_{\>k}$ (see (\ref{dVdEk})--(\ref{Pitail})) as
\begin{equation}
\label{dVtaildEk}
    \d_{E_k}V_{\> k} = \d_{E_k} \Pi_{\>k}^{\dagger}(\Pi_{\>k}^{-1}).
\end{equation}
For computing the Frechet derivative $\d_{E_k} \Pi_{\>k}$ in the following lemma,    which is required for (\ref{dVtaildEk}), we will use (\ref{cPtail}) along with the partitioning
\begin{equation}
\label{cABtail}
    \cA
    =
    \begin{bmatrix}
        \cA_{k-1} & 0\\
        D_k & \cA_{\>k}
    \end{bmatrix},
    \qquad
    \cB
    =
    \begin{bmatrix}
        \cB_{k-1}\\
        \cB_{\>k}
    \end{bmatrix}
\end{equation}
whose blocks are recovered from (\ref{cAk}) and (\ref{cBk}). More precisely,
\begin{equation}
\label{cA>=k}
    \cA_{\>k}
    =
    \begin{bmatrix}
        A_k & 0 & 0 & \cdots& 0\\
        B_{k+1} C_k & A_{k+1} & 0 & \cdots & 0 \\
        \cdots &\cdots  & \cdots &\cdots & \vdots \\
        B_N C_k & B_N C_{k+1} & \cdots & B_N C_{N-1} & A_N
    \end{bmatrix}
\end{equation}
 is a block lower triangular matrix with the diagonal blocks $A_k, \ldots, A_N$, and
\begin{equation}
\label{Dk}
  D_k
  =
    \begin{bmatrix}
    B_kC_1 & \cdots & B_kC_{k-1}\\
    \vdots & \ddots & \vdots\\
    B_NC_1 & \cdots & B_NC_{k-1}
    \end{bmatrix}
    =
  \cB_{\> k}
  \cC_{k-1},
\end{equation}
where use is made of the matrices $\cC_{k-1}$ from (\ref{cCk}) and $\cB_{\>k}$ from (\ref{cB>k}). The following lemma, which is similar to Lemmas~\ref{lem:cP} and \ref{lem:dPdRM}, carries out an infinitesimal perturbation analysis of the Schur complements in (\ref{Pitail}).

%%%%%%%%%%%%%%%%%%%%%%%%%%%%%%%%%%%%%%%%%%%%%%%%%%%%%%%%%%%%%%%%%%%%%%%%%%%%%%%%%%%%%%%%%%%%%%%%%%%
\begin{lemma}
\label{lem:dPidE}
Suppose the matrices $A_1, \ldots, A_N$ in (\ref{Ak}) are Hurwitz. Then, for any $k=1, \ldots, N$,  the Frechet derivatives of the Schur complement $\Pi_{\> k}$ in (\ref{Pitail}) with respect to the energy and coupling matrices of the $k$th oscillator can be computed as
\begin{align}
\label{dPi>=kdRk}
    \d_{R_k}
    \Pi_{\>k}
    = &
    4\bL_{\cA_{\> k}}
    \bS
    \Big(
    \[[[
        \begin{bmatrix}
          \Theta_k \\
          0
        \end{bmatrix},
        \Pi_{k\bullet}
    \]]]
    -
    \[[[ I, \cP_{k-1}^{-1}\cB_{k-1}\wt{\cB}_{\>k}^{\rT}\]]]
    \bL_{\cA_{\>k},\cA_{k-1}}
    \[[[
        \begin{bmatrix}
          \Theta_k \\
          0
        \end{bmatrix},
        Q_k
    \]]]
    \Big),\\
\nonumber
    \d_{M_k}
    \Pi_{\>k}
    = &
    4\bL_{\cA_{\> k}}
    \bS
    \Big(
    \[[[
        \begin{bmatrix}
          \Theta_k & 0 \\
          0        & I
        \end{bmatrix},
        \Pi_{k\bullet}
    \]]]
        \begin{bmatrix}
          2\bA \[[[M_k^{\rT} J, I_{n_k}\]]] \\
          \[[[\cB_{>k} J, I_{n_k}\]]]
        \end{bmatrix}\\
\nonumber
    &+
    \Big(
    \[[[
        \begin{bmatrix}
          \Theta_k \\
          0
        \end{bmatrix},
         \wt{\cB}_{\>k}^{\rT}
    \]]]
    \bT - \[[[ I, \cP_{k-1}^{-1}\cB_{k-1}\wt{\cB}_{\>k}^{\rT}\]]]\\
\label{dPi>=kdMk}
    & \x
    \bL_{\cA_{\>k},\cA_{k-1}}
    \Big(
    \[[[
        \begin{bmatrix}
          \Theta_k & 0 \\
          0        & I
        \end{bmatrix},
        Q_k
    \]]]
        \begin{bmatrix}
          2\bA \[[[M_k^{\rT} J, I_{n_k}\]]] \\
          \[[[\cB_{>k} J, I_{n_k}\]]]
        \end{bmatrix}
        +
    \[[[
        \begin{bmatrix}
          \Theta_k \\
          0
        \end{bmatrix},
         \cC_{k-1}\cP_{k-1}
         +
         \cB_{k-1}^{\rT}
    \]]]
    \bT
    \Big)
    \Big)
    \Big).
\end{align}
Here, $\Pi_{k\bullet} \in \mR^{n_k\x (n_k + \ldots + n_N)}$ denotes the first block-row of the matrix $\Pi_{\> k}$,   and use is made of the operators (\ref{bL})--(\ref{bSbA}) together with the matrix
\begin{equation}
\label{cBtailtilde}
    \wt{\cB}_{\>k}
    :=
  \cB_{\>k} - T_k \cB_{k-1},
\end{equation}
where $T_k$ is given by (\ref{Tk}).
\hfill$\square$
\end{lemma}

%%%%%%%%%%%%%%%%%%%%%%%%%%%%%%%%%%%%%%%%%%%%%%%%%%%%%%%%%%%%%%%%%%%%%%%%%%%%%%%%%%%%%%%%%%%%%%%%%%%
\begin{proof}
Similarly to (\ref{QkALE}) %and (\ref{PkkALE})
of Lemma~\ref{lem:cP}, for any $k=2, \ldots, N$,  the block $Q_{\> k}$ %and $\cP_{\> k}$
of the matrix $\cP$ in (\ref{cPtail}) satisfies the ASE
\begin{align}
\nonumber
   \cA_{\>k} Q_{\>k}  & +Q_{\>k}\cA_{k-1}^{\rT}\\
\label{QkALEtail}
    & +  D_k\cP_{k-1}   + \cB_{\>k}\cB_{k-1}^{\rT} = 0,
\end{align}
where
the matrices $\cB_{\> k}$ and $D_k$ are given by (\ref{cB>k}) and (\ref{Dk}). Similarly to (\ref{dPALE1})
%and (\ref{dPALE2})
in the proof of Lemma~\ref{lem:dPdRM}, the first variation of  the equation (\ref{QkALEtail})
%and (\ref{PkkALEtail})
with respect to $E_k$ takes the form
\begin{align}
\nonumber
   \cA_{\>k} \delta_{E_k}Q_{\>k}  & +(\delta_{E_k}Q_{\>k})\cA_{k-1}^{\rT}\\
\label{varQkALEtail}
    & +  (\delta_{E_k}\cA_{\>k} )Q_{\>k}+(\delta_{E_k}D_k)\cP_{k-1}   + (\delta_{E_k}\cB_{\>k})\cB_{k-1}^{\rT} = 0.
\end{align}
Here, use is made of the spatial causality discussed above, whereby the  matrices $\cA_{k-1}$, $\cB_{k-1}$, $\cP_{k-1}$,  pertaining to the first $k-1$ oscillators in the cascade,  are independent of $E_k$, and hence, their variations vanish. Also, in view of (\ref{cA>=k}), the first variations of the $E_k$-dependent matrices $\cA_{\>k}$, $\cB_{\>k}$, $D_k$ in (\ref{varQkALEtail}) are computed as
\begin{align}
\label{varcAtail}
    \delta_{E_k}\cA_{\> k}
    & =
    \begin{bmatrix}
        \delta_{E_k}A_k & 0 \\
        \cB_{>k}\delta_{E_k}C_k & 0
    \end{bmatrix}
    =
    2
    \begin{bmatrix}
        \Theta_k (\delta R_k + 2\bA(M_k^{\rT} J\delta M_k))& 0 \\
        \cB_{>k} J \delta M_k& 0
    \end{bmatrix},\\
\label{varcBtail}
    \delta_{E_k}\cB_{\> k}
    & =
    \begin{bmatrix}
        \delta_{E_k}B_k\\
        0
    \end{bmatrix}
    =
    2
    \begin{bmatrix}
        \Theta_k \delta M_k^{\rT}\\
        0
    \end{bmatrix},\\
\label{varDk}
    \delta_{E_k}D_k
    & =
    \begin{bmatrix}
        (\delta_{E_k}B_k) \cC_{k-1}\\
        0
    \end{bmatrix}
    =
    2
    \begin{bmatrix}
        \Theta_k \delta M_k^{\rT} \cC_{k-1}\\
        0
    \end{bmatrix},
\end{align}
where the matrices $\cB_{>k}$ and $\cC_{k-1}$ are independent of $E_k$. In terms of the operators (\ref{bL}) and (\ref{sand}), the solution of the ASE (\ref{varQkALEtail}) leads to
\begin{align}
\nonumber
    \d_{E_k}
    Q_{\> k}
    = &
    \bL_{\cA_{\>k},\cA_{k-1}}
    \big(
        \[[[I, Q_{\>k}\]]]
        \d_{E_k}\cA_{\>k}\\
\label{dQtaildEk}
        & +
        \[[[I,\cP_{k-1}\]]]\d_{E_k}D_k
        +
        \[[[I,\cB_{k-1}^{\rT}\]]]\d_{E_k}\cB_{\>k}
    \big).
\end{align}
Here, in view of (\ref{varcAtail})--(\ref{varDk}), the components of the Frechet derivatives $\d_{E_k}\cA_{\>k}$,  $\d_{E_k}\cB_{\>k}$ and $\d_{E_k}D_k$ are given by
\begin{align}
\label{dcAtaildRk}
    \d_{R_k}\cA_{\> k}
    & =
    2
    \[[[
        \begin{bmatrix}
          \Theta_k \\
          0
        \end{bmatrix},
        \begin{bmatrix}
          I_{n_k} & 0
        \end{bmatrix}
    \]]],\\
\label{dcBtaildRk}
    \d_{R_k}\cB_{\> k}
    & = 0,\\
\label{dDkdRk}
    \d_{R_k}D_k
    & =
    0,\\
\label{dcAtaildMk}
    \d_{M_k}
    \cA_{\> k}
    & =
    2
    \[[[
        \begin{bmatrix}
          \Theta_k & 0 \\
          0        & I
        \end{bmatrix},
        \begin{bmatrix}
          I_{n_k} & 0
        \end{bmatrix}
    \]]]
        \begin{bmatrix}
          2\bA \[[[M_k^{\rT} J, I_{n_k}\]]] \\
          \[[[\cB_{>k} J, I_{n_k}\]]]
        \end{bmatrix},\\
\label{dcBtaildMk}
    \d_{M_k}
    \cB_{\> k}
    & =
    2
    \[[[
        \begin{bmatrix}
          \Theta_k \\
          0
        \end{bmatrix},
          I_m
    \]]]
    \bT,\\
\label{dDkdMk}
    \d_{M_k}D_k
    & =
    2
    \[[[
        \begin{bmatrix}
          \Theta_k \\
          0
        \end{bmatrix},
         \cC_{k-1}
    \]]]
    \bT.
\end{align}
Substitution of (\ref{dcAtaildRk})--(\ref{dDkdMk}) into (\ref{dQtaildEk}) leads to
\begin{align}
\label{dQtaildRk}
    \d_{R_k}
    Q_{\> k}
    = &
    2
    \bL_{\cA_{\>k},\cA_{k-1}}
    \[[[
        \begin{bmatrix}
          \Theta_k \\
          0
        \end{bmatrix},
        Q_k
    \]]],\\
\nonumber
    \d_{M_k}
    Q_{\> k}
    = &
    2
    \bL_{\cA_{\>k},\cA_{k-1}}
    \Big(
    \[[[
        \begin{bmatrix}
          \Theta_k & 0 \\
          0        & I
        \end{bmatrix},
        Q_k
    \]]]
        \begin{bmatrix}
          2\bA \[[[M_k^{\rT} J, I_{n_k}\]]] \\
          \[[[\cB_{>k} J, I_{n_k}\]]]
        \end{bmatrix}\\
\label{dQtaildMk}
        & +
    \[[[
        \begin{bmatrix}
          \Theta_k \\
          0
        \end{bmatrix},
         \cC_{k-1}\cP_{k-1}
         +
         \cB_{k-1}^{\rT}
    \]]]
    \bT
    \Big),
\end{align}
where use is made of the relation $ \begin{bmatrix} I_{n_k} & 0 \end{bmatrix}Q_{\> k} = Q_k$ which follows from (\ref{Qk}) and (\ref{Q>=k}). Since $\cP_{k-1}$ does not depend on $E_k$, the corresponding derivatives of the matrix $T_k$ from (\ref{Tk}) are expressed in terms of (\ref{dQtaildRk}) and (\ref{dQtaildMk}) as
\begin{equation}
\label{dTkdEk}
    \d_{E_k}T_k
    =
    \[[[I, \cP_{k-1}^{-1}\]]]
    \d_{E_k}Q_{\> k}.
\end{equation}
Now, in view of the partitioning (\ref{cPtail}), the matrix $\cP$ can be factorized in terms of $T_k$ and  the matrix $\Pi_{\>k}$ from (\ref{Pitail}) as
\begin{equation}
\label{cPfact}
    \cP
    =
    \begin{bmatrix}
      I & 0 \\
      T_k & I
    \end{bmatrix}
    \begin{bmatrix}
      \cP_{k-1} & 0 \\
      0 & \Pi_{\> k}
    \end{bmatrix}
    \begin{bmatrix}
      I & T_k^{\rT} \\
      0 & I
    \end{bmatrix}
\end{equation}
for any $k=1,\ldots, N$. This factorization is closely related to the decomposition
$$
    \Xi_{\>k} = T_k \Xi_{k-1} + \Xi_{\>k} - \wh{\Xi}_{\> k},
$$
where, in view of (\ref{Xihat}), the random vectors $\Xi_{k-1}$ and $\Xi_{\>k} - \wh{\Xi}_{\> k}$  are uncorrelated and have covariance matrices $\cP_{k-1}$ and $\Pi_{\>k}$, respectively. Substitution of (\ref{cPfact}) into the ALE (\ref{cPALE}) allows the latter to be represented in the form
\begin{equation}
\label{cPALET}
    \wt{\cA}_k
    \begin{bmatrix}
      \cP_{k-1} & 0 \\
      0 & \Pi_{\> k}
    \end{bmatrix}
    +
        \begin{bmatrix}
      \cP_{k-1} & 0 \\
      0 & \Pi_{\> k}
    \end{bmatrix}
    \wt{\cA}_k^{\rT} + \wt{\cB}_k \wt{\cB}_k^{\rT} = 0,
\end{equation}
where
\begin{align}
\label{cAktilde}
    \wt{\cA}_k
    & :=
    \begin{bmatrix}
      I & 0 \\
      -T_k & I
    \end{bmatrix}
    \cA
    \begin{bmatrix}
      I & 0 \\
      T_k & I
    \end{bmatrix}
    =
    \begin{bmatrix}
        \cA_{k-1} & 0\\
        \cA_{\>k}T_k - T_k \cA_{k-1} + D_k & \cA_{\>k}
    \end{bmatrix},\\
\label{cBktilde}
    \wt{\cB}_k
    & :=
    \begin{bmatrix}
      I & 0 \\
      -T_k & I
    \end{bmatrix}
    \cB
    =
    \begin{bmatrix}
        \cB_{k-1}\\
        \wt{\cB}_{\>k}
    \end{bmatrix},
\end{align}
with the matrix $\wt{\cB}_{\>k}$ given by (\ref{cBtailtilde}),
and use is made of (\ref{cABtail}). By substituting (\ref{cAktilde}) and (\ref{cBktilde}) into (\ref{cPALET}), it follows that the second diagonal block of this ALE takes the form
\begin{equation}
\label{PitailALE}
    \cA_{\> k} \Pi_{\>k} + \Pi_{\> k} \cA_{\> k}^{\rT} + \wt{\cB}_{\>k}\wt{\cB}_{\>k}^{\rT} = 0.
\end{equation}
By a reasoning, similar to (\ref{dQtaildEk}), the Frechet differentiation of both parts of (\ref{PitailALE}) with respect to $E_k$ leads to
\begin{equation}
\label{dPitaildEk}
    \d_{E_k}
    \Pi_{\>k}
    =
    2\bL_{\cA_{\> k}}
    \bS
    \big(
    \[[[I, \Pi_{\> k}\]]]
    \d_{E_k}\cA_{\> k}
    +
    \[[[I, \wt{\cB}_{\>k}^{\rT}\]]]
    \d_{E_k}\wt{\cB}_{\>k}
    \big).
\end{equation}
Here,
\begin{align}
\nonumber
    \d_{E_k}
    \wt{\cB}_{\>k}
    & =
    \d_{E_k}
    \cB_{\>k} - \[[[ I, \cB_{k-1}\]]]\d_{E_k}T_k\\
\label{dcBtaildEk}
    & =
    \d_{E_k}
    \cB_{\>k} - \[[[ I, \cP_{k-1}^{-1}\cB_{k-1}\]]]
    \d_{E_k}Q_{\> k}
\end{align}
in view of (\ref{dTkdEk}), (\ref{cBtailtilde}) and the fact that the matrix $\cB_{k-1}$ does not depend on $E_k$. Substitution of (\ref{dcAtaildRk})--(\ref{dQtaildMk}) into (\ref{dPitaildEk}) and (\ref{dcBtaildEk}) leads to
\begin{align*}
    \d_{R_k}
    \Pi_{\>k}
    = &
    2\bL_{\cA_{\> k}}
    \bS
    \big(
    \[[[I, \Pi_{\> k}\]]]
    \d_{R_k}\cA_{\> k}
    +
    \[[[I, \wt{\cB}_{\>k}^{\rT}\]]]
    \d_{R_k}\wt{\cB}_{\>k}
    \big)\\
    = &
    4\bL_{\cA_{\> k}}
    \bS
    \Big(
    \[[[
        \begin{bmatrix}
          \Theta_k \\
          0
        \end{bmatrix},
        \Pi_{k\bullet}
    \]]]
    -
    \[[[ I, \cP_{k-1}^{-1}\cB_{k-1}\wt{\cB}_{\>k}^{\rT}\]]]
    \bL_{\cA_{\>k},\cA_{k-1}}
    \[[[
        \begin{bmatrix}
          \Theta_k \\
          0
        \end{bmatrix},
        Q_k
    \]]]
    \Big),\\
    \d_{M_k}
    \Pi_{\>k}
    = &
    2\bL_{\cA_{\> k}}
    \bS
    \big(
    \[[[I, \Pi_{\> k}\]]]
    \d_{M_k}\cA_{\> k}
    +
    \[[[I, \wt{\cB}_{\>k}^{\rT}\]]]
    \d_{M_k}\wt{\cB}_{\>k}
    \big)\\
    = &
    4\bL_{\cA_{\> k}}
    \bS
    \Big(
    \[[[
        \begin{bmatrix}
          \Theta_k & 0 \\
          0        & I
        \end{bmatrix},
        \Pi_{k\bullet}
    \]]]
        \begin{bmatrix}
          2\bA \[[[M_k^{\rT} J, I_{n_k}\]]] \\
          \[[[\cB_{>k} J, I_{n_k}\]]]
        \end{bmatrix}\\
    &+
    \Big(
    \[[[
        \begin{bmatrix}
          \Theta_k \\
          0
        \end{bmatrix},
         \wt{\cB}_{\>k}^{\rT}
    \]]]
    \bT - \[[[ I, \cP_{k-1}^{-1}\cB_{k-1}\wt{\cB}_{\>k}^{\rT}\]]]\\
    & \x
    \bL_{\cA_{\>k},\cA_{k-1}}
    \Big(
    \[[[
        \begin{bmatrix}
          \Theta_k & 0 \\
          0        & I
        \end{bmatrix},
        Q_k
    \]]]
        \begin{bmatrix}
          2\bA \[[[M_k^{\rT} J, I_{n_k}\]]] \\
          \[[[\cB_{>k} J, I_{n_k}\]]]
        \end{bmatrix}
        +
    \[[[
        \begin{bmatrix}
          \Theta_k \\
          0
        \end{bmatrix},
         \cC_{k-1}\cP_{k-1}
         +
         \cB_{k-1}^{\rT}
    \]]]
    \bT
    \Big)
    \Big)
    \Big),
\end{align*}
which establishes (\ref{dPi>=kdRk}) and (\ref{dPi>=kdMk}) in view
of the representation $\Pi_{k\bullet}= \begin{bmatrix}I_{n_k} & 0\end{bmatrix} \Pi_{\> k}$ for the first block-row of the matrix $\Pi_{\> k}$.
\end{proof}
%%%%%%%%%%%%%%%%%%%%%%%%%%%%%%%%%%%%%%%%%%%%%%%%%%%%%%%%%%%%%%%%%%%%%%%%%%%%%%%%%%%%%%%%%%%%%%%%%%%

Application of Lemma~\ref{lem:dPidE} to (\ref{dVtaildEk}) is carried out similarly to the proof of Theorem~\ref{th:dVdRM} by using the relations
\begin{equation}
\label{dV>=kdEk}
    \rho_k
    =
    \d_{R_k} \Pi_{\>k}^{\dagger}(\Pi_{\>k}^{-1}),
    \qquad
    \mu_k
    =
    \d_{M_k} \Pi_{\>k}^{\dagger}(\Pi_{\>k}^{-1}).
\end{equation}
Here, the adjoint operators can be found from (\ref{dPi>=kdRk}) and (\ref{dPi>=kdMk}) (the resulting expressions are cumbersome and are omitted). Their evaluation at $\Pi_{\>k}^{-1}$ in (\ref{dV>=kdEk}) involves the observability Gramian $\cQ_{\>k}$ of the pair $(\cA_{\>k}, \Pi_{\>k}^{-1/2})$ satisfying the ALE
\begin{equation}
\label{cQ>=kALE}
    \cA_{\>k}^{\rT} \cQ_{\>k} + \cQ_{\>k}\cA_{\>k} + \Pi_{\>k}^{-1} = 0,
\end{equation}
which reduces to (\ref{cQALE}) in the case $k=1$. Note that the dimension of (\ref{cQ>=kALE}) decreases with $k$ and becomes relatively small  at the end of the cascade $k=N$.

%%%%%%%%%%%%%%%%%%%%%%%%%%%%%%%%%%%%%%%%%%%%%%%%%%%%%%%%%%%%%%%%%%%%%%%%%%%%%%%%%%%%%%%%%%%%%%%%%%%%%%%%
\section{Minimization of the mean square sensitivity index}\label{sec:opt}
%%%%%%%%%%%%%%%%%%%%%%%%%%%%%%%%%%%%%%%%%%%%%%%%%%%%%%%%%%%%%%%%%%%%%%%%%%%%%%%%%%%%%%%%%%%%%%%%%%%%%%%%

In what follows, we will use the vectorization of matrices \cite{M_1988,SIG_1998} in regard to the matrix-valued parameters of the cascaded oscillators and the corresponding Frechet derivatives.  The full (rather than half-) vectorization is denoted by $\vec{(\cdot)}$ or $\vect(\cdot)$ interchangeably and is applicable to symmetric matrices and assemblages of matrices. Accordingly, the $N$-tuple  $E$ of the pairs $E_k$ in (\ref{Ek}) and (\ref{E}) is vectorized as
\begin{equation}
\label{EEvec}
  \vec{E}
  :=
  \begin{bmatrix}
    \vec{E}_1\\
    \vdots\\
    \vec{E}_N
  \end{bmatrix},
  \qquad
  \vec{E}_k
  :=
  \begin{bmatrix}
    \vec{R}_k\\
    \vec{M}_k
  \end{bmatrix},
\end{equation}
where $\vec{R}_k \in \mR^{n_k^2}$, so that $\vec{E}_k\in \mR^{n_k(n_k+m)}$. In terms of the vectorization,  the linear operator in (\ref{bL}) is represented as
\begin{align}
\label{bLvec}
    \vect(\bL_{\alpha,\beta}(\gamma))
    & = -(\beta \op \alpha)^{-1}\vec{\gamma},%\\
\end{align}
where $\alpha\op \beta:= \alpha \ox I + I\ox \beta$ is the Kronecker sum of matrices. 
Furthermore, we denote by $\Ups_r$ the duplication $r^2 \x \frac{1}{2}r(r+1)$-matrix \cite{M_1988,SIG_1998} which relates the full vectorization $\vec{M}$ of a symmetric matrix $M$ of order $r$ to its half-vectorization $\vech(M)$ (that is,
the column-wise vectorization of its triangular part below and
including the main diagonal) by
\begin{equation}
\label{Ups}
    \vec{M} = \Ups_r \vech(M),
    \qquad
    M \in \mS_r.
\end{equation}
For example,
$$
    \Ups_2
    =
    \begin{bmatrix}
     1 &    0  &   0\\
     0 &    1  &   0\\
     0 &    1  &   0\\
     0 &    0  &   1
     \end{bmatrix}.
$$
The presence of linear degeneracies in the full vectorization $\vec{E}_k$ in (\ref{EEvec}) (coming from the symmetry of the energy matrices $R_k \in \mS_{n_k}$) can be taken into account by considering a smaller number of independent variables as
\begin{equation}
\label{Evech}
  \vec{E}_k
  =
  \mho_k
  e_k,
  \qquad
  \mho_k
  :=
  \begin{bmatrix}
    \Ups_{n_k} & 0\\
    0 & I_{mn_k}
  \end{bmatrix},
    \qquad
  e_k
  :=
  \begin{bmatrix}
    \vech(R_k)\\
    \vec{M}_k
  \end{bmatrix}.
\end{equation}
Here, the matrix $\mho_k$ is of full column rank, and,  in contrast to $\vec{E}_k$,  the entries of the vector $e_k\in \mR^{n_k(\frac{1}{2}(n_k+1) + m)}$ are linearly independent.

Now, suppose the energy and coupling matrices of the component oscillators are subject to infinitesimal perturbations (or implementation errors) $\delta R_k$ and $\delta M_k$, which, in accordance with (\ref{EEvec}) and (\ref{Evech}),  are represented in the vectorized form as $\delta \vec{E}_k$ or $\delta e_k$.  The corresponding (linearized)  variation of the functional $V$ is
\begin{align}
\nonumber
    \delta V
    & =
    \sum_{k=1}^N
    \bra
        \d_{E_k} V,
        \delta E_k
    \ket\\
\nonumber
      & = \sum_{k=1}^N \d_{\vec{E}_k} V^{\rT}\delta \vec{E}_k\\
\label{dV}
      & =
      \sum_{k=1}^N
      \d_{\vec{E}_k} V^{\rT}\mho_k \delta e_k,
\end{align}
where
\begin{equation}
\label{dVdek}
    \d_{\vec{E}_k} V
    =
    \begin{bmatrix}
        \vec{\rho}_k\\
        \vec{\mu}_k
    \end{bmatrix}
    =
    \mho_k
    d_k,
    \qquad
    d_k
    :=
    \begin{bmatrix}
        \vech(\rho_k)\\
        \vec{\mu}_k
    \end{bmatrix}        ,
\end{equation}
and use is made of the vectorizations of the partial Frechet derivatives $\rho_k$ and $\mu_k$ from (\ref{rhokmuk}).
If the modelling errors $\delta e_1, \ldots, \delta e_N$ (which encode the vectors $\delta \vec{E}_1, \ldots,\delta\vec{E}_N$) are of a classical random nature and are statistically uncorrelated with each other for different subsystems, they can be regarded as zero mean random vectors with covariance matrices
\begin{equation}
\label{EEcov}
    \cov (\delta e_j, \delta e_k)
    =
    \left\{
    \begin{matrix}
     \eps \Sigma_k & {\rm if}\ j=k\\
     0 & {\rm otherwise}
    \end{matrix}
    \right.
\end{equation}
for all $j,k=1, \ldots, N$. Here,
$\eps>0$ is a small scaling parameter, and   $\Sigma_1, \ldots, \Sigma_N$ are appropriately dimensioned positive definite matrices (which specify the shape of the scattering ellipsoids). Note that (\ref{EEcov}) completely specifies the covariance operators $\bE(\delta E_j \ox \delta E_k)$ for the random matrices $\delta E_1, \ldots, \delta E_N$,  whereby the variance of the quantity $\delta V$ in (\ref{dV}) takes the form
\begin{align}
\nonumber
    \bE ((\delta V)^2)
    & =
    \sum_{j,k=1}^N
    \Bra
        \bE(\delta E_j \ox \delta E_k),\,
        \d_{E_j} V
        \ox
        \d_{E_k} V
    \Ket\\
\nonumber
    & =
    \sum_{j,k=1}^N
    \d_{\vec{E}_j} V^{\rT}
    \mho_j
    \cov (\delta e_j, \delta e_k)
    \mho_k^{\rT}
     \d_{\vec{E}_k} V\\
\nonumber
    & =
    \eps
    \sum_{k=1}^N
    \d_{\vec{E}_k} V^{\rT}
    \mho_k
    \Sigma_k
    \mho_k^{\rT}
    \d_{\vec{E}_k} V\\
\label{EdV2}
     & =
    \eps Z
\end{align}
and is proportional (with a constant coefficient $\eps$) to the functional
\begin{equation}
\label{Z}
  Z := \sum_{k=1}^N Z_k,
  \qquad
  Z_k:= \|\mho_k^{\rT}\d_{\vec{E}_k} V_{\> k}\|_{\Sigma_k}^2.
\end{equation}
Here, use is also made of the relations (\ref{dVdEk}) and (\ref{Vtail}).
Note that $Z$ in (\ref{Z}) is a quadratic function of the gradient $\d_{\vec{E}}V$  and, in view of (\ref{EdV2}), can be regarded as a mean square measure for the sensitivity of the purity functional to the modelling errors in the energy and coupling matrices.

As mentioned in Section~\ref{sec:diff}, the sensitivity index $Z$ is not invariant under the symplectic similarity transformations (\ref{SXk}) of the component oscillators. Moreover, by applying (\ref{rhokmuknew}) of Theorem~\ref{th:dVdRMnew}, it follows that the gradient vectors in (\ref{dVdek}) are transformed as
\begin{align}
\nonumber
    \d_{\vec{E}_k} V
    & \mapsto
    \begin{bmatrix}
        \vect(S_k \rho_kS_k^{\rT})\\
        \vect(\mu_kS_k^{\rT})
    \end{bmatrix}\\
\nonumber
    & =
    \begin{bmatrix}
        S_k\ox S_k & 0\\
        0 & S_k\ox I_{mn_k}
    \end{bmatrix}
    \begin{bmatrix}
        \vec{\rho}_k\\
        \vec{\mu}_k
    \end{bmatrix}    \\
\label{dVdeknew}
    & =
    \begin{bmatrix}
        S_k\ox S_k & 0\\
        0 & S_k\ox I_{mn_k}
    \end{bmatrix}
    \mho_k
    d_k,
\end{align}
where the  partial Frechet derivatives $\rho_1, \ldots, \rho_N$ and $\mu_1, \ldots, \mu_N$ are evaluated at the original realization of the system as described in Theorem~\ref{th:dVdRM} and Section~\ref{sec:recder}.
Substitution of (\ref{dVdeknew}) into (\ref{Z})  leads to the following dependence of the quantities $Z_k$ in (\ref{Z}) on the transformation matrices $S_k \in \Sp(\Theta_k)$:
\begin{equation}
\label{ZSk}
      Z_k
      =
      \left\|
        \mho_k^{\rT}
        \begin{bmatrix}
            S_k\ox S_k & 0\\
            0 & S_k\ox I_{mn_k}
        \end{bmatrix}
        \mho_k
        d_k
      \right\|_{\Sigma_k}^2
      =:
      \Phi_k(S_k),
      \qquad
      k=1,\ldots, N,
\end{equation}
 Each of the functions $\Phi_k$ in (\ref{ZSk}) is a quartic polynomial of the entries of $S_k$, with the coefficients of this polynomial depending on the vector $d_k$ and the  matrix $\Sigma_k$.

In the context of robust Gaussian state generation, the above discussion suggests that the symplectic transformation matrices $S_1, \ldots, S_N$ (for a given cascade realization) can be found so as to minimize   the mean square sensitivity index $Z$ in (\ref{Z}). Since $Z$ has an additive structure, its minimization can be split into $N$ independent  problems of minimizing the functions $\Phi_1,\ldots, \Phi_N$  in (\ref{ZSk}) over the corresponding symplectic groups:
\begin{equation}
\label{Zmin}
  \inf_{S_k\in \Sp(\Theta_k),\ k=1,\ldots, N}\
  Z
  =
  \sum_{k=1}^N\
  \inf_{S_k\in \Sp(\Theta_k)}
  \Phi_k(S_k).
\end{equation}
The (optimal or suboptimal)  solutions $S_1, \ldots, S_N$ of
these problems can be used for balancing the original realization of the cascaded oscillators. The resulting equivalent realization of the composite quantum system will have a decreased mean square sensitivity of the purity functional with respect to  small random perturbations. This criterion differs from the previously known balanced realizations of classical linear systems  and their quantum analogues \cite{N_2013} which equate the controllability and observability Gramians of the system (in the spirit of the Kalman duality principle). Another related (entropy theoretic) optimality criterion  will be discussed in Appendix~\ref{sec:relent}.

We will now assume that there is prior information on parametric uncertainties which ensures the following upper bounds on the covariance matrices $\Sigma_k$ in (\ref{EEcov}):
\begin{equation}
\label{Sigmakup}
    \Sigma_k
    \preccurlyeq
    \begin{bmatrix}
        a_k I_{n_k(n_k+1)/2} & 0 \\
        0 & b_k I_{mn_k}
    \end{bmatrix},
\end{equation}
where $a_k$ and $b_k$ are positive scalars for all $k=1,\ldots, N$. In particular, (\ref{Sigmakup}) holds if the uncertainties $\delta R_k$ and $\delta M_k$ in the energy and coupling matrices are uncorrelated and satisfy
\begin{equation}
\label{covup}
    \cov(\vech(\delta R_k)) \preccurlyeq \eps a_k I_{n_k(n_k+1)/2},
    \qquad
    \cov(\delta \vec{M}_k) \preccurlyeq \eps b_k I_{mn_k},
\end{equation}
with $\eps$ the small scale factor as before.
Since the operator norm of the duplication matrix $\Ups_r$ in (\ref{Ups}) does not exceed $\sqrt{2}$ for any $r$, then  (\ref{Sigmakup}) implies that
\begin{equation}
\label{Sigmamhokup}
    \mho_k\Sigma_k\mho_k^{\rT}
    \preccurlyeq
    \begin{bmatrix}
        a_k \Ups_{n_k}\Ups_{n_k}^{\rT} & 0 \\
        0 & b_k I_{mn_k}
    \end{bmatrix}
    \preccurlyeq
    \begin{bmatrix}
        2a_k I_{n_k^2} & 0 \\
        0 & b_k I_{mn_k}
    \end{bmatrix},
\end{equation}
where the matrix $\mho_k$ is given by (\ref{Evech}). The last inequality in (\ref{Sigmamhokup}) leads to the following upper bound on the function $\Phi_k$ in (\ref{ZSk}):
\begin{align}
\nonumber
    \Phi_k(S_k)
    & =
    \left\|
        \begin{bmatrix}
        \vect(S_k \rho_kS_k^{\rT})\\
        \vect(\mu_kS_k^{\rT})
    \end{bmatrix}
    \right\|_{\mho_k\Sigma_k \mho_k^{\rT}}^2\\
\nonumber
    & \<
    2a_k |\vect(S_k \rho_kS_k^{\rT})|^2 + b_k |\vect(\mu_kS_k^{\rT})|^2\\
\label{Psikup}
    & \<
    2a_k \|S_k \rho_kS_k^{\rT}\|^2 + b_k\|\mu_kS_k^{\rT}\|^2
     =:
    \Psi_k(S_k),
\end{align}
where we have used the isometric property $|\vec{M}| = \|M\|$  of the full vectorization for any matrix $M$. Therefore, a suboptimal solution to the problem of minimizing the function $\Phi_k$ in (\ref{Zmin}) can be obtained by replacing it with the right-hand side of (\ref{Psikup}):
\begin{equation}
\label{Psikupmin}
    \Psi_k(S_k)
    \longrightarrow \min,
    \qquad
    S_k \in \Sp(\Theta_k).
\end{equation}
This suboptimal approach can be applied to all the oscillators in the cascade by solving (\ref{Psikupmin}) independently for every $k=1, \ldots, N$. In the case $n_k=2$, when $O_k$ is a one-mode oscillator and the CCR matrix $\Theta_k$ in (\ref{XXk}) coincides with the matrix $\bJ$ from (\ref{bJ}) up to a scalar factor, the replacement problem reduces effectively to a quadratic optimization problem and
its solution is provided below.

%%%%%%%%%%%%%%%%%%%%%%%%%%%%%%%%%%%%%%%%%%%%%%%%%%%%%%%%%%%%%%%%%%%%%%%%%%%%%%%%%%%%%%%%%%%%%%%%%%%%
\begin{theorem}
\label{th:Psimin}
Suppose $\rho\in \mS_2$ and $\mu\in\mR^{m\x 2}$ are given matrices, the second of which is of full column rank:
\begin{equation}
\label{taupos}
        \tau:= \mu^{\rT}\mu \succ 0 .
\end{equation}
Then any locally optimal solution of the minimization problem
\begin{equation}
\label{Psimin}
  \Psi(\sigma)
  :=
  \frac{1}{2}
  \|\sigma\rho\sigma^{\rT}\|^2
  +
  \|\mu \sigma^{\rT}\|^2
  \longrightarrow \min,
  \qquad
  \sigma \in \Sp(\bJ),
\end{equation}
can be represented as
\begin{equation}
\label{sigopt}
  \sigma = \sR(\phi) \sqrt{U}.
\end{equation}
Here,
\begin{equation}
\label{rot}
    \sR(\phi):=
    \begin{bmatrix}
        \cos \phi & -\sin \phi\\
        \sin \phi & \cos\phi
    \end{bmatrix}
\end{equation}
is the matrix of rotation by an arbitrary angle $\phi$, and $U\in \mS_2$ is a positive definite matrix which is computed as
\begin{equation}
\label{U}
  U
  :=
  \tau^{-1/2}
  f_{\lambda}(\tau^{-1/2} \rho \tau^{-1/2})
  \tau^{-1/2},
\end{equation}
where the function
\begin{equation}
\label{flam}
  f_{\lambda}(z)
  :=
  \frac{\lambda}{1 + \sqrt{1 + 2\lambda z^2}}
\end{equation}
of a scalar variable $z$ is applied to the real symmetric matrix $\tau^{-1/2} \rho \tau^{-1/2}$ and depends on a parameter $\lambda>0$. The latter is found as a unique solution of the equation
\begin{equation}
\label{lamsol}
  h(\lambda):= f_{\lambda}(r_1)f_{\lambda}(r_2) = \det\tau,
\end{equation}
where $r_1$ and $r_2$ are the eigenvalues of $\rho \tau^{-1}$.
\hfill$\square$
\end{theorem}
%%%%%%%%%%%%%%%%%%%%%%%%%%%%%%%%%%%%%%%%%%%%%%%%%%%%%%%%%%%%%%%%%%%%%%%%%%%%%%%%%%%%%%%%%%%%%%%%%%%%
\begin{proof}
In view of the symmetry of the matrices $\rho$ and $\tau$ in (\ref{taupos}), the quartic polynomial $\Psi$ in (\ref{Psimin}) can be represented as
\begin{align}
\nonumber
    \Psi(\sigma)
    & =
    \frac{1}{2}
    \bra
        \sigma\rho\sigma^{\rT},
        \sigma\rho\sigma^{\rT}
    \ket
  +
  \bra
    \mu \sigma^{\rT},
    \mu \sigma^{\rT}
  \ket\\
\nonumber
    & =
    \frac{1}{2}
    \bra
        \rho,
        U \rho U
    \ket
    +
    \bra
        \tau,
        U
    \ket\\
\label{PsiU}
    & =:
        \wt{\Psi}(U),
\end{align}
which is a quadratic function of a real positive semi-definite symmetric matrix
\begin{equation}
\label{Usig}
    U:= \sigma^{\rT}\sigma.
\end{equation}
Due to the identity $\sigma\bJ \sigma^{\rT} = \det\sigma \bJ$, which holds for any matrix $\sigma\in \mR^{2\x 2}$, the symplectic property $\sigma \in \Sp(\bJ)$ is equivalent to $\det\sigma = 1$, and hence, the problem  (\ref{Psimin}) reduces to a constrained  quadratic minimization problem
\begin{equation}
\label{PsiUmin}
  \wt{\Psi}(U)
  \longrightarrow
  \min,
  \qquad
  U = U^{\rT}\succ 0,
  \quad
  \ln\det U = 0.
\end{equation}
By endowing the constraint $\ln\det U = 0$ with a Lagrange multiplier $\lambda\in \mR$, the Lagrange function for (\ref{PsiUmin}) takes the form
\begin{equation}
\label{cL}
    \cL_{\lambda}(U)
    =
    \wt{\Psi}(U) - \frac{\lambda}{2}\ln\det U,
\end{equation}
where the $\frac{1}{2}$ factor is introduced for convenience.
Any locally optimal  solution $U$ of the problem (\ref{PsiUmin})  is necessarily a stationary point of the Lagrange function $\cL_{\lambda}(U)$ for some $\lambda$, and hence, the corresponding (unconstrained) Frechet derivative of (\ref{cL}) vanishes:
\begin{align}
\nonumber
    \d_U
    \cL_{\lambda}(U)
    & =
    \d_U\wt{\Psi}(U)
    -
    \frac{\lambda}{2}\d_U \ln\det U\\
\label{dLdU0}
    & =
    \rho U \rho + \tau
    -
    \frac{\lambda}{2}
    U^{-1}
    =0,
\end{align}
where use is made of (\ref{PsiU}) and (\ref{dlndet}). Since the matrices $\tau$ and $U$ in (\ref{taupos}) and (\ref{Usig}) are positive definite, and  $\rho$ is symmetric, then (\ref{dLdU0}) implies that $\lambda>0$ and
\begin{align*}
    U
    & =
    \frac{\lambda}{2}
    (\rho U \rho + \tau)^{-1}\\
    & =
    \frac{\lambda}{2}
    \tau^{-1/2}
    (\tau^{-1/2}\rho U \rho\tau^{-1/2} + I_2)^{-1} \tau^{-1/2}.
\end{align*}
By left and right multiplying both sides of this equality by $\sqrt{\tau}$, it follows that an appropriate transformation of the matrix $U$ satisfies
\begin{equation}
\label{T}
    T
     :=
    \sqrt{\tau}
    U
    \sqrt{\tau}
    =
    \frac{\lambda}{2}
    (\wt{\rho}T \wt{\rho}+ I_2)^{-1},
\end{equation}
where
\begin{equation}
\label{rhotilde}
  \wt{\rho}
  :=
  \tau^{-1/2}\rho \tau^{-1/2}
\end{equation}
is an auxiliary  real symmetric matrix.   A reasoning, similar to that for solving a class of algebraic Riccati equations in \cite[Lemma~10.1]{VP_2015}, shows that (\ref{T}) has a unique solution $T\succ 0$ and this solution is representable as
\begin{equation}
\label{Tfrho}
    T = f_{\lambda}(\wt{\rho}).
\end{equation}
Here, for any given $\lambda >0$, a function  $f_{\lambda}(z)$ of a real variable $z$ is evaluated at the real symmetric matrix $\wt{\rho}$ (see, for example,  \cite{H_2008}), whereby the matrix $T$ in (\ref{Tfrho}) commutes with $\wt{\rho}$. This allows  $f_{\lambda}(z)$ to be found as the unique positive solution of the scalar equation
$$
    f_{\lambda}(z) = \frac{\lambda}{2(1 + z^2 f_{\lambda}(z))}
$$
or, equivalently, as the unique positive root $w$ of the quadratic polynomial
$
    z^2 w^2  + w - \frac{\lambda}{2}
$, which leads to (\ref{flam}). Now, since $\det\sigma = 1$, then (\ref{Usig}), (\ref{T}) and (\ref{Tfrho}) imply that the Lagrange multiplier $\lambda$ must satisfy
\begin{equation}
\label{detT}
    \det T
    =
    \det \tau
    =
    \det f_{\lambda}(\wt{\rho})
    =
    f_{\lambda}(r_1)
    f_{\lambda}(r_2),
\end{equation}
where $r_1$ and $r_2$ are the eigenvalues of the matrix $\wt{\rho}$ in (\ref{rhotilde}). The right-hand side of (\ref{detT}) is a strictly increasing continuous function of $\lambda>0$, which varies from 0 to $+\infty$, whereby (\ref{lamsol}) has a unique solution $\lambda$. The corresponding  matrix $U$ in (\ref{T}) is recovered from $T$ in (\ref{Tfrho}) as described by (\ref{U}), and the symplectic matrix $\sigma$ is found from (\ref{Usig}) according to (\ref{sigopt}) and (\ref{rot}). It now remains to note that the matrix $\wt{\rho} = \tau^{-1/2}\rho\tau^{-1}\sqrt{\tau}$ in (\ref{rhotilde}) is obtained through a similarity transformation from and is, therefore,  isospectral to $\rho\tau^{-1}$.
\end{proof}
%%%%%%%%%%%%%%%%%%%%%%%%%%%%%%%%%%%%%%%%%%%%%%%%%%%%%%%%%%%%%%%%%%%%%%%%%%%%%%%%%%%%%%%%%%%%%%%%%%%%

In addition to its monotonicity, the left-hand side $h(\lambda)$ of (\ref{lamsol}) is a smooth convex function of $\lambda$ (see, for example,  Fig.~\ref{fig:lamsol}), which makes this equation  amenable to effective numerical solution by  a standard root-finding algorithm. In particular,  application of the Newton-Raphson method with the initial condition
\begin{equation}
\label{lam0}
    \lambda_0
    :=
    2\sqrt{\det\tau}
\end{equation}
produces a monotonically decreasing sequence $\lambda_1, \lambda_2,\ldots$ as $\lambda_{k+1}:= \lambda_k + \frac{\det\tau - h(\lambda_k)}{h'(\lambda_k)}$  (so that $\lambda_0 \< \lambda_1 \> \lambda_2 \> \ldots$) which converges at a quadratic rate to the solution of (\ref{lamsol}),
with the derivative of the function $h$ computed as
$$
    h'(\lambda)
    =
    \left(
        \frac{2}{\lambda}
        -
        \frac{r_1^2}{\left(1+\sqrt{1+2r_1^2\lambda}\right)\sqrt{1+2r_1^2\lambda}}
        -
        \frac{r_2^2}{\left(1+\sqrt{1+2r_2^2\lambda}\right)\sqrt{1+2r_2^2\lambda}}
    \right)
    h(\lambda)
$$
in view of (\ref{flam}). The particular choice of the initial condition in (\ref{lam0}) yields $h(\lambda_0)\< \frac{1}{4}\lambda_0^2 = \det\tau$, whereby $\lambda_0$ is a guaranteed lower bound for the solution.

%%%%%%%%%%%%%%%%%%%%%%%%%%%%%%%%%%%%%%%%%%%%%%%%%%%%%%%%%%%%%%%%%%%%%%%%%%%%%%%%%%%%%%%%%%%%%%%%%%%%
\begin{figure}[thpb]
      \centering
      \includegraphics[width=90mm]{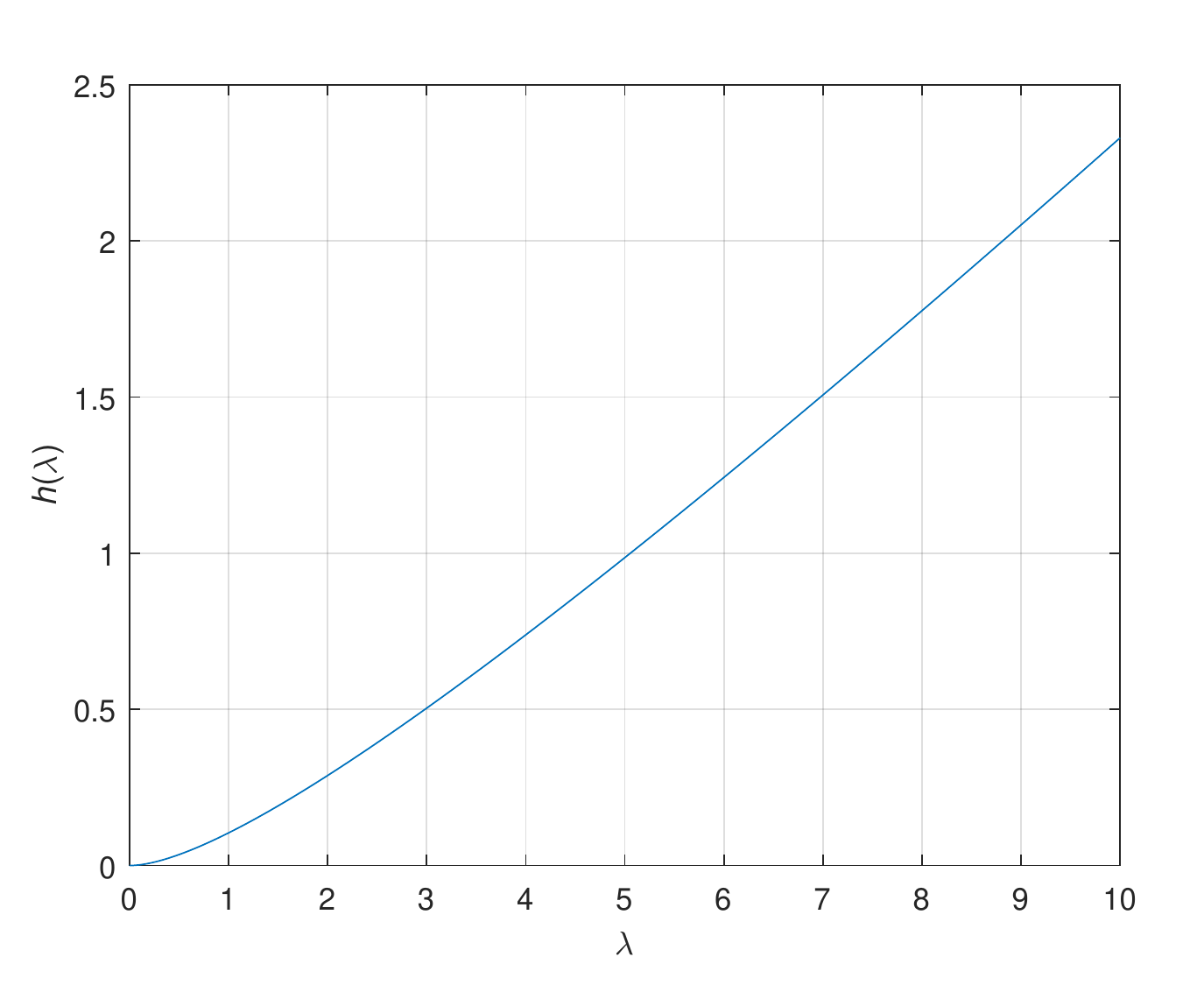}
      \caption{The left-hand side of the equation (\ref{lamsol}) for finding the Lagrange multiplier $\lambda>0$. In this example, $r_1=-0.7228$ and     $r_2 = 1.9527$.}
      \label{fig:lamsol}
   \end{figure}
%%%%%%%%%%%%%%%%%%%%%%%%%%%%%%%%%%%%%%%%%%%%%%%%%%%%%%%%%%%%%%%%%%%%%%%%%%%%%%%%%%%%%%%%%%%%%%%%%%%%

Although the ``quartic-to-quadratic''  reduction (\ref{PsiU}) remains valid for many-mode settings,
the proof of Theorem~\ref{th:Psimin} relies on the surjectivity of the map $\sigma \mapsto U$ in (\ref{Usig}) in the sense that any positive definite matrix $U \in \mS_2$ with $\det U = 1$ is factorizable as $U = \sigma^{\rT}\sigma$ for some symplectic matrix $\sigma \in \Sp(\bJ)$. This property does not extend to higher dimensions $\nu$, for which the image of the corresponding symplectic group  is a proper subset of $\fU:= \{U \in \mS_{\nu}:\ U\succ 0,\ \det U = 1\}$. In this case, the minimization of the quadratic function $\wt{\Psi}(U)$ in (\ref{PsiUmin})  over the larger set $\fU$ can only provide a lower (and not necessarily achievable) bound for the actual minimum value of the quartic polynomial $\Psi(\sigma)$.

%%%%%%%%%%%%%%%%%%%%%%%%%%%%%%%%%%%%%%%%%%%%%%%%%%%%%%%%%%%%%%%%%%%%%%%%%%%%%%%%%%%%%%%%%%%%%%%%%%%%%%%%
\section{An application to balancing cascaded one-mode oscillators}\label{sec:example}
%%%%%%%%%%%%%%%%%%%%%%%%%%%%%%%%%%%%%%%%%%%%%%%%%%%%%%%%%%%%%%%%%%%%%%%%%%%%%%%%%%%%%%%%%%%%%%%%%%%%%%%%

Consider a cascade of one-mode OQHOs with dimensions
\begin{equation}
\label{n2}
    n_1=\ldots =n_N = 2.
\end{equation}
In regard to such systems, it can be assumed, without loss of generality, that the CCR matrices in (\ref{XXk}) are given by
\begin{equation}
\label{ThetaJ}
    \Theta_1 = \ldots = \Theta_N = \frac{1}{2}\bJ.
\end{equation}
In accordance with (\ref{posmom}) and (\ref{bJ}), this corresponds to the pairs of position and momentum operators as system variables of the component oscillators. The symplectic similarity transformations of the oscillators  are specified by elements $S_1, \ldots, S_N$ of the common symplectic group $\Sp(\bJ)$ (regardless of the factor $\frac{1}{2}$ in (\ref{ThetaJ})) which can be represented as
\begin{equation}
\label{symp22}
    S_k
    =
    \sR(\phi_k)
    \begin{bmatrix}
        \sqrt{\varsigma_k} & 0\\
        0& \frac{1}{\sqrt{\varsigma_k}}
    \end{bmatrix}
    \sR(\psi_k)
\end{equation}
in terms of positive scalars $\varsigma_k$ and angles $\phi_k$, $\psi_k$ (see, for example, \cite{WPGGCRSL_2012}), where use is made of the rotation matrix (\ref{rot}).

In the framework of the mean square sensitivity minimization for the purity functional, described in Section~\ref{sec:opt},  the matrix $S_k$ can be found by independently solving the  optimization problem (\ref{Psikupmin}), associated with the $k$th oscillator, for every $k=1, \ldots, N$.  In the one-mode case (\ref{n2}), this amounts to applying Theorem~\ref{th:Psimin} with the parameters
\begin{equation}
\label{rhomuab}
    \rho:= 2\sqrt{a_k} \rho_k,
    \qquad
    \mu := \sqrt{b_k} \mu_k,
\end{equation}
where $\rho_k$ and $\mu_k$ are the Frechet derivatives in (\ref{rhokmuk}) evaluated at the original realization of the component oscillator, and $a_k$ and $b_k$ are the upper bounds for the covariances of the uncertainties in the energy and coupling matrices in the sense of (\ref{Sigmakup}) (or (\ref{covup}) if they are uncorrelated). A comparison of (\ref{symp22}) with (\ref{sigopt})  leads to
\begin{equation}
\label{SSU}
    S_k^{\rT} S_k
    =
    \sR(\psi_k)^{\rT}
        \begin{bmatrix}
        \varsigma_k & 0\\
        0& \frac{1}{\varsigma_k}
    \end{bmatrix}
    \sR(\psi_k)
    =
    U.
\end{equation}
Therefore, whilst the angle $\phi_k$ in (\ref{symp22}) can be arbitrary,
$\varsigma_k$ and  $\frac{1}{\varsigma_k}$ are the eigenvalues of the corresponding matrix $U$, specified in Theorem~\ref{th:Psimin}, and the other angle $\psi_k$ is such that the columns of the matrix $\sR(\psi_k)^{\rT} = \sR(-\psi_k)$ are the eigenvectors of $U$.

We will now provide the results of a numerical experiment on balancing a cascade of $N=3$ one-mode oscillators with $m=6$ input field channels and   the energy and coupling matrices
\begin{align*}
    R_1
    & =
    \begin{bmatrix}
    0.8612 &   0.1874\\
    0.1874 &   0.1597
    \end{bmatrix},
    \qquad\quad
    R_2
    =
    \begin{bmatrix}
   -0.1423 &  -0.1922\\
   -0.1922 &   0.5288
   \end{bmatrix},
   \quad\ \
   R_3
   =
   \begin{bmatrix}
   -0.7664 &   0.8088\\
    0.8088 &  -0.9606
    \end{bmatrix},\\
    M_1
    & =
    \begin{bmatrix}
    1.2255 &  -0.3905\\
    0.4671 &  -0.4381\\
   -0.9366 &   1.4755\\
    0.5177 &   2.9142\\
    0.0847 &  -0.8018\\
    1.2692 &   1.0480
    \end{bmatrix},
    \quad
    M_2
    =
    \begin{bmatrix}
    0.0272 &  -2.5614\\
   -0.5771 &  -1.1778\\
    0.3814 &  -0.6270\\
    1.1624 &  -1.2453\\
   -0.7125 &   0.3623\\
    0.6019 &   0.9182
    \end{bmatrix},
    \quad
    M_3
    =
    \begin{bmatrix}
    1.0184 &  -1.5465\\
   -0.4887 &   3.5010\\
    1.6960 &  -0.0958\\
    0.5983 &  -0.2899\\
   -0.3359 &  -2.4758\\
   -0.4739 &  -0.7769
   \end{bmatrix}
\end{align*}
which were randomly generated subject to the condition that
the matrices $A_1$, $A_2$, $A_3$ in (\ref{Ak}) are Hurwitz.  The Frechet derivatives
\begin{align*}
    \rho_1
    & =
    \begin{bmatrix}
   2.5889 &  0.6171\\
   0.6171 &   -2.4492
    \end{bmatrix},
    \qquad\quad\
    \rho_2
    =
    \begin{bmatrix}
    -1.8661 &  0.7260\\
   0.7260 &  0.1425
   \end{bmatrix},
   \qquad
   \rho_3
   =
   \begin{bmatrix}
    -4.5517 &   -1.5005\\
    -1.5005 &   -0.2675
    \end{bmatrix},\\
    \mu_1
    & =
    \begin{bmatrix}
   21.3088 &  -3.1397\\
   -6.3340 &  -1.2695\\
    8.5925 & -12.1129\\
    3.3551 &   7.8397\\
    2.3532 &   3.2141\\
  -13.2210 &  -8.7381
    \end{bmatrix},
    \quad
    \mu_2
    =
    \begin{bmatrix}
   -2.8669 &  -0.9059\\
    4.4807 &   1.5072\\
    5.6762 &  -0.8180\\
    2.3510 &   0.4416\\
    5.6586 &  -0.5636\\
    4.1997 &   0.0501
    \end{bmatrix},
    \quad
    \mu_3
    =
    \begin{bmatrix}
   -0.7576 &  -1.9211\\
  -11.4170 &   0.8482\\
   -3.2624 &   2.4921\\
    7.2670 &   4.9159\\
  -14.6064 &  -1.3306\\
    0.0638 &   7.1250
   \end{bmatrix}
\end{align*}
in (\ref{rhokmuk}) are computed as described in Section~\ref{sec:recder}. The upper bounds $a_k$ and $b_k$ (also randomly generated) on the covariances of the parametric uncertainties in (\ref{Sigmakup})  and the contributions $\Psi_k(S_k)$ from the component oscillators in (\ref{Psikup}) to the mean square sensitivity index before and after the balancing (\ref{SRS}) (and their ratios) are given in Tab.~\ref{Tab:numer}.
%%%%%%%%%%%%%%%%%%%%%%%%%%%%%%%%%%%%%%%%%%%%%%%%%%%%%%%%%%%%%%%%%%%%%%%%%%%%%%%%%%%%%%%%%%%
\begin {table}[htbp]
\centering
\caption {Covariance bounds (\ref{Sigmakup}) for parametric uncertainties and contributions (\ref{Psikup}) to the mean square sensitivity index from three oscillators.}\label{Tab:numer}
    \begin{tabular}{|c||c|c|c|}
    \hline
        $k$ & 1 & 2 & 3 \\
      \hline
      \hline
      $a_k$ & 0.0222  &  0.0283 &   0.0067\\
      \hline
      $b_k$ & 0.0351  &  0.2898 &   0.0388\\
      \hline
      $\Psi_k(I_2)$ & 37.9918  & 35.0268 &  19.5730\\
      \hline
      $\Psi_k(S_k)$   & 34.6230  &  12.8844 &   14.4265\\
      \hline
      $\Psi_k(S_k)/\Psi_k(I_2)$   & 0.9113  &   0.3678  &  0.7371\\
      \hline
    \end{tabular}
\end{table}
%%%%%%%%%%%%%%%%%%%%%%%%%%%%%%%%%%%%%%%%%%%%%%%%%%%%%%%%%%%%%%%%%%%%%%%%%%%%%%%%%%%%%%%%%%%

The symplectic transformation matrices, which minimize the functions $\Psi_k$ for the oscillators, are found by using Theorem~\ref{th:Psimin} in combination with (\ref{rhomuab}), (\ref{SSU}) and the Newton-Raphson algorithm specified in Section~\ref{sec:opt}:
$$
    S_1
    =
    \begin{bmatrix}
    0.8085  &  0.0167\\
    0.0167  &  1.2372
    \end{bmatrix},
    \qquad
    S_2
    =
    \begin{bmatrix}
    0.4382 &  -0.0469\\
   -0.0469 &   2.2873
   \end{bmatrix},
   \qquad
    S_3
    =
    \begin{bmatrix}
    0.6788 &  -0.1027\\
   -0.1027 &   1.4886
   \end{bmatrix}.
$$
The reduction in the total value of the mean square sensitivity index, which is achieved  for these matrices, is
$$
    \frac
    {\sum_{k=1}^3 \Psi_k(S_k)}
    {\sum_{k=1}^3 \Psi_k(I_2)}
    =
    0.6689,
$$
which is an intermediate value in comparison with the reduction ratios for the component oscillators at the bottom of Tab.~\ref{Tab:numer}.

%%%%%%%%%%%%%%%%%%%%%%%%%%%%%%%%%%%%%%%%%%%%%%%%%%%%%%%%%%%%%%%%%%%%%%%%%%%%%%%%%%%%%%%%%%%%%%%%%%%%%%%%
\section{Conclusion}\label{sec:conc}
%%%%%%%%%%%%%%%%%%%%%%%%%%%%%%%%%%%%%%%%%%%%%%%%%%%%%%%%%%%%%%%%%%%%%%%%%%%%%%%%%%%%%%%%%%%%%%%%%%%%%%%%

We have considered an approach to balancing cascaded linear quantum stochastic systems via symplectic similarity transformations of their realizations for the generation of Gaussian invariant quantum states. The optimality criterion considered is based on minimizing the mean square sensitivity of the purity functional quantified by the norm of its logarithmic Frechet derivative with respect to the energy and coupling matrices subject to small random perturbations. We have discussed a recursive computation of this performance functional together with a link to the classical Fisher information metric and reduced its optimization to a set of independent problems of minimizing quartic polynomials on symplectic groups. We have developed an effective algorithm for solving these problems in the one-mode setting when they reduce to quadratic optimization problems. We have also outlined the infinitesimal perturbation analysis for translation invariant cascades of identical oscillators using spatial $z$-transforms.  The results of this paper may be of use in the robust generation of pure Gaussian states for the purposes of state preparation in quantum computing and quantum information processing involving quantum stochastic networks, especially those with fractal-like architectures (in the form of acyclic directed graphs such as trees). Beyond the quantum domain, similar ideas are applicable to perturbation analysis and optimization of stochastic versions of port-Hamiltonian systems \cite{VJ_2014} whose dynamics bear special structures reflecting the energetics of classical physical systems.

\appendix

%%%%%%%%%%%%%%%%%%%%%%%%%%%%%%%%%%%%%%%%%%%%%%%%%%%%%%%%%%%%%%%%%%%%%%%%%%%%%%%%%%%%%%%%%%%%%%%%%%%
\section{Classical relative entropy deviation for Gaussian states}
\label{sec:relent}
%%%%%%%%%%%%%%%%%%%%%%%%%%%%%%%%%%%%%%%%%%%%%%%%%%%%%%%%%%%%%%%%%%%%%%%%%%%%%%%%%%%%%%%%%%%%%%%%%%%
We will now discuss a connection of the mean square sensitivity index (\ref{Z}) for the purity functional with an alternative criterion 
which quantifies the perturbations in the matrix $\cP$ in terms of classical distance measures applied to invariant Gaussian quantum states.\footnote{As mentioned in Section~\ref{sec:inv}, inaccuracies in the energy and coupling matrices do not affect the CCRs (\ref{cXX}), thus allowing the Gaussian invariant states to be distinguished by the real parts of their quantum covariance matrices.}
In particular, an important role is played by entropy theoretic statistical distances. Recall that the relative entropy (or the Kullback-Leibler informational divergence \cite{CT_2006}) of a probability measure $\sP$ with respect to a reference probability measure $\sP_*$ (on the same measurable space $(\fX,\fB)$) is defined as the expectation
\begin{align}
\nonumber
    \bD(\sP\|\sP_*)
    & :=
    \bE_{\sP}
    \ln \frac{\rd \sP}{\rd \sP_*}\\
\nonumber
    & =
    \int_{\fX} \ln \frac{\rd \sP}{\rd \sP_*} \rd \sP\\
\label{bD}
    & =
    \lim_{\alpha\to 1}
    \Big(
    \frac{1}{\alpha-1}
    \ln
        \int_{\fX}
        \Big(
            \frac{\rd \sP}{\rd \sP_*}
        \Big)^{\alpha}
        \rd \sP_*
    \Big).
\end{align}
Here, $\sP$ is assumed to be absolutely continuous with respect to $\sP_*$,
and $\frac{\rd \sP}{\rd \sP_*}$ denotes the corresponding Radon-Nikodym derivative \cite{SG_1977,S_1996}. The last equality in (\ref{bD}) links the functional $\bD$ with the relative Renyi  $\alpha$-entropy \cite{R_1961} which,  in the case $\alpha=\frac{1}{2}$, is present in the Hellinger distance between $\sP$ and $\sP_*$ given by
$$
    \sqrt{1-\int_{\fX}
            \sqrt{\frac{\rd \sP}{\rd \sP_*}}
        \rd \sP_*}.
$$
While the choice of a statistical distance is not critical for our purposes, the Kullback-Leibler relative entropy (\ref{bD}) is particularly convenient  in application  to Gaussian distributions and has direct links with the Shannon information theory \cite{G_2008}.

The following lemma, which is concerned with the Gaussian case, is well-known in the context of the maximum entropy principle \cite{CT_2006,ME_1981} (see also \cite[Lemma~4 on pp.~313--314]{P_2006}, \cite[Lemma~9.1 on p.~1129]{VP_2015} and references therein).

%====================================================================
\begin{lemma}
\label{lem:Gauss}
For two zero-mean Gaussian distributions $\cN(0,P)$ and $\cN(0,P_*)$ in $\mR^n$ with nonsingular covariance matrices $P$ and $P_*$, the relative entropy (\ref{bD}) of $\cN(0,P)$ with respect to $\cN(0,P_*)$ takes the form
\begin{align}
\nonumber
    \bitDelta(P\|P_*)
    & :=
    \bD(\cN(0,P)\|\cN(0,P_*))\\
\label{bDGG}
    &=
    \frac{1}{2}
    (
         \Tr \chi
         -
         \ln\det\chi
         -
         n
    ),
\end{align}
where
\begin{equation}
\label{chi}
    \chi
    :=
    P_*^{-1/2}PP_*^{-1/2}.
\end{equation}
The asymptotic behaviour of this quantity, as $P\to P_*$,  is described by
\begin{equation}
\label{Pclose}
    \bitDelta(P,P_*)
    =
    \frac{1}{4}
    \|\chi-I_n\|^2(1+o(1)).
\end{equation}
\hfill$\square$
\end{lemma}
%%====================================================================

The relation (\ref{Pclose})\footnote{which follows from (\ref{bDGG}) and (\ref{chi}) in view of the Frechet derivatives (\ref{dlndet}) and $\d_{\chi}^2\ln\det \chi(M) = -\chi^{-1} M \chi^{-1}$ on the set of real positive definite symmetric matrices $\chi$.}  shows that, in a small neighbourhood of the reference matrix  $P_*$,  the quantity $\sqrt{\bitDelta(P,P_*)}$ gives rise to the relative deviation of the actual covariance matrix $P$ from $P_*$:
\begin{equation}
\label{gPP}
    \|P_*^{-1/2} (P-P_*)P_*^{-1/2}\|
    =
    \sqrt{\Bra P-P_*, g(P-P_*)\Ket}.
\end{equation}
Here, $g$ is a positive definite self-adjoint operator which acts on the Hilbert space $\mS_n$ as
\begin{equation}
\label{gM}
    g(M):= P_*^{-1} MP_*^{-1}
\end{equation}
and  specifies the local metric tensor associated with the Fisher information distance  \cite{S_1984}  (on the set of positive definite covariance matrices of order $n$ regarded as a Riemannian manifold).

In application to the robustness of Gaussian state generation with respect to the implementation errors (modelled in Section~\ref{sec:opt} as zero-mean  classical  random vectors with the covariance matrices  (\ref{EEcov})), the relations  (\ref{gPP}) and (\ref{gM}) suggest that the balancing of the oscillators can also be based on minimizing the functional
\begin{equation}
\label{Znew}
    Z
    :=
    \sum_{k=1}^N
    \Bra
        \mho_k^{\rT}(\d_{\vec{E}_k}\vec{\cP})^{\rT}(\cP \ox \cP)^{-1}\d_{\vec{E}_k}\vec{\cP}\mho_k,
        \Sigma_k
    \Ket,
\end{equation}
where $\mho_k$ is the matrix from (\ref{Evech}).
Indeed, averaging of the leading quadratic term of the Kullback-Leibler deviation of the perturbed covariance matrix from its nominal value leads to
\begin{align}
\nonumber
    \bE (\| & \cP^{-1/2}  (\delta \cP) \cP^{-1/2}\|^2)
    =
    \bE
    \Bra
        \delta \cP,
        \cP^{-1} (\delta \cP ) \cP^{-1}
    \Ket\\
\nonumber
    & =
    \bE
    \big(
        \delta \vec{\cP}^{\rT}
        (\cP \ox \cP)^{-1}\delta\vec{\cP}
    \big)\\
\nonumber
    & =
    \bE
    \sum_{j,k=1}^N
    \delta \vec{E}_j^{\rT}
    (\d_{\vec{E}_j}\vec{\cP})^{\rT}
        (\cP \ox \cP)^{-1}
    \d_{\vec{E}_k}\vec{\cP} \delta \vec{E}_k\\
\nonumber
    & =
    \sum_{j,k=1}^N
    \Bra
    (\d_{\vec{E}_j}\vec{\cP})^{\rT}
        (\cP \ox \cP)^{-1}
    \d_{\vec{E}_k}\vec{\cP},
    \cov(\delta \vec{E}_j, \delta \vec{E}_k)
    \Ket\\
\nonumber
    & =
    \eps
    \sum_{k=1}^N
    \Bra
    (\d_{\vec{E}_k}\vec{\cP})^{\rT}
        (\cP \ox \cP)^{-1}
    \d_{\vec{E}_k}\vec{\cP},
    \mho_k \Sigma_k \mho_k^{\rT}
    \Ket\\
\label{epsZnew}
    & =
    \eps Z,
\end{align}
which is proportional to the quantity $Z$ in  (\ref{Znew}) with a constant multiplier $\eps$. In (\ref{epsZnew}), use has also been made of the symmetry of the matrix $\cP$ together with the identity $\cP^{-1} \ox \cP^{-1} = (\cP \ox \cP)^{-1}$. Note  that the entropy theoretic criterion (\ref{Znew}) yields an upper bound on the purity functional sensitivity index (\ref{Z}) in the sense that
\begin{align}
\nonumber
    \bra
        \cP^{-1},
        \delta \cP
    \ket^2
    & =
    \bra
        I_n,
        \cP^{-1/2}
        (\delta \cP)
        \cP^{-1/2}
    \ket^2  \\
\label{ZZbound}
    & \<
    n
    \|
        \cP^{-1/2}
        (\delta \cP)
        \cP^{-1/2}
    \|^2
\end{align}
in view of the Cauchy-Bunyakovsky-Schwarz inequality (applied here to the Frobenius inner product), though the coefficient of proportionality $n$ in (\ref{ZZbound}) is large for many-mode systems. The above described connection is a consequence of the linear relation between the differential entropy \cite{CT_2006} of a classical Gaussian distribution and the log determinant of its covariance matrix. In fact, the quantities $V_k$ and $V_{\>k}$ in (\ref{VV}) and (\ref{Vtail}) (which are parts of the purity functional) are linearly related to the conditional differential entropies \cite{CT_2006,G_2008} of the auxiliary classical Gaussian random vectors  $\xi_k$ and $\Xi_{\>k}$ in (\ref{Xi>=k}) with respect to the vector $\Xi_{k-1}$ in (\ref{Xik}).

%%%%%%%%%%%%%%%%%%%%%%%%%%%%%%%%%%%%%%%%%%%%%%%%%%%%%%%%%%%%%%%%%%%%%%%%%%%%%%%%%%%%%%%%%%%%%%%%%%%%%%%%
\section{Translation invariant cascades}\label{sec:trans}
%%%%%%%%%%%%%%%%%%%%%%%%%%%%%%%%%%%%%%%%%%%%%%%%%%%%%%%%%%%%%%%%%%%%%%%%%%%%%%%%%%%%%%%%%%%%%%%%%%%%%%%%
We will now briefly discuss a translation invariant case when all the component oscillators, described in Section~\ref{sec:sys}, are identical and have a common dimension $n$, CCR matrices $\Theta \in \mA_n$, energy matrices $R\in \mS_n$, coupling matrices $M \in \mR^{m\x n}$, and the state-space matrices
 \begin{align}
\label{A}
    A
    & := 2\Theta (R + M^{\rT}JM),\\
\label{B}
    B
    &:= 2\Theta M^{\rT},\\
\nonumber
    C
    & := 2J M,
\end{align}
corresponding to (\ref{Ak})--(\ref{Ck}), and the transfer functions in (\ref{Fk}) and (\ref{Gk}) given by
\begin{align}
\label{F}
    F(s)
    & := (sI_n - A)^{-1}B,\\
\label{G}
    G(s)
    & := C F(s) + I_m.
\end{align}
In this case, the initial spaces of the oscillators are copies of a common Hilbert space $\fH$,  and the QSDEs (\ref{dXk}) and (\ref{dYk}) take the form
 \begin{align}
\label{dXktrans}
    \rd X_k
    & =
    AX_k \rd t + B \rd Y_{k-1},\\
\label{dYktrans}
    \rd Y_k
    & = CX_k\rd t + \rd Y_{k-1},
\end{align}
with $Y_0:= W$ as before. Following \cite{VP_2014} (see also \cite{SVP_2015a,SVP_2015b}) and regarding the composite system as an infinite cascade of oscillators, we will use the spatial $z$-transforms of the system and output variables:
 \begin{align}
\label{Xz}
    \cX_z(t)
    & :=
    \sum_{k=1}^{+\infty}
    z^{-k}X_k(t),\\
\label{Yz}
    \cY_z(t)
    & :=
    \sum_{k=0}^{+\infty}
    z^{-k}Y_k(t).
\end{align}
These series are mean square convergent for sufficiently large values of the complex parameter $z$. More precisely, since the output variables of the oscillators are linearly related to the system variables and the input fields, the convergence of both  series is guaranteed if
\begin{equation}
\label{zbig}
  |z|>  \limsup_{k\to +\infty}\sqrt[2k]{\bE (X_k(t)^{\rT}X_k(t))},
\end{equation}
provided the upper limit is finite,
which is obtained by using an appropriate modification of  the Cauchy-Hadamard theorem.
The vectors $\cX_z$ and $\cY_z$ consist of time-varying (and not necessarily self-adjoint) operators on the system-field space $\fH^{\ox\infty}\ox \fF$ and satisfy the CCRs
\begin{align}
\nonumber
    [\cX_z, \cX_v^{\rT}]
    & =
    \sum_{j,k=1}^{+\infty}
    z^{-j}v^{-k}
    [X_j,X_k^{\rT}]\\
\nonumber
    & =
    2i
    \sum_{k=1}^{+\infty}
    (zv)^{-k}
    \Theta\\
\label{XzXvcomm}
    & =
    \frac{2i}{zv-1}\Theta%,\\
\end{align}
which follow from the bilinearity of the commutator and the CCRs (\ref{XXk}), provided $|z||v|>1$.

%%%%%%%%%%%%%%%%%%%%%%%%%%%%%%%%%%%%%%%%%%%%%%%%%%%%%%%%%%%%%%%%%%%%%%%%%%%%%%%%%%%%%%%%%%%%%%%%%%%%%%%%
\begin{lemma}
\label{lem:trans}
For all $z\in \mC$ large enough (for example, such as in (\ref{zbig})), the processes $\cX_z$ and $\cY_z$ in (\ref{Xz}) and (\ref{Yz}) satisfy the QSDEs
\begin{align}
\label{dXz}
    \rd \cX_z
    & =
    \cA_z\cX_z \rd t + \cB_z\rd W,\\
\label{dYz}
    \rd \cY_z
    & = \cC_z\cX_z\rd t + \cD_z\rd W.
\end{align}
Here, the $z$-dependent state-space matrices are computed as
\begin{align}
\label{ABz}
    \cA_z
    & :=
    A + \frac{1}{z-1}BC,
    \qquad
    \cB_z := \frac{1}{z-1} B,\\
\label{CDz}
    \cC_z
    & := \frac{z}{z-1}C,
    \qquad\qquad\ \
    \cD_z := \frac{z}{z-1} I_m.
\end{align}
\hfill$\square$
\end{lemma}
%%%%%%%%%%%%%%%%%%%%%%%%%%%%%%%%%%%%%%%%%%%%%%%%%%%%%%%%%%%%%%%%%%%%%%%%%%%%%%%%%%%%%%%%%%%%%%%%%%%%%%%%
\begin{proof}
Application of the $z$-transforms (\ref{Xz}) and (\ref{Yz}) to the sequence of QSDEs (\ref{dXktrans}) and (\ref{dYktrans}) with $k=1, 2, \ldots$ leads to
\begin{align}
\label{dXz1}
    \rd \cX_z
    & =
    A\cX_z \rd t + \frac{1}{z}B \rd \cY_z,\\
\label{dYz1}
    \rd (\cY_z-W)
    & = C\cX_z\rd t + \frac{1}{z}\rd \cY_z.
\end{align}
Here, use is also made of the convention $Y_0=W$ and the identities
$$
    \sum_{k=1}^{+\infty}
    z^{-k} Y_{k-1}
     =
    \frac{1}{z}
    \cY_z,
    \qquad
    \sum_{k=1}^{+\infty}
    z^{-k} Y_k
    =
    \cY_z - W.
$$
The QSDE (\ref{dYz}) with the matrices (\ref{CDz}) is obtained by solving (\ref{dYz1}) for $\rd \cY_z$ as
$$
    \rd \cY_z
    =
    \frac{z}{z-1}
    (C\cX_z\rd t + \rd W)
$$
and substituting the result into (\ref{dXz1}), which leads to  the QSDE (\ref{dXz}) with the matrices (\ref{ABz}).
\end{proof}
%%%%%%%%%%%%%%%%%%%%%%%%%%%%%%%%%%%%%%%%%%%%%%%%%%%%%%%%%%%%%%%%%%%%%%%%%%%%%%%%%%%%%%%%%%%%%%%%%%%%%%%%

In application to the matrices $A$ and $B$ of an individual oscillator   in (\ref{A}) and (\ref{B}),   the PR property (\ref{APRk}) takes the form
$$
    A \Theta + \Theta A^{\rT} + BJB^{\rT}  = 0
$$
and leads to a similar PR condition for the  matrices $\cA_z$ and $\cB_z$ in (\ref{ABz}):
\begin{equation}
\label{XzXvPR}
    \frac{1}{zv-1}(\cA_z \Theta + \Theta \cA_v^{\rT}) + \cB_z J \cB_v^{\rT} = 0,
\end{equation}
which is equivalent to the preservation of the CCRs (\ref{XzXvcomm}) under the QSDE (\ref{dXz}).

Now,
if the matrix $A$ in (\ref{A}) is Hurwitz, then so is the matrix $\cA_z$ in (\ref{ABz}) for all sufficiently  large $z$, and the set
\begin{equation}
\label{fZ}
    \fZ:=\{z\in \mC:\ \cA_z\ {\rm is\ Hurwitz}\}
\end{equation}
is nonempty and open. In this case, not only every finite subsystem of the infinite cascade has a unique Gaussian invariant state, but a similar property also holds for the process $\cX_z$ which involves all the system variables of the cascade.

In view of (\ref{FG}), (\ref{F}) and (\ref{G}), the transfer function from the input quantum Wiener process $W$ to the vector $X_k$ of system variables of the $k$th oscillator in the translation invariant cascade is $FG^{k-1}$ for all $k=1,2,\ldots$. Hence (see also (\ref{Xktilde}) and (\ref{Yktilde})),  the transfer function from $W$ to $\cX_z$ can be computed in the right-half plane $\Re s>0$  in terms of the transfer functions $F$ and $G$ and directly from the QSDE (\ref{dXz}) as
\begin{align}
\nonumber
    \Phi_z(s)
    & :=
    \sum_{k=1}^{+\infty}
    z^{-k} F(s)G(s)^{k-1}\\
\nonumber
    & =
    \frac{1}{z}
    F(s)
    \sum_{k=1}^{+\infty}
    \Big(\frac{1}{z}G(s)\Big)^{k-1}\\
\nonumber
    & =
    F(s)(zI_m - G(s))^{-1}\\
\label{Phi}
    & =
    (sI_n - \cA_z)^{-1} \cB_z
\end{align}
for any $z\in \mC$ satisfying
\begin{equation}
\label{zGbig}
      |z|> \|G\|_{\infty}.
\end{equation}
Therefore, if $\cA_z$ is Hurwitz (that is, $z\in \fZ$ in view of (\ref{fZ})), the last equality in (\ref{Phi}) implies that $\Phi_z$ belongs to the Hardy space $\cH_2$. On the other hand, under the assumption that the matrix $A$ is Hurwitz, the second to last equality  in (\ref{Phi}) leads to an upper bound for the $\cH_2$-norm of $\Phi_z$:
\begin{align}
\nonumber
    \|\Phi_z\|_2
    & :=
    \sqrt{\frac{1}{2\pi}\int_{-\infty}^{+\infty}\|\Phi_z(i\lambda)\|^2 \rd \lambda }\\
\nonumber
    & \< \|F\|_2\|(zI_m - G)^{-1}\|_{\infty}\\
\label{Phiup}
    & \< \frac{\|F\|_2}{|z|- \|G\|_{\infty}}
\end{align}
in view of (\ref{zGbig}).
In (\ref{Phiup}), we have also used the inequality\footnote{well known in the context of Banach algebras and their applications to the small gain theorem}  $\|(I_m - H)^{-1}\|_{\infty} \< \frac{1}{1-\|H\|_{\infty}}$ which holds under the condition $\|H\|_{\infty}<1$ and follows from the submultiplicativity of the $\cH_{\infty}$-norm.  
In particular, if, for given $z$ and $v$ from the set (\ref{fZ}), the corresponding processes $\cX_z$ and $\cX_v$ have finite initial second moments (for example, both $z$ and $v$ satisfy (\ref{zbig}) at time $t=0$), then they also have a steady-state quantum covariance matrix
\begin{align}
\nonumber
    \cP_{zv}
    & =
    \lim_{t\to +\infty} \bE(\cX_z(t)\cX_v(t)^{\rT})\\
\label{cPzv}
    & =
    \sum_{j,k=1}^{+\infty}
    z^{-j}v^{-k}P_{jk},
\end{align}
which is the unique solution of the ASE
\begin{equation}
\label{PzvASE}
    \cA_z \cP_{zv} + \cP_{zv} \cA_v^{\rT} + \cB_z \Omega \cB_v^{\rT} = 0.
\end{equation}
Note that (\ref{XzXvPR}) is part of this equation.
In accordance with (\ref{cPN}), the matrix $P_{jk}$ in (\ref{cPzv}) is the real part of the invariant quantum covariance matrix of the vectors $X_j$ and $X_k$ of system variables for the $j$th and $k$th oscillators.
In accordance with (\ref{bLvec}),  the vectorization of the ASE (\ref{PzvASE}) allows its solution  to be represented as
\begin{equation}
\label{Pzvvec}
    \vec{\cP}_{zv}
    =
    -(\cA_v \op \cA_z)^{-1} \vect(\cB_z \Omega \cB_v^{\rT}).
\end{equation}
In view of the rational dependence of the matrices $\cA_z$ and $\cB_z$ on $z$ in (\ref{ABz}), it follows from (\ref{Pzvvec}) that $\cP_{zv}$ is also a rational function of $z$ and $v$:
\begin{align}
\nonumber
    \vec{\cP}_{zv}
    = &
    -\frac{1}{(z-1)(v-1)}\\
\nonumber
    & \x
    \Big(
        I_{n^2}
        +
        \frac{1}{v-1}
        K
        +
        \frac{1}{z-1}
        L
    \Big)^{-1}\\
\label{Pzvvec1}
    & \x (A\op A)^{-1}
    \vect(B\Omega B^{\rT}).
\end{align}
Here, the matrix $A\op A$ is nonsingular (moreover, Hurwitz) due to $A$ being Hurwitz, and use is made of the auxiliary matrices
\begin{align}
\label{K}
    K & := (A\op A)^{-1}
    ((BC)\ox I_n),\\
\label{L}
    L & := (A\op A)^{-1}
    (I_n\ox (BC)).
\end{align}
The representation (\ref{Pzvvec1})--(\ref{L}) can, in principle, be employed as a generating function of the steady-state matrices $P_{jk}$ for the infinite cascade. However, its straightforward use is complicated by the fact that, in general, the matrices $K$ and $L$ do not commute. We will therefore consider an upper bound for the covariance matrices.

%%%%%%%%%%%%%%%%%%%%%%%%%%%%%%%%%%%%%%%%%%%%%%%%%%%%%%%%%%%%%%%%%%%%%%%%%%%%%%%%%%%%%%%%%%%%%%%%%%%%%%%%
\begin{theorem}
\label{th:Pkk}
Suppose the matrix $A$ of a component oscillator in the translation invariant cascade is Hurwitz. Then, for any     $k = 1,2,\ldots$,  the steady-state covariance matrix of the system variables in the $k$th OQHO satisfies
\begin{equation}
\label{Pkkupper}
    \Tr P_{kk} \<     2\|F\|_2^2    \|G\|_{\infty}^{2(k-1)},
\end{equation}
where $\|F\|_2$ and $\|G\|_{\infty}$ are the $\cH_2$ and $\cH_{\infty}$-norms of the transfer functions in (\ref{F}) and (\ref{G}).
\hfill$\square$
\end{theorem}
%%%%%%%%%%%%%%%%%%%%%%%%%%%%%%%%%%%%%%%%%%%%%%%%%%%%%%%%%%%%%%%%%%%%%%%%%%%%%%%%%%%%%%%%%%%%%%%%%%%%%%%%
\begin{proof}
A translation invariant version of the frequency-domain representation (\ref{cPpos1})--(\ref{Gk}) for the steady-state covariances implies that
\begin{equation}
\label{Pjktrans}
    P_{jk} + i\delta_{jk} \Theta
    =
    \frac{1}{2\pi}
    \int_{-\infty}^{+\infty}
    F(i\lambda) G(i\lambda)^{j-1}
    \Omega
    \big(
        F(i\lambda) G(i\lambda)^{k-1}
    \big)^*
    \rd \lambda
\end{equation}
for all $j,k=1, 2, \ldots$. By using the spectral radius $\lambda_{\max}(\Omega)=2$ of the quantum Ito matrix $\Omega$ from (\ref{YYk}) and submultiplicativity  of the $\cH_{\infty}$-norm, it follows from (\ref{Pjktrans}) that
 \begin{align}
 \nonumber
    P_{kk} + i\Theta
    = &
    \frac{1}{2\pi}
    \int_{-\infty}^{+\infty}
    F(i\lambda) G(i\lambda)^{k-1}\Omega (G(i\lambda)^{k-1})^* F(i\lambda)^*
    \rd \lambda\\
 \label{Pkkup}
    \preccurlyeq  &
    \frac{1}{\pi}
    \int_{-\infty}^{+\infty}
    F(i\lambda) F(i\lambda)^*
    \rd \lambda
    \|G\|_{\infty}^{2(k-1)}.
 \end{align}
 The inequality (\ref{Pkkupper}) can now be obtained by taking the trace on both sides of (\ref{Pkkup}) in view of the antisymmetry of the CCR matrix $\Theta$.
\end{proof}
%%%%%%%%%%%%%%%%%%%%%%%%%%%%%%%%%%%%%%%%%%%%%%%%%%%%%%%%%%%%%%%%%%%%%%%%%%%%%%%%%%%%%%%%%%%%%%%%%%%%%%%%

 In accordance with (\ref{Gprodnorm}),  the $\cH_{\infty}$-norm of the transfer function $G$ of an individual oscillator in (\ref{G}) satisfies $\|G\|_{\infty}\> 1$. Therefore, (\ref{Pkkupper}) can only bound the exponential growth of the steady-state covariances for distant oscillators along the cascade. Also note that, since the operator norms of the matrices $P_{jj}$, $P_{jk}$, $P_{kk}$ satisfy
 $$
    \|P_{jk}\|_{\infty} \< \sqrt{\|P_{jj}\|_{\infty}\|P_{kk}\|_{\infty}}
 $$
 (due to the positive semi-definiteness which the matrix $\small{\begin{bmatrix}P_{jj} & P_{jk}\\ P_{kj} & P_{kk}\end{bmatrix}}$ inherits from $\cP_N$ in (\ref{cPN})),  the upper bound (\ref{Pkkupper}) ensures absolute convergence of the series in (\ref{cPzv}) for all $z$ and $v$ satisfying (\ref{zGbig}), that is,   $|z|, |v|> \|G\|_{\infty}$.

\end{document}